\documentclass{amsart}
\usepackage{amsrefs,amssymb}
\usepackage{mathtools}          
\usepackage[inline]{enumitem}   
\usepackage{scalerel}           

\usepackage[pdftex, breaklinks, colorlinks]{hyperref}

\DeclarePairedDelimiter\abs{\lvert}{\rvert}
\DeclarePairedDelimiter\norm{\lVert}{\rVert}
\newcommand\sm[1]{{\scaleobj{0.6}{#1}}}
\newcommand\Sm[1]{{\scaleobj{0.8}{#1}}}
\newcommand\key[1]{\emph{#1}}
\newcommand\reals{\mathbb{R}}
\newcommand\posreals{\mathbb{R^+}}
\newcommand\posrationals{\mathbb{Q^+}}
\newcommand\sphere{\mathbb{S}}
\newcommand\nats{\mathbb{N}}
\newcommand{\renorm}{\mathcal{R}}
\newcommand{\deriv}{D}
\newcommand{\nonlin}{N}
\newcommand{\schwarz}{S}
\newcommand\compose{\bigcirc}
\newcommand\eps{\varepsilon}
\newcommand\diff{\mathrm{Diff}}
\newcommand\sdiff{\mathrm{Diff}^\mathrm{S}}
\newcommand\lorenz{\mathcal{L}}
\newcommand\inv[2][1]{#2^{-#1}}
\newcommand\Ck[1]{C^{#1}}
\newcommand\ubdist{\tfrac12\log\alpha}
\DeclarePairedDelimiterX\zoom[2]{[}{]}{#1|#2}
\newcommand\Dl{\mathcal{D}}
\newcommand\D{\Dl(b,\Delta,\gamma,\delta)}
\newcommand\sref[1]{\S\ref{#1}}
\newcommand\ldiff[1]{\ell^1(\diff^3; #1)}
\newcommand\lpure[1]{\ell^1(\reals; #1)}
\newcommand\Wu{\mathcal{W}_b^{\text{u}}}
\newcommand\Wuu{\mathcal{W}_b^{\text{uu}}}
\newcommand\tc{\mathcal{T}}
\newcommand\tu{\mathcal{U}}
\newcommand\wu{\mathcal{W}^\text{u}}
\newcommand\wuu{\mathcal{W}^\text{uu}}
\newcommand\family{\mathcal{F}}
\newcommand\crit{c}
\newcommand\relc{\tilde c}
\newcommand\relv{\tilde v}
\DeclareMathOperator{\dist}{dist}
\DeclareMathOperator{\HD}{HD}
\DeclareMathOperator{\id}{id}
\DeclareMathOperator{\Int}{Int}
\DeclareMathOperator{\depth}{depth}
\DeclareMathOperator{\vol}{vol}

\DeclareMathOperator{\graph}{graph}

\newtheorem{theorem}{Theorem}
\newtheorem{proposition}[theorem]{Proposition}
\newtheorem{lemma}[theorem]{Lemma}
\newtheorem*{coexistence-thm}{Coexistence Theorem}
\newtheorem*{expansion-lemma}{Expansion Lemma}
\newtheorem*{flipping-lemma}{Flipping Lemma}
\newtheorem*{homotopy-lemma}{Homotopy Lemma}
\newtheorem*{degeneration-thm}{Degeneration Theorem}
\newtheorem*{a-priori-bounds}{A Priori Bounds}
\newtheorem*{unstable-mfd-thm}{Dimensional Discrepancy Theorem}
\newtheorem*{T-conjecture}{Conjecture}

\theoremstyle{remark}
\newtheorem*{remark}{Remark}

\begin{document}

\title{Instability of renormalization}
\author{Marco Martens}
\author{Bj\"orn Winckler}
\address{Institute for Mathematical Sciences, Stony Brook University, Stony
    Brook NY 11794-3660, USA}
\email{marco@math.stonybrook.edu}
\email{bjorn.winckler@gmail.com}
\date{\today}
\thanks{Research supported by NSF grant 1600554 and the Knut\&Alice Wallenberg
    Foundation}

\begin{abstract}
    In the theory of renormalization for classical dynamical systems, e.g.\
    unimodal maps and critical circle maps, topological conjugacy classes are
    stable manifolds of renormalization.
    Physically more realistic systems on the other hand may exhibit instability
    of renormalization within a topological class.
    This instability gives rise to new phenomena and opens up directions of
    inquiry that go beyond the classical theory.
    In phase space it leads to the coexistence phenomenon, i.e.\ there are
    systems whose attractor has bounded geometry but which are topologically
    conjugate to systems whose attractor has degenerate geometry;
    in parameter space it causes dimensional discrepancy, i.e.\ 
    a topologically full family has too few dimensions to realize all possible
    geometric behavior.
\end{abstract}

\maketitle

\section{Introduction}
\label{introduction}

To understand the behavior of dynamical systems it is natural to consider:
\begin{enumerate*}[label=(\alph*)]
    \item the geometry of attractors, \label{aspect2}
    \item bifurcation patterns of families, \label{aspect3}
    \item topological and combinatorial aspects, as well as \label{aspect1}
    \item measure theoretical aspects. \label{aspect4}
\end{enumerate*}
Renormalization is a tool which was originally introduced into dynamics
\cites{CT78,F78} to
analyze \ref{aspect2} and~\ref{aspect3} but it turned out to connect all four
of the above aspects.
The renormalization of a system is a new system which describes the dynamics
on a smaller scale; it is a microscope on attractors.
Intrinsic to this scheme is a characterization of~\ref{aspect1} which provides
a natural setting for understanding \ref{aspect4} and culminating in a
description of topological conjugacy classes as invariant manifolds of
renormalization.

In classical systems, e.g.\ unimodal maps and critical circle maps,\footnote{%
    With some restriction on the topology; for simplicity, assume stationary
    combinatorics.
    }
instability of renormalization is exclusively
associated with changes in topology: a topological conjugacy class is the stable
manifold of a hyperbolic fixed point of renormalization and its unstable
manifold is a topologically full family (see e.g.\ \cites{AL11,FMP06,Y02} and
references therein).
As a consequence there is rigidity and parameter universality.
Rigidity is the phenomenon that two topologically conjugate systems are
automatically smoothly conjugate on their attractors, i.e.\ topology determines
geometry.
Intuitively, repeated renormalization is like increasing the magnification
factor whilst looking at an attractor under a microscope;
successive renormalizations converge to the fixed point, so asymptotically the
attractor looks like the attractor of the fixed point.
Parameter universality is the phenomenon that the metric aspects of the
bifurcation patterns of a family is determined by the bifurcation patterns of
the unstable manifold.
It is a consequence of the fact that a topologically full family has the same
dimension as the unstable manifold; hence it meets the stable manifold and
successive renormalizations of the family accumulate on the
unstable manifold.

A question of great interest is what happens to these phenomena as more
physically relevant systems are considered?
Coullet and Tresser \cite{CT78} conjectured that universality would occur in
real world systems and this has since been confirmed in many contexts (e.g.\
\cites{LM80,L81}).
However, the main message here is that in more realistic systems there may be
instability of renormalization inside a topological class, leading to new
phenomena which are much more intricate than for the classical systems.
The theory of renormalization is at the beginning of a new chapter which goes
beyond the classical theory.
The first indication of this was observed in H\'enon dynamics where the
rigidity paradigm turned out to only hold in a probabilistic sense: the
conjugacy between two attractors of period-doubling type is smooth except on
a set of measure zero where it is at most H\"older \cite{CLM05}.

In this paper we show that within the context of Lorenz dynamics there is
instability of renormalization inside most topological classes, even for
stationary combinatorics.
Lorenz maps are one-dimensional systems that are closely related to unimodal
maps, but they are physically more relevant in the sense that they describe the
dynamics of certain higher-dimensional flows (see e.g.\ \cites{ACT81,GW79,V00}
and references therein).
Instability of renormalization has two distinct consequences for these systems:
\begin{enumerate*}[label=(\roman*)]
    \item degeneration of successive renormalizations within the topological
        class of a renormalization fixed point,
        \label{instability1}
    \item unstable manifolds of renormalization fixed points have strictly
        larger dimension ($\geq3$) than topologically full families
        ($2$).
        \label{instability2}
\end{enumerate*}

Item \ref{instability1} leads to what we call the \key{coexistence phenomenon},
which is when a topological class contains both systems whose attractor has
bounded geometry and systems whose attractor has degenerate geometry.
In particular, there is no rigidity in the traditional sense that the
topological class is a rigidity class; instead the topological class is
partitioned by rigidity classes.
This has consequences on the bifurcation patterns of a family:
in the classical setting they are given by intersections with topological
classes, but here they are given by intersections with rigidity classes.
Because of \ref{instability2}, which we call \key{dimensional discrepancy},
a topologically full family (of dim $2$) is too small to realize all geometric
aspects of the possible bifurcation patterns.
For example, such a family will not meet the stable manifold generically and
hence the rigidity class of the renormalization fixed point will not be
represented;
neither will successive renormalizations of the family accumulate on the
unstable manifold of the fixed point, as in the classical setting.
In particular, there is no parameter universality in the traditional sense, but
instead we see much more intricate behavior (see the conjecture
in~\sref{results}).

\subsection{Results}
\label{results}

Fix a topological class $\tc$ of infinitely $(a,b)$--renormalizable Lorenz
maps with $a$ and~$b$ sufficiently large (see \sref{preliminaries} for
definitions).
We would like to highlight two results, illustrated in Figure~\ref{topclass},
which are the consequences of the instability of renormalization discussed in
the introduction.

\begin{figure}
    {\footnotesize
    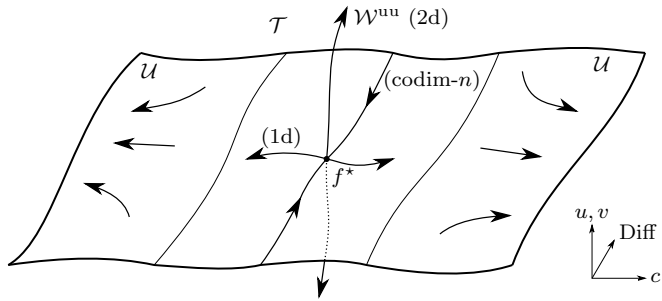
    \caption{%
        Illustration of the dynamics on $\tc$.
        \label{topclass}
        }
    }
\end{figure}

\begin{coexistence-thm}
    There exist a nonempty open set $\tu \subset \tc$ and a
    renormalization fixed point $f^\star \in \tc\setminus\tu$
    such that the successive renormalizations of $f \in \tu$ degenerate and the
    successive renormalizations of $f \in \tc\setminus \tu$ have convergent
    subsequences.
\end{coexistence-thm}

\begin{unstable-mfd-thm}
    The fixed point $f^\star$ has an unstable manifold $\wu$.
    The dimension of $\wu$ is at least three and
    every neighborhood of $f^\star$ in $\wu$ intersects
    $\tc\setminus f^\star$.
    Furthermore, $f^\star$ has a two-dimensional strong unstable manifold
    $\wuu \subset \wu$.
    It is a topologically full family and $\wuu \cap \tc = f^\star$.
\end{unstable-mfd-thm}

The Coexistence Theorem follows from the Degeneration Theorem of
\sref{applications} (which defines ``degeneration'') and the existence
of fixed points, see \sref{existence-of-fixed-points}.
The Dimensional Discrepancy Theorem follows from the results of
\sref{unstable-manifolds}.
That $\wuu$ meets $\tc$ in a unique point is related to the question of
monotonicity of entropy; in fact, we prove this intersection property for a
class of families which contains the family $\wuu$.
We do not prove anything related to the stable manifold, but we know that
it has finite codimension (see the remark after Theorem~\ref{R-diff}).

We make the following conjecture regarding the structure of~$\tc$:
\begin{T-conjecture}
    $\tc$ is a codim--$2$ manifold and $f^\star$ is hyperbolic.
    For $a$ and $b$ sufficiently large:
    \begin{enumerate*}
        \item $\dim \wu = 3$,
        \item $\tc \setminus \tu$ is the stable manifold of $f^\star$,
        \item $\tu$ is foliated by codim--$1$ rigidity classes and
            $\tc\setminus\tu$ is the rigidity class of $f^\star$,
        \item a generic $2$--dim family $\mathcal F$ intersects $\tu$ in
            a unique rigidity class and the domains of $n$ times
            renormalizability inside $\mathcal F$ shrink super-exponentially in
            a universal way; if $\mathcal F$ intersects
            $\tu\setminus\tc$, then generically these domains shrink
            exponentially in a universal way.
    \end{enumerate*}
\end{T-conjecture}
By the conjecture a generic $3$--dim family intersects $\tc$ in a curve and
points on this curve corresponds to different rigidity classes.
In other words, there is a kind of parameter universality for such families.
The conjecture also shows that there are two kinds of parameter universality
for $2$--dim families, depending on whether they hit $\tu$ (the generic case)
or not.
The above is in agreement with the Rigidity Conjecture \cite{MPW17} but it is
slightly stronger as we do not expect to see probabilistic rigidity classes as
in the H\'enon case.
It is important to note that the condition on $a$ and $b$ both being large is
essential: e.g.\ conjecturally there is no instability within the topological
classes for the combinatorics of \cite{W10} ($a=1$, $b=2$, critical exponent
$\alpha=2$) and \cite{MW14} ($a$ small, $b$ large).
Moreover, we conjecture that there are topological classes (e.g.\ $a=2$,
$b=8$, critical exponent $\alpha=2$) for which the fixed point has $2$--dim
unstable behavior, but where there still is instability inside the class due to
the presence of a period--$2$ point of renormalization with $3$--dim unstable
behavior.

In the process of proving the above theorems we are able to get precise control
over how the critical point moves under renormalization as well as to get
bounds on the distortion, see \sref{technical-lemmas}.
In \sref{applications} we use these results to derive many important dynamical
properties for infinitely $(a,b)$--renormalizable Lorenz maps.
In particular we show that there are no wandering intervals (implying that
two such maps are topologically conjugate), see
Theorem~\ref{dynamical-properties}.
Under what conditions there are no wandering intervals for Lorenz maps in
general is an important question which is still wide open.
Furthermore, we prove that such maps have a measure zero minimal Cantor
attractor and describe the invariant measures on the attractor as well as its
Hausdorff dimension.

As a closing remark, note that our techniques are based on real analytical
tools and work for arbitrary real exponents $\alpha > 1$.
We generally work in the $\Ck3$ category which traditionally presents
significant difficulties as the classical renormalization operator is not
differentiable in this class \cite{FMP06}.
This issue is avoided in \sref{the-internal-structure-of-renormalization} by
defining renormalization over internal structures instead of over
diffeomorphisms \cite{M98}.
In particular, our renormalization operator is differentiable in this category,
see Theorem~\ref{R-diff}.
We construct unstable manifolds in the space of internal structures and by
composing we obtain results for the actual system,
see~\sref{proofs-of-the-unstable-manifold-theorems}.

\section{Preliminaries}
\label{preliminaries}

In this section we define the space of Lorenz maps as well as the
renormalization operator acting on this space.
For more details, see \cite{MW14}.

\subsection{The space of Lorenz maps}

The \key{standard family} $(u,v,c) \mapsto q(x)$ is defined by
$q|_{[0,c)} = q_-$ and $q|_{(c,1]} = q_+$, where $u,v \in [0,1]$, $c \in (0,1)$
and
\begin{equation} \label{q}
    \left\{\begin{aligned}
        q_-(x) &= u \Big( 1 - \abs[\Big]{\frac{c-x}{c}}^\alpha \Big),
        \\
        q_+(x) &= 1 + v \Big( -1 + \abs[\Big]{\frac{x-c}{1-c}}^\alpha \Big).
    \end{aligned} \right.
    \qquad\qquad\qquad
    \raisebox{-17mm}{\small 
\begingroup%
  \makeatletter%
  \providecommand\color[2][]{%
    \errmessage{(Inkscape) Color is used for the text in Inkscape, but the package 'color.sty' is not loaded}%
    \renewcommand\color[2][]{}%
  }%
  \providecommand\transparent[1]{%
    \errmessage{(Inkscape) Transparency is used (non-zero) for the text in Inkscape, but the package 'transparent.sty' is not loaded}%
    \renewcommand\transparent[1]{}%
  }%
  \providecommand\rotatebox[2]{#2}%
  \ifx\svgwidth\undefined%
    \setlength{\unitlength}{87.27787348bp}%
    \ifx\svgscale\undefined%
      \relax%
    \else%
      \setlength{\unitlength}{\unitlength * \real{\svgscale}}%
    \fi%
  \else%
    \setlength{\unitlength}{\svgwidth}%
  \fi%
  \global\let\svgwidth\undefined%
  \global\let\svgscale\undefined%
  \makeatother%
  \begin{picture}(1,1.04663679)%
    \put(10.62201219,-8.17679408){\color[rgb]{0,0,0}\makebox(0,0)[lt]{\begin{minipage}{0.649568\unitlength}\centering \end{minipage}}}%
    \put(0,0){\includegraphics[width=\unitlength,page=1]{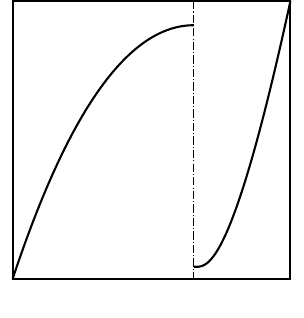}}%
    \put(0.37305236,0.57919983){\color[rgb]{0,0,0}\makebox(0,0)[b]{\smash{$q_-$}}}%
    \put(0.75364597,0.57806148){\color[rgb]{0,0,0}\makebox(0,0)[b]{\smash{$q_+$}}}%
    \put(0.63949631,0.00702871){\color[rgb]{0,0,0}\makebox(0,0)[b]{\smash{$c$}}}%
    \put(0.04170526,0.00702871){\color[rgb]{0,0,0}\makebox(0,0)[b]{\smash{$0$}}}%
    \put(0.95831809,0.00702871){\color[rgb]{0,0,0}\makebox(0,0)[b]{\smash{$1$}}}%
    \put(0.47938249,0.16546152){\color[rgb]{0,0,0}\makebox(0,0)[b]{\smash{$1-v$}}}%
    \put(0.72147892,0.94157465){\color[rgb]{0,0,0}\makebox(0,0)[b]{\smash{$u$}}}%
    \put(0,0){\includegraphics[width=\unitlength,page=2]{lorenzmap.pdf}}%
  \end{picture}%
\endgroup%
}
\end{equation}
Here $\alpha \in \reals$ the \key{critical exponent}, and $c$ the
\key{critical point};
$\alpha > 1$ is fixed throughout.

Let $\diff^2$ denote the set of orientation-preserving $\Ck2$--diffeomorphisms
on~$[0,1]$;
it is a Banach space with norm $\norm{\phi} = \sup\abs{\nonlin\phi}$ and linear
structure
\begin{equation*}
    t_1 \phi_1 \oplus t_2\phi_2
    =
    \inv\nonlin(t_1\nonlin\phi_1 + t_2\nonlin\phi_2),
    \qquad \phi_i \in \diff^2, \; t_i \in \reals.
\end{equation*}
The bijection $\nonlin:\diff^2\to\Ck0$;
$\phi \mapsto \deriv\log\deriv\phi$,
is called the \key{nonlinearity operator}.
Let $\sdiff \subset \diff^2$ denote the convex subset of
$\Ck3$--diffeomorphisms with non-positive Schwarzian derivative,
$\schwarz\phi = \deriv\nonlin\phi - (\nonlin\phi)^2/2 \leq 0$.

Let $f$ be a map with two increasing branches $f_\pm$ of the form
$f_- = \phi\circ q_-$ and $f_+ = \psi\circ q_+$, with $\phi,\psi\in\sdiff$.
Note that $\schwarz f < 0$ and that $f$ is undefined at the critical point
$c$, but the \key{critical values} $f_-(c)$ and $f_+(c)$ are well-defined.
We call $f$ a \key{Lorenz map} iff
$x < f_-(x)$, $\forall x \in (0,c]$ and $f_+(x) < x$, $\forall x\in[c,1)$;
it is identified with the tuple $(u,v,c,\phi,\psi)$ using \eqref{q}.
The set of all Lorenz maps is denoted $\lorenz$.
We identify $\lorenz$ with a subset of $\reals^3\times\sdiff\times\sdiff$ and
use the product topology.
Define
\begin{equation*}
    \lorenz_\delta
    =
    \{(u,v,c,\phi,\psi)\in\lorenz \mid \norm\phi,\norm\psi < \delta\}.
\end{equation*}

We say that a branch $f_\pm$ is \key{trivial} iff $f_\pm(c)=c$
and we say that $f_\pm$ is \key{full} iff $f_-(c)=1$ resp.\ $f_+(c)=0$.
We say that $f$ is \key{full} iff both branches are full.

\subsection{Renormalization}

We say that $f \in \lorenz$ is \key{renormalizable} iff there exists a closed
interval $C$ such that $\Int C \ni c$, $C \neq [0,1]$, and such that the
first-return map to~$C$ is affinely conjugate to some $g \in \lorenz$.
The \key{renormalization operator} $\renorm$ is defined by taking the largest
such $C$ and sending $f$ to $g$.
We call $\renorm f$ the \key{renormalization} of~$f$ and we call $C$ the
\key{return interval} of~$f$.
We say that $f$ is \key{infinitely renormalizable} iff $\renorm^n f$ is
renormalizable, $\forall n\geq 0$.

We say that $f$ is \key{$(a,b)$--renormalizable} iff
\begin{gather*}
    f^1(C_-) > \dots > f^a(C_-) > c, \quad C_- \subset f^{a+1}(C_-) \subset C,
    \quad C_- = C \cap \{x<c\},
    \\
    f^1(C_+) < \dots < f^b(C_+) < c, \quad C_+ \subset f^{b+1}(C_+) \subset C,
    \quad C_+ = C \cap \{x>c\}.
\end{gather*}
In this case $\partial_- C$ resp.\ $\partial_+ C$ are periodic points of
periods $a+1$ resp.\ $b+1$, and identifying $\renorm f$ with
$(u',v',c',\phi',\psi')$ we have
\begin{equation} \label{Rf}
    u' = \frac{\abs{q(C_-)}}{\abs U},
    \quad
    v' = \frac{\abs{q(C_+)}}{\abs V},
    \quad
    c' = \frac{\abs{C_-}}{\abs C},
    \quad
    \phi' = \zoom{\Phi}{U},
    \quad
    \psi' = \zoom{\Psi}{V},
\end{equation}
where $\Phi = f_+^a\circ\phi$, $\Psi = f_-^b\circ\psi$, $U = \inv\Phi(C)$,
$V = \inv\Psi(C)$, and $\zoom{g}{I}$ denotes the affine rescaling of the domain
and range of $g|_I$ to $[0,1]$.\footnote{%
    That is, $\zoom{g}{I} = \inv h_{g(I)}\circ g\circ h_I$, where
    $h_{[x,y]}(t) = (1-t)x + ty$.}
Note that $c'\neq c$ in general.
This movement of the critical point is key in understanding the new
renormalization phenomena we describe here.

We will almost exclusively consider $(a,b)$--renormalizable maps and will
maintain the convention of using primes to denote variables associated with
the renormalization (we consistently use $\deriv$ for derivatives).

\section{Technical lemmas}
\label{technical-lemmas}

In this section we state and prove technical lemmas which we then apply in the
following section to prove things like ergodicity, non-existence of wandering
intervals, etc.
This section can be skipped on a first read-through and referenced back to
later on.
We use the letter $K$ to denote an anonymous constant and follow the
convention of reusing the same letter for different constants.

\subsection{A fundamental lemma}

The following lemma is the starting point for all other results.
Its main content is that the critical values of a renormalizable map are
``expanded away'' from the critical point.
It is fundamental in controlling the expansion along postcritical orbits.

\begin{expansion-lemma}
    For every $\delta < \infty$ and $\gamma \in (0,1)$ there exists $\rho > 0$
    such that if $f \in \lorenz_\delta$ is renormalizable and $c < 1 - \gamma$,
    then $c - f_-^{-1}(c) \geq \rho c$ and
    $f_+^{-1}(c) - c \geq \rho c^{1/\alpha}$.
\end{expansion-lemma}

\begin{remark}
    We will repeatedly make use of this lemma in the following way:
    let $f$ be $(a,b)$--renormalizable so that $f_-(c) > \inv[a] f_+(c)$
    and $f_+(c) < \inv[b] f_-(c)$.
    Because of the lemma the backward orbits $\inv[a] f_+(c)$ and
    $\inv[b] f_-(c)$ approach $1$ and $0$ exponentially fast, respectively.
    Hence, all $(a,b)$--renormalizable maps
    are exponentially close (in $\min\{a,b\}$) to the set of full maps.
    This allows us to make perturbation arguments away from the full maps and
    this is the central idea behind all estimates.

    Note that the convergence of the above backward orbits depends on $\gamma$
    and this leads to us having to treat the case where the critical point $c$
    is bounded away from $0$ and $1$ separately from the case where $c$
    is allowed to approach $0$ or $1$.
\end{remark}

\begin{proof}
    We claim that
    $\forall\delta<\infty$ $\exists\rho \in (0,1)$ such that
    \begin{equation} \label{corner-expansion-claim}
        \frac{\phi^{-1}(c)}{u} \leq 1 - \rho^\alpha.
    \end{equation}
    Let us show that \eqref{corner-expansion-claim} implies the lemma before
    proving the claim.

    From \eqref{corner-expansion-claim}, the fact that
    $\inv f_-(c) = \inv q_- \circ \inv\phi(c)$ and \eqref{q}
    we get $c - \inv f_-(c) \geq \rho c$.
    Since $f_+(c) < \inv f_-(c)$, $f_+(c) = \psi(1 - v)$ and $v \leq 1$ we can
    apply this and \eqref{q}
    to get
    \begin{equation*}
        \left( \frac{\inv f_+(c) - c}{1 - c} \right)^\alpha
        = \frac{\inv\psi(c) - (1-v)}{v}
        = \frac1v \abs*{\inv\psi\left( [f_+(c),c] \right)}
        \geq e^{-\delta} \abs*{c - \inv f_-(c)}.
    \end{equation*}
    Hence $\inv f_+(c) - c \geq \gamma (e^{-\delta} \rho c)^{1/\alpha}$
    and the lemma follows.

    We now prove \eqref{corner-expansion-claim}.
    Let $I = [\inv f_-(c),c]$.  Since $f$ is renormalizable
    $f^2(I) \supset I$.
    Hence, $\exists t\in I$ such that $\deriv f^2(t) \geq 1$ which together
    with \eqref{q}
    and~\eqref{q}
    gives
    \begin{equation*}
        1 \leq \frac{e^{2\delta} \alpha^2 u}{c (1 - c)}
            \left( 1 - \frac{\inv f_-(c)}{c} \right)^{\alpha - 1}
        = \frac{e^{2\delta} \alpha^2 u}{c (1 - c)}
            \left( 1 - \frac{\inv\phi(c)}{u} \right)^{1 - 1/\alpha}
            \mkern-40mu.
    \end{equation*}
    Since $1 - c > \gamma$ and $\inv\phi(c) \leq e^\delta c$ this shows that
    \begin{equation} \label{corner-expansion-phic-over-u}
        \frac{u}{\inv\phi(c)}
            \left( 1 - \frac{\inv\phi(c)}{u} \right)^{1 - 1/\alpha}
        \geq \frac 1K.
    \end{equation}
    The left-hand side approaches~$0$ as $\inv\phi(c) \to u$,
    so \eqref{corner-expansion-phic-over-u}
    implies \eqref{corner-expansion-claim}.
\end{proof}

\subsection{Controlling Koebe space}

In this subsection we deal with the central task of controlling the distortion
of first-entries to the return interval $C$ of a renormalizable map.
We employ the Koebe Lemma~\ref{koebe-lemma} for this purpose and to that end we
spend most of this subsection controlling the ``Koebe space'' around $C$.
Lemma~\ref{monotone-extension} shows where the Koebe space around $C$ is
located and the remaining lemmas are concerned with estimating its size.
The mechanisms which govern the size of the Koebe space are different when the
critical point is bounded compared to when the critical point approaches the
boundary.
The former case is covered by Lemma~\ref{bounded-space} and the latter case is
covered by Lemmas \ref{corner-space-big-branch}
and~\ref{corner-space-small-branch}.
Both cases are then summarized in Lemma~\ref{bounded-geometry}.

\begin{lemma}[Monotone extension] \label{monotone-extension}
    Let $f$ be renormalizable with return times $(m,n)$.
    If $f^k|_I : I \to C$ is a first-entry map to~$C$ with monotone
    extension $f^k|_J$, then $f^k(J) \setminus C$ contains $f^j(C_+)$ in the
    left component and $f^i(C_-)$ in the right component,
    for some $j \in \{1,\dotsc,n-1\}$ and~$i \in \{1,\dotsc,m-1\}$.
\end{lemma}

\begin{remark}
    In particular, if $f$ is $(a,b)$--renormalizable then the Koebe space
    around $C$ extends at least to the preimages $\inv f_-(c)$ and
    $\inv f_+(c)$ since these points are contained in $f^b(C_+)$ and
    $f^a(C_-)$, respectively.
\end{remark}

\begin{proof}
    By definition $\exists i > 0$ such that $f^{k-i}(J) \setminus f^{k-i}(I)$
    contains $C_-$ in its right component.  Hence $f^k(J) \setminus C$ contains
    $f^i(C_-)$ in its right component.  If $i > m$, then $f^{k-i+m}(J)$ would
    contain $c$, which contradicts monotonicity of $f^k|_J$.  If $i = m$, then
    $f^k(I) \cap C_- = \emptyset$ since $f^k|_J$ is monotone and
    $C_- \subset f^m(C_-) \subset C$ by renormalizability, which
    contradicts $f^k(I) = C$.  Hence $i < m$.
    Repeat the argument on the left.
\end{proof}

In the following lemma we consider the case where $c$ is allowed to approach
$0$.
In this case the right branch of $f$ is called the ``big branch'' as it may
take up most of the unit interval.
We will only state and prove results for $c$ close to $0$.
There is always a symmetrical statement and proof for $c$ close to $1$ which
can be obtained by conjugating with the involution $x \mapsto 1 - x$.

\begin{lemma}[Size of $C$ in big branch] \label{corner-space-big-branch}
    For every $\delta < \infty$ and $\eps > 0$ there exists $\gamma > 0$
    such that if $f \in \lorenz_\delta$ is $(a,b)$--renormalizable,
    $c < \gamma$ and $b > \alpha + 1$, then
    $\abs{C_+} \leq \eps \min\{\abs{c - f_-^{-1}(c)}, \abs{f_+^{-1}(c) - c}\}$.
\end{lemma}

\begin{proof}
    By the Expansion Lemma $\inv f_+(c) - c$ is larger
    than the whole domain of the small branch as $c\downarrow 0$, so we only
    need to prove that $\abs{C_+} \leq \eps \abs{c - \inv f_-(c)}$.

    We claim that $\forall\eps > 0$,
    $\forall\delta < \infty$ $\exists \gamma > 0$ such that if $c < \gamma$
    and $n > \alpha$, then
    \begin{equation} \label{corner-space-big-branch-eps}
        \inv[n] f_-(c) \leq (\eps c)^\alpha.
    \end{equation}
    Before proving the claim let us show how this implies the lemma.  Since
    $f(C_+)$ first enters $C$ after $b$ steps
    $f(C_+) \subset [0, \inv[b+1] f_-(c)]$, so
    $\abs{f(C_+)} < (\eps c)^\alpha$ for $b-1 > \alpha$
    by~\eqref{corner-space-big-branch-eps}.
    Equation~\eqref{q} can be used to estimate
    $\abs{f(C_+)} \geq e^{-\delta} v (\abs{C_+}/(1-c))^\alpha$.
    From $c \geq f_+(c) = \psi(1-v)$ we get
    $v \geq 1 - e^\delta c$.  Taken all together we get
    $\abs{C_+} \leq K \eps c$.
    This finishes the proof, since $c - \inv f_-(c) \geq \rho c$ by the
    Expansion Lemma.

    Let us prove \eqref{corner-space-big-branch-eps}.
    The Expansion Lemma gives
    $c - \inv f_-(c) \geq \rho c$ which means that $0$ uniformly attracts
    $\inv f_-(c)$ under iteration of $\inv f_-$.  Hence
    $\inv[n] f_-(c) / c \leq K \deriv \inv[n]{f(0)}$, $\forall n$.
    Equation \eqref{q}
    can be used to estimate
    $\deriv f(0) \geq e^{-\delta} \alpha u / c$.
    Since $f$ is renormalizable $\inv f_+(c) \leq f_-(c) = \phi(u)$, so the
    Expansion Lemma implies that
    $u \geq c^{1/\alpha} / K$.
    Conclusively,
    $\inv[n] f_-(c) \leq K c (K c^{1-1/\alpha})^n \leq (K c^{1-1/\alpha})^{n-\alpha} c^\alpha$.
    Define $\gamma$ by $(K \gamma^{1-1/\alpha})^{n-\alpha} = \eps^\alpha$ and
    the claim follows.
\end{proof}

The following lemma is a simple induction result on certain backward orbits of
the critical point which will be needed in a few places.

\begin{lemma} \label{critical-preimages}
    Let $\kappa = e^{\delta/(\alpha-1)}\alpha/(\alpha - 1)$
    and $\alpha_n = 1/\alpha^n$.
    Then
    $c - \inv[n] f_-(c) \leq \kappa c (1 - c)^{\alpha_n}$
    and
    $\inv[n] f_+(c) - c \leq \kappa (1 - c) c^{\alpha_n}$,
    for every $n \geq 1$
\end{lemma}

\begin{proof}
    Let $s_n = 1 + \alpha_1 + \dots + \alpha_n$.
    We claim that for all $n \geq 1$
    \begin{equation} \label{critical-preimages-induction}
        \inv[n]f_+(c) - c
        \leq (1 - c) e^{s_{n-1} \delta/\alpha} s_{n-1} c^{\alpha_n}.
    \end{equation}
    From this the lemma follows since $s_n < \alpha/(\alpha - 1)$.

    The proof of the claim is by induction.  From \eqref{q}
    and $v\leq 1$ we get
    \begin{equation*}
        \frac{\inv f_+(x) - c}{1 - c}
        \leq \inv\psi(x)^{1/\alpha}
        \leq (e^\delta x)^{1/\alpha}.
    \end{equation*}
    Let $x=c$ to prove the base case.  Make the induction assumption that
    \eqref{critical-preimages-induction} holds for some $n$.  Then
    \begin{align*}
        \frac{\inv[(n+1)] f_+(c) - c}{1 - c}
        &\leq e^{\delta/\alpha} \inv[n]f_+(c)^{1/\alpha}
        \leq e^{\delta/\alpha}
            \left( c + e^{s_{n-1} \delta/\alpha}
            s_{n-1} c^{\alpha_n} \right)^{1/\alpha}
        \\
        &\leq e^{\delta/\alpha(1 + s_{n-1}/\alpha)}
            \left( c^{1-\alpha_n}
            e^{-s_{n-1} \delta/\alpha}
            + s_{n-1} \right)^{1/\alpha}
            c^{\alpha_{n+1}}.
    \end{align*}
    The term in parenthesis is less than $1 + s_{n-1}$.  Use the fact that
    $(1 + s_{n-1})^{1/\alpha} < 1+s_{n-1}/\alpha = s_n$ to finish the proof.
\end{proof}

The following lemma is the counterpart to Lemma~\ref{corner-space-big-branch}
but for the ``small branch,'' i.e.\ the left branch when $c$ is allowed to
approach $0$.

\begin{lemma}[Size of $C$ in small branch] \label{corner-space-small-branch}
    For every $\delta < \infty$ and $\eps > 0$ there exist $N < \infty$ and
    $\gamma > 0$ such that if $f \in \lorenz_\delta$ is
    $(a,b)$--renormalizable, $c < \gamma$ and $a \geq N$, then
    $\abs{C_-} \leq \eps \min\{\abs{c - f_-^{-1}(c)}, \abs{f_+^{-1}(c) - c}\}$.
\end{lemma}

\begin{proof}
    We only need to consider the case
    $\abs{C_-} > \abs{C_+}$, else we could apply
    Lemma~\ref{corner-space-big-branch}.  As in the proof of that
    lemma it suffices to show that $\abs{C_-} \leq \eps \abs{c - \inv f_-(c)}$.

    Let $T = [c, f_-(c)]$, $R_k = f^{a+1-k}(C_-)$ and
    $\theta_0 = \eps^\alpha e^{-\delta}$.
    If $\abs{R_a} \leq \theta_0 \abs{T}$, then \eqref{L-over-LM} gives
    $\abs{C_-} / \abs{c - \inv f_-(c)}
        = (\abs{\inv\phi(R_a)} / \abs{\inv\phi(T)})^{1/\alpha}
        \leq (e^\delta \abs{R_a} / \abs{T})^{1/\alpha} \leq \eps$.
    Hence we are done if there is enough space around $R_a$ inside $T$.
    The following claim gives two sufficient conditions for this to happen.

    Let $S_k = [\inv[k+1] f_+(c), R_k]$ for $k\leq a$,
    $S_{a+1} = \inv f_-(S_a)$ and $L_k = S_k \setminus R_k$ for $k\leq a+1$.
    We claim that $\forall\delta < \infty, \eps > 0$
    $\exists\theta_1,\theta_2 < \infty$ such that if
    \begin{enumerate}[label=(\roman*)]
        \item $\abs{R_k} \leq \theta_1 \abs{L_k}$ for some $k \leq a$, or if
            \label{corner-space-small-branch-i}
        \item $[f_+(c),1]$ contains a $\inv\theta_2$--scaled neighborhood of
            $R_k$ for some $k \leq a-1$,
            \label{corner-space-small-branch-ii}
    \end{enumerate}
    then $\abs{R_a} \leq \theta_0 \abs{T}$.
    Let use prove the claim.

    Assume \ref{corner-space-small-branch-i}.  From \eqref{M-over-MR-weak}
    we get
    \begin{equation} \label{corner-space-small-branch-L-over-S}
        \frac{\abs{L_{i+1}}}{\abs{S_{i+1}}}
        = \frac{\abs{\inv q_+ \circ \inv\psi(L_i)}}%
            {\abs{\inv q_+ \circ \inv\psi(S_i)}}
        \geq \frac{\abs{\inv\psi(L_i)}}{\abs{\inv\psi(S_i)}}
        \geq e^{-\delta \abs{S_i}} \frac{\abs{L_i}}{\abs{S_i}}
    \end{equation}
    and consequently
    $\abs{L_a}/\abs{S_a} \geq \exp\{-\delta \sum\abs{S_i}\}
        \abs{L_k}/\abs{S_k} \geq e^{-2\delta} \inv\theta_1$,
    since $\sum\abs{S_i} \leq 2$.
    Hence, $\abs{R_a}/\abs{S_a} \leq 1 - e^{-2\delta} \inv\theta_1$ and since
    $S_a \subset T$, it follows that $\abs{R_a} \leq \theta_0 \abs{T}$ for
    $\theta_1$ sufficiently small.

    Assume \ref{corner-space-small-branch-ii}.  The branch
    $f^{a-k}|_{R_a}: R_a \to R_k$ has monotone extension whose image is
    $[f_+(c),1]$ and its domain is in~$[c,1]$.  Hence the claim follows from
    the Macroscopic Koebe Principle (see~\cite{dMvS93}*{p.~287})
    by choosing $\theta_2$ small enough.  This
    concludes the proof of the claim.

    Now assume that $\delta$ and $\eps$ have been chosen and that the
    constants $\theta_1$ and~$\theta_2$ have been determined.
    If \ref{corner-space-small-branch-i} was true then there would be
    nothing to prove so assume that it is false.
    Let $N \geq 3$ be given, assume $a \geq N$, and choose $\gamma > 0$ such
    that if $c < \gamma$, then
    $\dist(R_i, 1) \geq \inv\theta_2 \abs{R_i}$,
    $\forall i = 0,\dotsc,N-1$.
    This is possible, since $R_i \subset [0, \inv[N] f_+(c)]$
    $\forall i = 0,\dotsc,N-1$
    and $\inv[N] f_+(c) - c \leq K c^{1/\alpha^N}$ by
    Lemma~\ref{critical-preimages}.
    If one of the $R_i$ also had space on the left, i.e.\
    if $\dist(R_k, f_+(c)) \geq \inv\theta_2 \abs{R_k}$ for some $k < N$,
    then we could apply \ref{corner-space-small-branch-ii} and be done.
    Consequently we assume that \ref{corner-space-small-branch-ii}
    does not hold for any $k < N$.  Let us prove that this leads to a
    contradiction.

    Let $\hat R_i = [f_+(c),R_i]$ and $\hat L_i = [f_+(c),L_i]$.
    By assumption $\abs{R_i} > \theta_1 \abs{L_i}$ and
    $\abs{\hat R_i} > (1 + \theta_2) \abs{\hat L_i}$,
    $\forall i \in \{1,\dotsc,N-1\}$.
    These are exactly the conditions needed for~\eqref{M-over-MR}.
    Repeating the estimate in \eqref{corner-space-small-branch-L-over-S}
    but using \eqref{M-over-MR} instead of \eqref{M-over-MR-weak} we get that
    $\exists \rho > 0$ (only depending on $\theta_1$ and~$\theta_2$) such
    that
    \begin{equation} \label{corner-space-small-branch-L-over-S-improved}
        \frac{\abs{L_i}}{\abs{S_i}}
        \geq e^{-\delta \abs{S_{i-1}}} \frac 1\rho
            \frac{\abs{L_{i-1}}}{\abs{S_{i-1}}},
            \qquad \forall i \in \{1,\dotsc,N-1\}.
    \end{equation}
    From \eqref{L-over-LM} and \eqref{M-over-LM} we get
    \begin{align} \label{corner-space-small-branch-L-over-S-edges}
        \frac{\abs{L_1}}{\abs{S_1}}
            &\geq e^{-\delta\abs{S_0}/\alpha}
            \left( \frac{\abs{L_0}}{\abs{S_0}} \right)^{1/\alpha},
        &
        \frac{\abs{L_{a+1}}}{\abs{S_{a+1}}}
        &\geq e^{-\delta\abs{S_a}} \frac 1\alpha
        \frac{\abs{L_a}}{\abs{S_a}},
    \end{align}
    respectively.  Since $L_{a+1} \subset L_0$,
    $R_{a+1} = C_- \subset R_0 \subset C$
    we can estimate
    \begin{equation} \label{corner-space-small-branch-L-over-S-loop}
        \frac{\abs{L_{a+1}}}{\abs{S_{a+1}}}
        = \frac{\abs{L_{a+1}}}{\abs{L_{a+1}} + \abs{R_0}}
            \left( 1 +
            \frac{\abs{R_0} - \abs{R_{a+1}}}{\abs{S_{a+1}}} \right)
        \leq \frac{\abs{L_0}}{\abs{S_0}}
            \left( 1 + \frac{\abs{C_+}}{\abs{C_-}} \right).
    \end{equation}
    By assumption $\abs{C_-} > \abs{C_+}$, which combined with
    \eqref{corner-space-small-branch-L-over-S}--\eqref{corner-space-small-branch-L-over-S-loop}
    gives
    \begin{equation}
        \frac{\abs{L_0}}{\abs{S_0}}
        \geq \inv[N+2]\rho \frac{1}{2\alpha} e^{-\delta \sum\abs{S_i}}
        \left( \frac{\abs{L_0}}{\abs{S_0}} \right)^{1/\alpha}
        \mkern-20mu.
    \end{equation}
    Since $\sum\abs{S_i} \leq 2$ this implies that
    $(\abs{L_0}/\abs{S_0})^{1-1/\alpha} \geq \inv[N]\rho/K$ which is
    impossible for $N$ large enough as the left-hand side is at most $1$
    but the right-hand side is unbounded in~$N$.  We conclude that either
    \ref{corner-space-small-branch-i} holds for some $k\leq a$ or
    \ref{corner-space-small-branch-ii} holds for some $k < N$.
\end{proof}

The following lemma controls the Koebe space when the critical point is bounded
away from the boundary.
Note that in this case we get explicit bounds on the size of the Koebe space.
We need these explicit bounds later on.

\begin{lemma}[Size of $C$] \label{bounded-space}
    For every closed interval $\Delta \subset (0,1)$ and $\delta < \ubdist$
    there exist $N < \infty$, $K < \infty$ and $\lambda > 1$ such that if
    $f \in \lorenz_\delta$ is $(a,b)$--renormalizable,
    $\min(a,b) \geq N$ and $c \in \Delta$, then
    \begin{equation*}
        \abs{C} \leq
            K \min\left\{\abs{c - f_-^{-1}(c)}, \abs{f_+^{-1}(c) - c}\right\}
            \lambda^{-\min(a,b)/\alpha}.
    \end{equation*}
\end{lemma}

\begin{remark}
    The \emph{only} reason for demanding $\delta < \ubdist$ is because it
    directly implies that $\deriv f(x) > 1$ for $x = 0, 1$ as shown in the
    proof below.
    It is by no means necessary but it makes many arguments simpler.
    We will later see that the distortion of the renormalization is tiny when
    the return times are large, so this condition can be automatically
    satisfied by renormalizing once.
    Hence we allow ourselves the convenience to assume $\delta < \ubdist$ from
    now on.
\end{remark}

\begin{proof}
    Let $\lambda$ be the infimum of $\deriv f(0)$ over all $f$ satisfying the
    assumptions of the lemma.
    Then $\lambda \geq e^{-\delta} \alpha u / c$.  From
    $c \leq f_-(c) = \phi(u)$ and $\phi(u) \leq e^\delta u$, we get
    $u/c \geq e^{-\delta}$.
    Since $\delta < \ubdist$ it follows that
    $\lambda \geq \alpha e^{-2\delta} > 1$.
    By the Expansion Lemma
    $\inv[n] f_-(c) \leq K \inv[n]{\deriv f(0)} \leq K \inv[n] \lambda$.
    Now argue as in the proof of Lemma~\ref{corner-space-big-branch} to get
    $\abs{C_+} \leq K \lambda^{-b/\alpha}$.
    Use the Expansion Lemma to see that the lemma holds with $\abs{C_+}$ in
    place of~$\abs{C}$.  Since $c$ is bounded we can repeat this argument for
    $\abs{C_-}$ and since it holds for both $\abs{C_\pm}$, it must hold
    for~$\abs{C}$.
\end{proof}

Finally, the following lemma summarizes the previous results on the size of the
Koebe space independently of the whether the critical point is bounded or not.
It shows that we can make the Koebe space large by increasing the
return times.

\begin{lemma}[Koebe space] \label{bounded-geometry}
    For every $\delta < \ubdist$ and $\tau > 0$ there exists $N < \infty$
    such that if $f \in \lorenz_\delta$ is $(a,b)$--renormalizable and
    $\min\{a,b\} \geq N$, then $[f_-^{-1}(c), f_+^{-1}(c)]$ contains a
    $\tau$--scaled neighborhood of~$C$.
\end{lemma}

\begin{proof}
    From Lemmas \ref{corner-space-big-branch}
    and~\ref{corner-space-small-branch} we get an $N_0 < \infty$ and a closed
    interval $\Delta \subset (0,1)$ such that the statement is true if
    $c \notin \Delta$ and $\min\{a,b\} \geq N_0$.
    From Lemma~\ref{bounded-space} we get an $N_1 < \infty$ such that the
    statement holds for $c \in \Delta$ and $\min\{a,b\} \geq N_1$.
    Let $N = \max\{N_0,N_1\}$ to finish the proof.
\end{proof}

\subsection{Controlling the critical point}

Having established the necessary results to control distortion in the previous
subsection we now turn to the main difficulty of Lorenz renormalization,
namely to control how the critical point moves under renormalization.
This is an essential problem --- we will later see that the fact that the
critical point may move under renormalization contributes to an ``extra''
unstable direction when the return times are large.
The mechanism behind this phenomenon is given by the important Flipping Lemma
below.

The following lemma provides a central relation between the critical point of
$f$ and the critical point of $\renorm f$.

\begin{lemma}[Position of the critical point] \label{crit-position}
    Let $\relc(f) = c / (1 - c)$ denote the relative critical point and let
    $\relv(f) = f_-(c) / (1 - f_+(c))$ denote the relative critical value.
    If~$f \in \lorenz_\delta$ is $(a,b)$--renormalizable, then
    \begin{equation*}
        \frac{\relc(\renorm f)^\alpha}{\relv(\renorm f)} =
        \kappa
        \frac{\relc(f)^\alpha}{\relv(f)}
        \frac{\deriv f^b(x)}{\deriv f^a(y)},
    \end{equation*}
    for some
    $x \in f(C_+)$, $y \in f(C_-)$ and
    $e^{-2 \delta} \leq \kappa \leq e^{2 \delta}$.
\end{lemma}

\begin{proof}
    By definition $\relc(\renorm f) = \abs{C_-}/\abs{C_+}$ and
    $\relv(\renorm f) = \abs{f^{a+1}(C_-)}/\abs{f^{b+1}(C_+)}$,
    so we are looking for an expression involving these quantities.

    By the mean-value theorem there exist $x \in f(C_+)$,
    $x_0 \in q(C_+)$, and $x_1 \in [1-v,1]$
    such that:
    \begin{enumerate*}[label=(\roman*)]
        \item $\abs{f^{b+1}(C_+)} = \deriv f^b(x) \abs{f(C_+)}$,
        \item $\abs{f(C_+)} = v \deriv\psi(x_0) ( \abs{C_+}/(1-c) )^\alpha$,
            by \eqref{q}, and
        \item $\deriv\psi(x_1) v = 1 - f_+(c)$,
            since $1 - f_+(c) = 1 - \psi(1-v) = \abs{\psi([1-v,1])}$.
    \end{enumerate*}
    Putting all of this together we get
    \begin{equation*}
        \abs{f^{b+1}(C_+)}
        = \frac{\deriv\psi(x_0)}{\deriv\psi(x_1)}
            (1 - f_+(c)) \deriv f^b(x)
            \left( \frac{\abs{C_+}}{1 - c} \right)^\alpha
            \mkern-10mu.
    \end{equation*}
    A similar argument gives an equation for $\abs{f^{a+1}(C_-)}$.
    Divide the two equations and apply Lemma~\ref{nonlinearity-prop} to finish
    the proof.
\end{proof}

The rest of this subsection is dedicated to the proof of the Flipping Lemma.
It shows that renormalization contracts the critical point very strongly toward
the boundary.
The name comes from the fact that if $c(f)$ is close to $0$, then $c(\renorm
f)$ is close to $1$.
That is, the position of the critical point ``flips'' under renormalization.
Note that it is essential that the return times are not too small; otherwise
this flipping does not occur (although we do not prove that statement here).

\begin{flipping-lemma}
    For every compact interval $P \subset \posreals$, $\delta < \infty$ and
    $\sigma \in (0,1)$ there exist $N < \infty$ and $\gamma > 0$ such that
    the following holds.
    Let $f \in \lorenz_\delta$ be $(a,b)$--renormalizable with
    $\min\{a,b\} \geq N$ and $a/b \in P$.
    If $\crit(f) < \gamma$, then
    $1 - \crit(\renorm f) < \sigma \crit(f)$.
    If $1- \crit(f) < \gamma$, then
    $\crit(\renorm f) < \sigma (1 - \crit(f))$.
\end{flipping-lemma}

\begin{remark}
    The condition on how large $a$ and~$b$ have to be is explicitly given
    by~\eqref{flipping-b}.  For $\alpha = 2$ it can be seen that this
    condition is satisfied if $\min\{a,b\} \geq 4$.  Computer experiments
    indicate that this is the optimal lower bound on $N$ in the sense that
    if $a\leq 3$ then we can choose $b$ such that the Flipping Lemma is false.
\end{remark}

The following lemma is a simple induction needed in the lemma following it.

\begin{lemma} \label{critical-preimages-away-from-c}
    Let $\alpha_n = 1/\alpha^n$. For every $\delta < \infty$ and
    $\gamma \in (0,1)$ there exists
    $\rho \in (0,1)$ such that if $f \in \lorenz_\delta$ is renormalizable and
    $c < 1 - \gamma$, then $\inv[n]f_+(c) - c \geq \rho c^{\alpha_n}$,
    for all $n \geq 1$.
\end{lemma}

\begin{proof}
    Let $s_n = 1 + \alpha_1 + \dots + \alpha_n$.  We claim that
    $\exists \rho \in (0,1)$ such that
    \begin{equation} \label{critical-preimages-away-from-c-induction}
        \inv[n]f_+(c) - c
        \geq \big( \gamma e^{-\delta/\alpha} \big)^{s_{n-1}}
            (\rho c)^{\alpha_n},
    \end{equation}
    from which the lemma follows, since
    $\rho^{\alpha_n} \uparrow 1$ and $s_n < \alpha/(\alpha - 1)$.
    The proof of~\eqref{critical-preimages-away-from-c-induction} is by
    induction.  The base case follows from the Expansion Lemma.
    Assume that \eqref{critical-preimages-away-from-c-induction} holds for
    some~$n$.  Then \eqref{q},
    $v \leq 1$, $1 - c > \gamma$ and $f_+(c) \leq c$ imply that
    \begin{multline*}
        \inv[(n+1)]f_+(c) - c
        = \frac{1 - c}{v}
            \abs*{\inv\psi \left( [f_+(c),\inv[n]f_+(c)] \right)}^{1/\alpha}
        \\
        \geq \gamma e^{-\delta/\alpha}
            \abs*{ c + \big( \gamma e^{-\delta/\alpha} \big)^{s_{n-1}}
            (\rho c)^{\alpha_n} - f_+(c) }^{1/\alpha}
        \geq
            \big( \gamma e^{-\delta/\alpha} \big)^{1 + s_{n-1}/\alpha}
            (\rho c)^{\alpha_{n+1}}.
    \end{multline*}
    This completes the proof, since $1 + s_{n-1}/\alpha = s_n$.
\end{proof}

The following lemma gives bounds on the derivative of returns to $C$ which
appear in the formula of Lemma~\ref{crit-position}.
We only need this in the proof of the Flipping Lemma.

\begin{lemma} \label{corner-return-deriv}
    Let $\alpha_n = 1 / \alpha^n$.
    For every $\delta < \infty$ there exist $K < \infty$ and $\gamma > 0$
    such that if $f \in \lorenz_\delta$ is $(a,b)$--renormalizable,
    $b > \alpha + 1$  and $c < \gamma$, then
    $\deriv f_-^b(x) \geq \inv[b] K c^{-b(1-\alpha_a)}$,
    and $\deriv f_+^a(y) \leq K^a c^{1 - \alpha_a}$
    for every $x \in \inv[b]f_-(C)$ and $y \in \inv[a]f_+(C)$.
\end{lemma}

\begin{proof}
    We first prove the bound on $\deriv f_-^b$.  By \eqref{q}
    \begin{equation} \label{corner-return-deriv-Dfb}
        \deriv f_-^b(x)
        \geq \left( \frac{u}{Kc} \right)^b \;
            \prod_{i=0}^{b-1} \left( 1 - \frac{f_-^i(x)}{c} \right)^{\alpha-1}
            \mkern-28mu.
    \end{equation}
    Since $f$ is $(a,b)$--renormalizable $\phi(u) = f_-(c) \geq \inv[a]f_+(c)$,
    so $u/c \geq c^{-1+\alpha_a}/K$ by
    Lemma~\ref{critical-preimages-away-from-c}.
    Hence, it remains to show that the product on the right-hand side is
    bounded.
    To that end, let $L = [\inv f_-(c),c]$ and $R = [c,f_-(c)]$.
    By \eqref{M-over-LM}
    $\abs{\inv f_-(C_+)} / \abs{L} \leq e^\delta \abs{C_+} / \abs{R}$,
    which is small for $c < \gamma$ by Lemma~\ref{corner-space-big-branch}.
    This and the Expansion Lemma show that $f^b(t)/c \leq \rho$ for all
    $t \in \inv f_-(C)$, where $\rho \in (0,1)$ only depends on $\delta$.
    Hence, any $t \in \inv f_-(C)$ is attracted to~$0$ under iteration of $\inv
    f_-$ and the product in \eqref{corner-return-deriv-Dfb} has a uniform
    bound as claimed.

    We now prove the bound on $\deriv f_+^a$.
    Assume first that $y \in \inv[a]f_+(C_-)$, so that
    $f_+^i(y) \leq f_+^{-a+i}(c)$.
    Then we can use Lemma~\ref{critical-preimages} and \eqref{q}
    to estimate
    \begin{equation} \label{corner-return-deriv-Dfa}
        \deriv f_+^a(y)
        \leq \prod_{i=0}^{a-1}
            K \left( \inv[a+i]f_+(c) - c \right)^{\alpha-1}
        \leq \prod_{i=0}^{a-1}
            K \left(K c^{\alpha_{a-i}} \right)^{\alpha-1}
        \leq K^a c^{1-\alpha_a}.
    \end{equation}
    To show that \eqref{corner-return-deriv-Dfa} also holds for
    $y \in \inv[a]f_+(C_+)$ we argue as follows.
    The image of the monotone extension of $\inv[a] f_+(C_+) \to C_+$ is
    $[f_+(c),1]$ and this contains a $1$--scaled neighborhood of $C_+$ for
    $c < \gamma$ by Lemma~\ref{corner-space-big-branch}.
    Hence the Koebe Lemma~\ref{koebe-lemma} implies that
    $\deriv f_+^a(y) \leq K \deriv f_+^a(c)$
    for all $y \in \inv[a]f_+(C_+)$.
\end{proof}

\begin{proof}[Proof of the Flipping Lemma]
    Assume without loss of generality that $c$ is close to $0$ and let
    $N > \alpha + 1$ to begin with.
    We claim that $1 - \crit(\renorm f) \leq \sigma \crit(f)$, if
    \begin{equation} \label{flipping-cond}
        \frac{(\sigma c^2)^\alpha}{e^{2\delta}}
            \frac{\deriv f_-^b(x)}{\deriv f_+^a(y)}
        \geq 1,
        \qquad \forall x \in \inv[b]f_-(C), \forall y \in \inv[a]f_+(C).
    \end{equation}
    To prove this we will use Lemma~\ref{crit-position} and the notation
    defined there.
    From
    $0 \leq f_+(c) \leq c \leq f_-(c) \leq 1$ we get
    $c \leq \relv(f) \leq \inv{(1-c)}$ and
    \begin{equation} \label{flipping-rel}
        \frac{c^\alpha}{(1-c)^{\alpha-1}}
        \leq \frac{\relc(f)^\alpha}{\relv(f)}
        \leq \frac{c^{\alpha-1}}{(1-c)^\alpha}.
    \end{equation}
    Let $c' = \crit(\renorm f)$.  Then \eqref{flipping-rel} and
    Lemma~\ref{crit-position} gives
    \begin{equation*}
        \frac{1}{(1-c')^\alpha} \frac{1}{c^\alpha}
        \geq \frac{(c')^{\alpha-1}}{(1-c')^\alpha}
            \frac{(1-c)^{\alpha-1}}{c^\alpha}
        \geq \frac{\relc(\renorm f)^\alpha}{\relv(\renorm f)}
            \frac{\relv(f)}{\relc(f)^\alpha}
        \geq e^{-2\delta} \tau,
    \end{equation*}
    where $\tau = \inf\deriv f_-^b(x) / \deriv f_+^a(y)$
    over all $x \in \inv[b]f_-(C)$ and $y \in \inv[a]f_+(C)$.
    This proves that $(1-c')^\alpha \leq e^{2\delta} \inv\tau \inv[\alpha] c$.
    Hence, if $e^{2\delta} \inv\tau \inv[\alpha] c \leq (\sigma c)^\alpha$, then
    $1 - c' \leq \sigma c$.  This concludes the proof of~\eqref{flipping-cond}.

    We now use \eqref{flipping-cond} to prove the lemma.  Let $\beta = a/b$,
    $t \in (0,1)$ and apply Lemma~\ref{corner-return-deriv}:
    \begin{equation*}
        \tau \geq \inv[a] K \inv[b] K c^{-(b+1)(1-\alpha_a)}
        = \big( K^{-(\beta+1)} c^{-(1-t)(1-\alpha_a)} \big)^b
            c^{-(tb+1)(1-\alpha_a)},
    \end{equation*}
    where $\alpha_a = 1/\alpha^a$.
    Choose $\gamma \in (0,1)$ such that
    $K^{-(\beta+1)} \gamma^{-(1-t)(1-\alpha_a)} \geq 1$ for all $\beta \in P$.
    This is possible since $P$ is compact and the exponent of $\gamma$ is
    negative.  Hence, if $c < \gamma$, then
    $\tau > c^{-(tb+1)(1-\alpha^{-a})}$ for all~$b$.  Insert this
    into~\eqref{flipping-cond} to get
    \begin{equation*}
        \log\left\{ \frac{(\sigma c^2)^\alpha \tau}{e^{2\delta}} \right\}
        \geq \log\left\{ \frac{\sigma^\alpha}{e^{2\delta}} \right\}
            + \left( -2\alpha + (tb+1)(1-\alpha^{-a}) \right) \log \inv c.
    \end{equation*}
    The term in front of $\log\inv c$ is positive if
    \begin{equation} \label{flipping-b}
        b > \frac 1t \left( \frac{2\alpha}{1 - \alpha^{-a}} - 1 \right).
    \end{equation}
    This proves that \eqref{flipping-cond} holds if \eqref{flipping-b} is
    satisfied and $c < \gamma$, for $\gamma$ small enough.
\end{proof}

\section{Applications}
\label{applications}

In this section we apply the lemmas of the previous section.
The first result shows that the distortion of the renormalization may be
chosen arbitrarily small by increasing the return times.
Note that this holds irrespective of the position of the critical point.
It was not previously known if the distortion might blow up as the critical
point approached the boundary.

\begin{proposition}[Distortion invariance] \label{distortion-invariance}
    For every $\delta < \ubdist$ and $\eps > 0$ there exists
    $N < \infty$ such that if $f \in \lorenz_{\delta}$ is
    $(a,b)$--renormalizable and $\min\{a,b\} \geq N$, then
    $\renorm f \in \lorenz_{\eps}$.
    If in addition $c(f) \in \Delta$ then we may choose
    $\eps \leq \lambda^{\min\{a,b\}}$, where $\lambda < 1$ only depends on the
    closed interval $\Delta \subset (0,1)$ and $\delta$.
\end{proposition}

\begin{proof}
    The renormalization is given by \eqref{Rf}.
    By Lemmas \ref{monotone-extension} and~\ref{bounded-geometry} the
    monotone extensions of $\Phi$ and $\Psi$ of~\eqref{Rf} have
    images which contain arbitrary amounts of space around $C$.  Hence the
    first statement follows from the Koebe Lemma~\ref{koebe-lemma}.
    The second statement follows if we use Lemma~\ref{bounded-space} in place
    of Lemma~\ref{bounded-geometry}.
\end{proof}

The rest of this section concerns infinitely renormalizable maps whose
return times are not too small.
Some results also require that the return time on the left is comparable
to the return time on the right.
To this end we say that $f$ is of type $\Omega(N;P)$ if:
\begin{enumerate*}
    \item $\renorm^{k-1} f$ is $(a_k,b_k)$--renormalizable,
    \item $\min\{a_k,b_k\} \geq N$, and
    \item $a_k/b_k \in P$,
\end{enumerate*}
for all $k \geq 1$.
We write $\Omega(N)$ as shorthand for $\Omega(N;\posreals)$, i.e.\ maps of
type $\Omega(N)$ need not have comparable return times on the left and on the
right.
Finally, we say that $f$ is of \key{bounded type} if
$\sup_k \max\{a_k,b_k\} < \infty$.

Given an infinitely renormalizable map $f$ we are interested in the behavior of
the \key{successive renormalizations}, $f$, $\renorm f$, $\renorm^2 f$, etc.
By historical precedence we strongly expect the successive renormalizations
to have a convergent subsequence, at least for bounded type.
However, for Lorenz maps this is not always the case and instead we see the new
phenomenon of degeneration.
We will say that the successive renormalizations of $f$ \key{degenerate} if $f$
satisfies the following theorem.
Note that it is important here that the return times are comparable on the left
and on the right.
Without this condition it is more or less possible to choose the accumulation
points of $\{c(\renorm^k f)\}_k$ (but we do not prove that statement here).

\begin{degeneration-thm}
    For every compact interval $P \subset \posreals$ and $\delta < \ubdist$
    there exist a closed interval $\Delta \subset (0,1)$ and $N < \infty$ such
    that if $\renorm^n f \in \lorenz_\delta$ is infinitely
    renormalizable of type $\Omega(N;P)$ and $\crit(\renorm^n f) \notin \Delta$
    for some $n\geq 0$,
    then $\crit(\renorm^{2k} f) \to 0$ and $\crit(\renorm^{2k+1} f) \to 1$
    (or vice versa).   The rate of convergence of $\crit(\renorm^k f)$ to
    $\{0,1\}$ is faster than $\lambda^k$ for every $\lambda \in (0,1)$.
\end{degeneration-thm}

\begin{proof}
    Since distortion does not increase under renormalization
    (Proposition~\ref{distortion-invariance}), we can apply the Flipping Lemma
    repeatedly.
    The rate of convergence is faster than exponential since the contraction
    constant $\sigma$ in the Flipping Lemma is arbitrary.
\end{proof}

When the successive renormalizations of $f$ have a convergent subsequence we
say that $f$ has \key{a priori bounds}.
The exact conditions which give a priori bounds is the content of the following
theorem.

\begin{a-priori-bounds}
    For every $\delta < \ubdist$ there exists $N < \infty$ such that if
    $f \in \lorenz_\delta$ is infinitely renormalizable of type
    $\Omega(N)$ and if there exists a closed interval $\Delta \subset (0,1)$
    such that $\crit(\renorm^k f) \in \Delta$ for all $k \geq 0$, then
    $\{\renorm^k f\}_{k\geq 0}$ is a
    relatively compact family in the $\Ck0$--topology on $\lorenz_\delta$.
\end{a-priori-bounds}

\begin{proof}
    The successive renormalizations form an equicontinuous family since both
    the position of the critical point and the distortion is invariant under
    renormalization (by assumption and Proposition~\ref{distortion-invariance},
    respectively).
    Hence the result follows from the Arzel\`a--Ascoli Theorem.
\end{proof}

From the Degeneration Theorem we know that the topological classes of
type $\Omega(N;P)$ exhibit degeneration.
The following dichotomy shows that the only other alternative is a priori
bounds.
In this sense this can be seen as a precursor to the Coexistence Theorem.
That each class contains at least one map with a priori bounds is the main
content of the next section.

\begin{theorem}[Dichotomy] \label{dichotomy}
    For every compact interval $P \subset \posreals$ and $\delta < \ubdist$
    there exists $N < \infty$ such that if $f \in \lorenz_\delta$ is infinitely
    renormalizable of type $\Omega(N;P)$, then either the successive
    renormalizations of $f$ degenerate, or $f$ has a priori bounds.
\end{theorem}

\begin{proof}
This is a direct consequence of the Degeneration Theorem and the conditions
in the theorem on a priori bounds.
\end{proof}

We conclude this section with some dynamical properties of infinitely
renormalizable maps.
Note that the geometry of a degenerating map is very different from that of
a map with a priori bounds as evidenced by the difference in Hausdorff
dimension.
Also, note that there is no quasi-symmetric rigidity for topological classes
containing both degenerating maps and maps with a priori bounds.
To our knowledge this is the first example without quasi-symmetric rigidity
which is not contrived in some sense (e.g.\ such as considering conjugated maps
with different critical exponents or differing number of critical points).
This theorem is an improvement over previous results (see~\cite{MW14}) since it
holds independently of the behavior of the critical points of the successive
renormalizations.

\begin{theorem}[Dynamical properties] \label{dynamical-properties}
    For every $\delta < \ubdist$ there exists $N < \infty$ such that
    if $\renorm^n f \in \lorenz_\delta$ is infinitely
    renormalizable of type $\Omega(N)$ for some $n\geq 0$, then
    \begin{enumerate}
        \item $f$ has no wandering intervals and is ergodic,
        \item $f$ has a minimal Cantor attractor $\Lambda$ of measure zero,
        \item $\Lambda$ supports one or two ergodic invariant probability
            measures.
    \end{enumerate}
    Furthermore, if $f$ is of bounded type then
    \begin{enumerate}[resume]
        \item $\Lambda$ is uniquely ergodic and the measure on $\Lambda$ is
            physical,
        \item the Hausdorff dimension $\HD(\Lambda) = 0$ if the successive
            renormalizations of $f$ degenerate, otherwise
            $\HD(\Lambda) \in (0,1)$.
    \end{enumerate}
\end{theorem}

\begin{remark}
    If $f$ is of unbounded type, then it may not have a physical measure
    \cite{MW16}.
\end{remark}

\begin{proof}
    Assume without loss of generality that $f \in \lorenz_\delta$ is infinitely
    renormalizable of type $\Omega(N)$.  We will begin by discussing how to
    choose $N$.

    By Proposition~\ref{distortion-invariance} we can choose $N$ such that
    $\renorm^n f \in \lorenz_\delta$, $\forall n \geq 0$.
    We will need this to apply Lemma~\ref{bounded-geometry} to
    $\renorm^n f$ for each $n$.
    Since $f$ is infinitely renormalizable there exists a sequence of closed
    intervals $\{C_n\}$ such that $C_{n-1} \supset C_n$ and the first-return
    map of $f$ to $C_n$ is affinely conjugate to a map in $\lorenz$,
    $\forall n$.
    In particular, the boundary points of $C_n$ are periodic points whose
    orbits never enter $\Int C_n$.  In other words, the intervals $C_n$ are
    \key{nice} (see~\cite{MW14}*{\S3}).
    Let $T_n: D_n \to C_n$ be the \key{first-entry} map of $f$ to $C_n$.
    That is, for every $x \in D_n = \bigcup_{i\geq0} \inv[i]f(C_n)$ define
    $T(x) = f^{t(x)}(x)$, where $t(x)$ is the smallest non-negative
    integer such that $f^{t(x)}(x) \in C_n$.  Now let $f^i: I \to C_n$ be
    any branch of $T_n$ for
    some arbitrary $n$ and let $f^i: \hat I \to \hat C_n$ be its monotone
    extension.  By lemmas \ref{bounded-geometry} and~\ref{monotone-extension}
    we can choose $N$ such that $\hat C_n \cap C_{n-1}$ contains a $1$-scaled
    neighborhood
    of $C_n$.  Since the branch and $n$ was arbitrary this holds for all
    branches and for all $n$.
    This finishes our discussion on how to choose~$N$.

    To show that $f$ has no wandering intervals it suffices to show that the
    branches of $T_n$ shrink to points uniformly as $n \to \infty$
    (cf.\ \cite{MW14}*{Theorem~3.11}).  Let
    $f^i: I \to C_{n-1}$ and $f^j: J \to C_n$ be branches of $T_{n-1}$ and
    $T_n$, respectively, and assume that $I \supset J$.  We claim that there
    exists a constant $\tau > 0$, not depending on $f$, such that $I$ contains
    a $\tau$-scaled neighborhood of $J$.  From this claim it follows that
    $\abs{J} \leq (1+2\tau)^{-n}$, i.e.\ all branches shrink to points
    uniformly in~$n$.
    We will now prove the claim.  Let $f^j: \hat J \to \hat C_n$ be the
    monotone extension of $f^j|_J$.  By the above, $\hat C_n \cap C_{n-1}$
    contains a $1$-scaled neighborhood of $C_n$ and by the Macroscopic Koebe
    Principle (see~\cite{dMvS93}*{p.~287}) it follows that $\hat J \cap I$
    contains a $\tau$-scaled neighborhood
    of $J$, for some $\tau > 0$ not depending on $f$.
    This proves the claim and hence $f$ has no wandering intervals.
    That $f$ is ergodic follows from \cite{MW14}*{Theorem~3.12}.

    We claim that the $\omega$-limit set of $c$, $\omega(c)$, is a Cantor
    attractor for $f$ and hence that the attractor is minimal.
    Since $f$ is not defined at $c$ let us emphasize that $\omega(c)$ is
    defined as the union of the $\omega$-limit sets of the two critical values
    of $f$.  The claim follows from the fact that $\abs{D_n} = 1$ (see
    \cite{MW14}*{Proposition~3.7}) and
    $\abs{C_n} \leq \abs{C_{n-1}}/3$ (since $C_{n-1}$ contains a $1$-scaled
    neighborhood of $C_n$).  That is, almost all points pass arbitrarily close
    to $c$ after sufficiently many iterates.

    To see that the Cantor attractor has zero measure note that we can cover it
    by branches of $T_n$.  Let $\Lambda_n$ be the smallest such cover.  Then
    $\abs{\Lambda_n} \leq (1+2\tau)^{-1}\abs{\Lambda_{n-1}}$ since each branch
    $J$ of $T_n$ is contained in a $\tau$-scaled neighborhood of a branch $I$
    of $T_{n-1}$ by the above.  Hence
    $\abs{\omega(c)} \leq \lim \abs{\Lambda_n} = 0$.

    Unique ergodicity can be proved with techniques from
    \cite{GM06}, see also \cite{MW16}.

    We now prove the last statement.
    If $\{\renorm^n f\}$ degenerate then the longest interval
    of $\Lambda_n$ shrinks at a faster than exponential rate in $n$, by
    the Degeneration Theorem and Proposition~\ref{distortion-invariance}.
    Since $f$ is of bounded type, the number of intervals in
    $\Lambda_n$ is at most exponential in~$n$.  It follows that
    $\HD(\Lambda) = 0$ (see e.g.\ \cite{F03}*{Proposition~4.1}).

    If $\{\renorm^n f\}$ do not degenerate, then $f$ has
    a priori bounds by Theorem~\ref{dichotomy}.  It follows from standard
    arguments that $\HD(\Lambda) \in (0,1)$
    (see e.g.\ \cite{dMvS93}*{Theorem~VI.2.1}).
\end{proof}

\section{Existence of fixed points}
\label{existence-of-fixed-points}

This whole section is dedicated to the proof of the following theorem.

\begin{theorem}[Existence of fixed points] \label{fixed-points}
    For every $\beta \in \posrationals$ there exists $N < \infty$
    such that $\renorm$ has an $(a,b)$--renormalizable fixed point, for all
    $b \geq N$ and $a/b=\beta$.
\end{theorem}

\begin{remark}
    The combinatorics here are different from those of
    \cite{MW14}*{Theorem~6.1}.
    This causes drastic changes in the dynamics: here we see three unstable
    directions (see Theorem~\ref{3d-mfd}), whereas in \cite{MW14} there appear
    only to be two.
    The extra unstable direction complicates the arguments, but on the other
    hand here we are able to prove the existence of fixed points for infinitely
    many different types of combinatorics.
\end{remark}

Fix $\beta \in \posrationals$ and assume $a/b = \beta$.
We are going to define a subset of the domain of $(a,b)$--renormalization
and then prove that this subset contains a fixed point.
The reason why we only consider a subset is that we cannot control the critical
point and critical values of the renormalization on the whole domain.
For any $b \geq 1$, closed interval $\Delta \subset (0,1)$,
$\gamma \in (0,\tfrac12)$ and $\delta \in (0,\ubdist)$, define
$Q = [1-\gamma,1]^2$ and
\begin{equation*}
    \D =
    \{ \text{$(a,b)$--renormalizable $f \in \lorenz_\delta$} \mid
    c \in \Delta,\; (u',v') \in Q,\; a/b = \beta \}.
\end{equation*}
Here $f$ is identified with $(c,u,v,\phi,\psi)$ and $\renorm f$ is
identified with $(c',u',v',\phi',\psi')$.

The next lemma gives very explicit control over how the relative critical point
moves under renormalization.
In particular, the relative critical point of the renormalization
is a function of the relative critical point, up to a
multiplicative factor close to $1$ which we can control.

\begin{lemma} \label{relc}
    Let $\relc = c/(1-c)$ and $\relc' = c'/(1-c')$ denote the relative critical
    point of $f$ and $\renorm f$, respectively, and let
    $h(\relc)^\alpha = (\alpha (1+\relc))^{1-\beta}/\relc$.
    There exists a continuously differentiable map
    $g:\posreals\to\posreals$ such that if $f \in \D$, then
    \begin{equation*}
        \relc' = \relc g(\relc) h(\relc)^b
        \kappa \cdot
        \left( 1 + O(b (\delta + (1-u) + (1-v)) + \delta' + \alpha^{-b})
            \right),
    \end{equation*}
    where
    $1 - \gamma \leq \kappa^\alpha \leq (1 - \gamma)^{-1}$
    and $\delta' = \max\{\norm{\phi'},\norm{\psi'}\}$.
\end{lemma}

\begin{proof}
    The proof is about controlling the factors appearing in the formula which
    relates $\relc'$ to $\relc$ in Lemma~\ref{crit-position}.
    First we control the factor $\relv(\renorm f) / \relv(f)$, then we
    control the factor $\deriv f^b(x) / \deriv f^a(y)$.

    Note that by Lemma~\ref{nonlinearity-prop}
    \begin{equation*}
        \relv(f)
        =
        \frac{\abs*{\phi([0,u])}}{\abs*{\psi([1-v,1])}}
        =
        \frac uv \exp\{O(\delta)\}
        =
        1 + O(\delta + (1-u) + (1-v)).
    \end{equation*}
    The same argument shows that $\relv(\renorm f) = u'/v' (1 + O(\delta'))$.
    Since $u',v'\in[1-\gamma,1]$
    $\relv(\renorm f) = \kappa^\alpha \cdot (1 + O(\delta'))$,
    where $1 - \gamma \leq \kappa^\alpha \leq (1 - \gamma)^{-1}$.
    We have proved that
    \begin{equation} \label{relc-relv-term}
        \frac{\relv(\renorm f)}{\relv(f)}
        =
        \kappa^\alpha \cdot
        \left( 1 + O(\delta + (1 - u) + (1 - v) + \delta') \right).
    \end{equation}

    Next we show how to control
    $\deriv f_-^b(x)/\deriv f_+^a(y)$ by comparing with
    a full map of the standard family $\hat q$ (i.e.\ for which $u=1=v$):
    \begin{equation} \label{three-terms}
        \frac{\deriv f^b(x)}{\deriv\hat q^b(0)}
        =
        \frac{\deriv f^b(x)}{\deriv f^b(f_-^{-b} c)}
        \cdot
        \frac{\deriv f^b(f_-^{-b} c)}{\deriv \hat q^b(\hat q_-^{-b} c)}
        \cdot
        \frac{\deriv \hat q^b(\hat q_-^{-b} c)}{\deriv\hat q^b(0)}.
    \end{equation}
    The first factor is $\exp\{O(\delta' + \delta)\}$, since the diffeomorphic
    part of the renormalization is given by
    $\psi' = \zoom{\Psi}{\inv\Psi(C)}$,
    where $\Psi = f^b\circ\psi$, and $x \in f_-^{-b}(C)$.
    Consider the second factor.
    Let $c_k = f_-^{-k}(c)$, $\hat c_k = \hat q_-^{-k}(c)$ and use
    \eqref{q}
    to get
    \begin{equation*}
        \frac{\deriv f^b(f_-^{-b} c)}{\deriv \hat q^b(\hat q_-^{-b} c)}
        =
        \prod_{k=1}^b \deriv \phi(q c_k)
            \frac{\deriv q(c_k)}{\deriv \hat q(\hat c_k)}
        =
        \prod_{k=1}^b \deriv \phi(q c_k) u
            \left(1 + \frac{\hat c_k - c_k}{c - \hat c_k}\right)^{\alpha-1}
            \mkern-28mu.
    \end{equation*}
    Note that $\abs{\deriv\phi}^b = 1 + O(\delta b)$ and
    $u^b = 1 + O(b(1-u))$ so we only need to control how fast
    $l_k = \hat c_k - c_k$ shrinks.
    Use $l_{k+1} = \abs*{\hat q_-^{-1}(\hat c_k) - q_-^{-1}(\phi^{-1} c_k)}$
    to estimate
    \begin{align*}
        l_{k+1}
        &\leq
        \abs*{\hat q_-^{-1}(\hat c_k) - \hat q_-^{-1}(c_k)}
        +
        \abs*{\hat q_-^{-1}(c_k) - q_-^{-1}(c_k)}
        +
        \abs*{q_-^{-1}(c_k) - q_-^{-1}(\phi^{-1} c_k)}
        \\
        &\leq
        \deriv \hat q_-^{-1}(t_k) l_k
        +
        K c_k (1 - u)
        +
        \deriv q_-^{-1}(s_k) \abs*{c_k - \phi^{-1}(c_k)},
    \end{align*}
    for some $t_k \in [c_k,\hat c_k]$ and $s_k \in [\phi^{-1}(c_k),c_k]$.
    Because of the Expansion Lemma, $t_k,s_k \leq K \lambda^k$ for some
    $\lambda \in (0,1)$.
    Let $D_k = \deriv \hat q_-^{-1}(t_k)$, then
    $D_k \to \deriv \hat q(0)^{-1} \leq \alpha^{-1}$
    exponentially fast and
    $\deriv q_-^{-1}(s_k) \abs*{c_k - \phi^{-1}(c_k)} \leq K \delta c_k$.
    We have the relation
    \begin{equation*}
        l_{k+1} \leq D_k l_k + K\lambda^{k+1} (\delta + 1-u).
    \end{equation*}
    Choose $\lambda \in (\alpha^{-1},1)$ so that
    $\limsup D_j/\lambda < 1$.
    An induction argument shows that $l_k \leq K \lambda^k (1 - u + \delta)$.
    Thus
    \begin{equation*}
        \sum \log\left( 1 + \frac{l_k}{c-\hat c_k} \right)
        \leq
        K (\delta + 1 - u).
    \end{equation*}
    We have shown that the second factor of \eqref{three-terms} is
    $1 + O(b\delta + b(1-u))$.

    Consider the third factor of \eqref{three-terms};
    call it $G_b^-(c)$.
    We claim that $G_n^-|_{[\eps,1-\eps]}$ converges in
    $\Ck1$ as $n\to\infty$, for every $\eps > 0$.
    Let us prove the claim by showing that it holds for $\log G_n^-$.
    As in the above we can appeal to the Expansion Lemma to see that
    $\forall \eps > 0$ there exists $m_\eps < \infty$ and $K_\eps < \infty$
    such that
    \begin{equation} \label{relc-chat-eps}
        \hat c_n \leq \min\left\{ \eps,\, K_\eps \alpha^{-n} \right\},
        \quad
        \forall n \geq m_\eps,\;
        \forall c \in [\eps, 1 - \eps].
    \end{equation}
    Here we used that the contraction rate of $\hat c_n \mapsto \hat c_{n+1}$
    is
    $\deriv \hat q_-^{-1}(0) = c/\alpha \leq 1/\alpha$.
    We will need the relation
    \begin{equation} \label{relc-chat}
        \left( 1 - \frac{\hat c_{n+1}}{c} \right)^\alpha
        =
        1 - \hat c_n,
    \end{equation}
    which is a consequence of \eqref{q}.
    By definition,
    $G_n^-(c) = \deriv \hat q^n(\hat c_n)/\deriv\hat q^n(0)$, which
    combined with \eqref{relc-chat} and \eqref{q}
    gives
    \begin{equation} \label{relc-G}
        \log G_n^-(c)
        =
        \frac{\alpha - 1}{\alpha}
        \sum_{k=0}^{n-1} \log \left( 1 - \hat c_k \right).
    \end{equation}
    This sequence is uniformly convergent and hence
    $G^- = \lim G_n^- \in \Ck0$, since
    \begin{equation*}
        \sum_{k\geq m_\eps} \abs*{\log \left( 1 - \hat c_k \right)}
        \leq
        \sum_{k\geq m_\eps} \frac{\hat c_k}{1 - \hat c_k}
        \leq
        \frac{K_\eps}{(1 - \eps)} \sum_{k\geq m_\eps} \alpha^{-k}
        =
        \frac{K_\eps \alpha^{-(m_\eps - 1)}}{(\alpha - 1) (1 - \eps)}.
    \end{equation*}
    The tail can be written as $\log\{G^-(c)/G_n^-(c)\}$, so this
    also shows that
    \begin{equation*}
        \frac{G_b^-(c)}{G^-(c)}
        =
        1 + O(\alpha^{-b}),
        \quad \forall c \in \Delta.
    \end{equation*}

    Before proving that $G^-$ is $\Ck1$ let us get back to \eqref{three-terms},
    which by the above is now
    \begin{equation*}
        \frac{\deriv f^b(x)}{\deriv\hat q^b(0)}
        =
        (1 + O(\delta + \delta'))
            (1 + O(b\delta + b(1-u)))
            G^-(c)
            (1 + O(\alpha^{-b})).
    \end{equation*}
    Now do an analogous estimate for $\deriv f^a(y)$, using that $a/b = \beta$,
    to get
    \begin{equation} \label{relc-deriv-term}
        \frac{\deriv f^b(x)}{\deriv f^a(y)}
        =
        \frac{\deriv \hat q^b(0)}{\deriv \hat q^a(1)}
        \frac{G^-(c)}{G^+(c)}
        \left( 1 + O(b (\delta + (1-u) + (1-v)) + \delta' + \alpha^{-b})
            \right),
    \end{equation}
    where $G^+$ is defined analogously to $G^-$.
    Define $G(c) = G^-(c) / G^+(c)$ and
    \begin{equation*}
        H(c)
        =
        \left( \frac{\deriv \hat q^b(0)}{\deriv \hat q^a(1)} \right)^{1/b}
        =
        \alpha^{1-\beta} \frac{(1-c)^\beta}{c}.
    \end{equation*}
    Let $\tau(c) = c/(1-c)$ be the diffeomorphism which sends an absolute
    critical point to its relative critical point.
    Define $g(\relc)^\alpha = G \circ \tau^{-1}(\relc)$ and
    $h(\relc)^\alpha = H \circ \tau^{-1}(\relc)$.
    To finish the proof of the main statement
    apply Lemma~\ref{crit-position}, using equations \eqref{relc-relv-term}
    and~\eqref{relc-deriv-term} together with the definitions of $h$
    and~$g$.

    It remains to show that $g:\posreals\to\posreals$ is $\Ck1$, which we prove
    by showing that $\deriv\log G^-_n$
    converges uniformly to $\deriv\log G^-$ on $[\eps,1-\eps]$,
    $\forall\eps > 0$.
    Differentiating \eqref{relc-G} gives (nb.\ were write $\partial_c$ instead
    of $\tfrac\partial{\partial c}$)
    \begin{equation} \label{relc-DG}
        \deriv \log G_n^-(c)
        =
        -\frac{\alpha - 1}{\alpha}
        \sum_{k=0}^{n-1} (1 - \hat c_k)^{-1}
            \partial_c \hat c_k.
    \end{equation}
    Solving \eqref{relc-chat} for $\hat c_{n+1}$ and differentiating gives
    \begin{equation*}
        \partial_c \hat c_{n+1}
        =
        \partial_c \hat c_n \frac c\alpha (1 - \hat c_n)^{-1+1/\alpha}
            + 1 - (1 - \hat c_n)^{1/\alpha}.
    \end{equation*}
    By \eqref{relc-chat-eps}, for $n+1 \geq m_\eps$ and
    $\eps \leq c \leq 1-\eps$ we get
    \begin{equation*}
        \abs*{\partial_c \hat c_{n+1}}
        \leq
        \abs*{\partial_c \hat c_n} \frac 1\alpha
            (1 - \eps)^{1/\alpha}
            + \frac 1\alpha (1 - \eps)^{-1+1/\alpha} K_\eps \alpha^{-n}.
    \end{equation*}
    Here we used that $\varphi(t) = 1 - (1 - t)^{1/\alpha}$ is convex on $(0,1)$
    and $\varphi(0) = 0$, so $\varphi(t) \leq \deriv\varphi(t)t$.
    An induction argument gives
    \begin{equation*}
        \abs*{\partial_c \hat c_{m_\eps+k}}
        \leq
        \abs*{\partial_c \hat c_{m_\eps}} \alpha^{-k}
        + K'_\eps \alpha^{-(m_\eps+k)},
    \end{equation*}
    so the tail of \eqref{relc-DG} converges absolutely.
    Hence $\deriv\log G_n^-$ converges uniformly.
\end{proof}

From the previous lemma we will be able to deduce that each $\beta = a / b$
uniquely defines the asymptotic critical point of an $(a,b)$--renormalization
fixed point (as $b\to\infty$ with $a/b = \beta$).
Call this asymptotic critical point $c_\infty(\beta)$;
it is defined as the unique solution to
\begin{equation*}
    \frac x\alpha = \left( \frac{1-x}\alpha \right)^\beta
    \mkern-10mu.
\end{equation*}
We will write $c_\infty$ instead of $c_\infty(\beta)$ when there is no need to
emphasize the dependence on $\beta$.
Let $\relc_\infty = c_\infty / (1 - c_\infty)$ be the relative asymptotic
critical point.
Note that $c_\infty$ is defined by the relative asymptotic critical point
satisfying $h(\relc_\infty) = 1$, where $h$ is defined
in the previous lemma.

The next lemma serves to define the subset of the domain of renormalization
where we can control the critical point and the critical values of the
renormalization.
The main point is that within this subset the critical point of the
renormalization expands away from $c_\infty$ while staying bounded, and that the
critical values are close to $1$.
It also shows that any renormalization fixed point must have critical point
close to $c_\infty$.
Some of the conclusions of the lemma will only be needed later when we prove
the existence of unstable manifolds.

\begin{lemma} \label{Delta-b}
    For every $\eps > 0$, $l$ and $r$ satisfying
    $0 < l\pm\eps < c_\infty < r\pm\eps < 1$,
    there exist $N < \infty$,
    $0 < \gamma < \min\{l-\eps,1-(r+\eps)\}$,
    $K < \infty$ and
    closed intervals $\{\Delta_n\subset(l+\eps,r-\eps)\}_{n\geq N}$ such that
    if $f \in \Dl(b,\Delta_b,\gamma,\delta_b=1/b^2)$ and $b\geq N$, then
    the following holds:
    \begin{enumerate}
        \item $c(\renorm f) \in (l-\eps, r+\eps)$,
            \label{Delta-b-1}
        \item if $c(f) \in \partial\Delta_b$ then
            $c(\renorm f) \not\in [l+\eps,r-\eps]$,
            \label{Delta-b-2}
        \item if $c(f) = c(\renorm f)$ then
            $\dist(c(\renorm f),\partial\Delta_b) \geq \inv K\abs{\Delta_b}$,
            \label{Delta-b-3}
        \item $K^{-1} \leq b \abs{\Delta_b} \leq K$,
            \label{Delta-b-4}
        \item $u, v > 1 - \gamma$ and
            $\renorm f \in \lorenz_{\delta'_b}$ where $\delta'_b < \delta_b$,
            \label{Delta-b-5}
        \item $\renorm f$ does not have a trivial branch.
            \label{Delta-b-6}
    \end{enumerate}
\end{lemma}

\begin{proof}
    Let $\tau: (0, 1) \to \posreals$ be the diffeomorphism taking a critical
    point to its relative critical point, $\tau(x) = x / (1 - x)$.
    We will work in $\posreals$ and then transfer back to $(0,1)$ via
    $\inv\tau$.
    Let $\rho(x) = x g(x) h(x)^b$, where $g$ and $h$ are defined in
    Lemma~\ref{relc}.
    Note that $\rho$ describes how the relative critical point moves under
    renormalization, up to a multiplicative factor which we need to control.

    We will now construct $\Delta_b$.
    Let $J = \tau([l,r])$ and let $J_\eps = \tau((l+\eps,r-\eps))$.
    Define $\tilde\Delta_b = \inv\rho(J) \cap J_\eps$ and
    $\Delta_b = \inv\tau(\tilde\Delta_b)$ so that
    $\Delta_b \subset (l+\eps,r-\eps)$.
    Note that $\rho$ maps the open interval $J_\eps$ over $J$ for $b$
    sufficiently large, since $h$ is strictly decreasing, $h(\relc_\infty) = 1$
    and $\relc_\infty \in J_\eps$.
    Hence we may choose $N < \infty$ such that
    $\rho$ maps $\tilde\Delta_b$ onto $J$ for $b \geq N$.
    By Lemma~\ref{relc} $\rho$ is differentiable, $g>0$ and $h>0$, so
    \begin{equation*}
        \deriv \log \rho(x)
        =
        b \cdot \deriv \log h(x) + \deriv \log g(x) + 1 / x.
    \end{equation*}
    Since $\deriv h < 0$, $J$ is compact, and $g, h \in \Ck1$, this shows that
    $-K b \leq \deriv \log \rho|_J \leq -b / K$ for $b$ sufficiently large.
    Since $\rho > 0$ is bounded on $\inv\rho(J)$, this in turn shows that
    $\exists N < \infty$ such that
    \begin{equation} \label{Delta-b-Drho}
        \deriv \rho|_J < 0
        \quad\text{and}\quad
        b/K \leq \abs*{\deriv \rho|_{\inv\rho(J)\cap J}} \leq Kb,
        \quad\forall b \geq N.
    \end{equation}
    In particular, $\rho$ is injective on $J$ so
    $\rho:\tilde\Delta_b \to J$ is a diffeomorphism and consequently
    $\tilde\Delta_b$ and $\Delta_b$ are compact intervals.
    Property~\ref{Delta-b-4} is a direct consequence of \eqref{Delta-b-Drho},
    since $\deriv\inv\tau$ has bounded distortion on compact intervals.

    We will now prove the remaining properties.
    Consider the multiplicative factor of Lemma~\ref{relc}.
    Note that $1-u$, $1-v$ and $\delta'$ are exponentially small in $b$
    since $c$ is bounded (and $\delta_b < \ubdist$);
    the Expansion Lemma gives the statement about $u$ and $v$ (see the remark
    after the lemma), whereas
    Proposition~\ref{distortion-invariance}
    gives the statement about $\delta'$.
    This proves property~\ref{Delta-b-5}, and together with Lemma~\ref{relc}
    shows that we may choose $N < \infty$ large and $\gamma > 0$ small enough
    so that for $b \geq N$:
    \begin{equation} \label{Delta-b-newc}
        1 - \eps
        \leq \frac{\relc'}{\rho(\relc)} \leq
        1 + \eps.
    \end{equation}
    It is clear that we can ensure $\gamma < \min\{l-\eps,1-(r+\eps)\}$.
    A calculation shows that
    \begin{equation} \label{Delta-b-eps}
        1 - \eps < \frac{\tau(x)}{\tau(y)} < 1+\eps
        \implies
        \abs{x - y} < \eps.
    \end{equation}
    Since $\rho:\tilde\Delta_b \to J$ is a diffeomorphism
    \eqref{Delta-b-newc} and \eqref{Delta-b-eps} imply properties
    \ref{Delta-b-1} and~\ref{Delta-b-2}.

    To prove property \ref{Delta-b-3} note first that $\rho|_J$ has a unique
    fixed point $\tau(p_b) \in \tilde\Delta_b$.
    Since $\tau(p_b) \in J_\eps$, it sits at a bounded distance away from
    $\partial J$, so $\dist(\tau(p_b),\partial J) \geq \inv K\abs{J}$.
    By \eqref{Delta-b-Drho} $\rho|_{\tilde\Delta_b}$ has bounded distortion
    (and so does $\tau|_{\Delta_b}$), so it follows that
    $\dist(p_b,\partial \Delta_b) \geq \inv K \abs{\Delta_b}$.
    It remains to prove that any critical point which is fixed under
    renormalization must be close to $p_b$ (relative to $\abs{\Delta_b}$).
    By equations \eqref{Delta-b-newc} and~\eqref{Delta-b-eps}
    $\abs{c' - c} < \eps$ for all $c \in \Delta_b$.
    In particular, any fixed point $c = c'$ must satisfy
    $\abs{c - p_b} \leq K \eps / \deriv (\inv\tau\circ\rho\circ\tau)(s)$,
    for some $s \in \Delta_b$.
    By the mean-value theorem
    $\deriv(\inv\tau\circ\rho\circ\tau)(t) = (r-l) / \abs{\Delta_b}$ for some
    $t \in \Delta_b$.
    Since $\inv\tau\circ\rho\circ\tau$ has bounded distortion on $\Delta_b$,
    $\abs{c - p_b} \leq K \eps \abs{\Delta_b} / (r-l)$.
    From this and the discussion before \eqref{Delta-b-newc} it follows that we
    may choose $N$ and $\gamma$ so that $\abs{c - p_b}/\abs{\Delta_b}$ is as
    small as we wish (independently of $b$) so property~\ref{Delta-b-3}
    follows.

    Finally, property~\ref{Delta-b-6} holds since $\gamma < c' < 1-\gamma$
    whereas $u',v' \geq 1-\gamma$.
    Explicitly, if the left branch of $\renorm f$ is trivial, then
    $u' = \phi'(c')$, but this is impossible since $\phi'(c') \to c'$ as
    $b \to \infty$.
    A similar argument holds for the right branch of $\renorm f$.
\end{proof}

For the remainder of this section we assume that $\Delta=[l,r]$, $\eps$, $N$
and $\gamma$ have been chosen such that the previous lemma holds and assume
that $b\geq N$.
Define
\begin{equation*}
    \Dl_b = \Dl(b,\Delta_b,\gamma,\delta_b).
\end{equation*}
The next lemma describes the topology of $\Dl_b$.
It relies on Lemma~\ref{DR2D}, which is at the end of this section as it
is independent of the presentation of the set $\Dl_b$.

Let $\pi_{uv}$ be the projection onto the $(u,v)$--coordinates and let
$\family_{\nu}$ denote the two-dimensional family
$(u,v) \mapsto (u,v,c,\phi,\psi)$, where $\nu = (c,\phi,\psi)$.
\begin{lemma} \label{Db}
    $\Dl_b$ is connected and
    $\pi_{uv}\circ\renorm$ maps $\Dl_b \cap \family_{\nu_0}$ diffeomorphically
    onto $Q$, for every family $\family_{\nu_0}$ intersecting $\Dl_b$.
    In particular, the intersection of $\Dl_b$ with the standard family is
    homeomorphic to $Q\times\Delta_b$.
\end{lemma}

\begin{proof}
    Assume that $\family_{\nu_0}$ intersects $\Dl_b$ and
    let $\xi_{\nu_0} = \pi_{uv}\circ\renorm|_{\family_{\nu_0}}$.
    Note that $\xi_{\nu_0}$ is a local diffeomorphism since
    $\det\deriv\xi_{\nu_0}(u,v) > 0$ by Lemma~\ref{DR2D}.
    We will show that $\xi_{\nu_0}: \inv\xi_{\nu_0}(Q) \to Q$ is a
    diffeomorphism.

    We begin by proving that $\xi_{\nu_0}:J\to Q$ is a diffeomorphism for
    any connected component $J \subset \inv\xi_{\nu_0}(Q)$.
    The boundary of $J$ is either a preimage of $\partial Q$, or a boundary
    point of $\Dl_b$.
    A point in $\partial\Dl_b$ which is not in
    $\inv\xi_{\nu_0}(\partial Q)$ must have a renormalization with a trivial
    branch.
    Hence Lemma~\ref{Delta-b}\eqref{Delta-b-6} implies that
    $\partial J \subset \inv\xi_{\nu_0}(\partial Q)$
    and from this we get that if $y \in \Int Q$ has a preimage $x \in J$,
    then $x \in \Int J$.
    But $\xi_{\nu_0}$ is a local diffeomorphism, so a neighborhood of~$x$ is
    mapped
    to a neighborhood of~$y$.
    Thus $\xi_{\nu_0}(J)$ is open (in the subspace topology on~$Q$).
    But $\xi_{\nu_0}(J)$ is also closed, since it is compact by the fact that
    $\inv\xi_{\nu_0}(Q)$ is closed.
    Thus $\xi_{\nu_0}(J) = Q$.
    Since $\xi_{\nu_0}:J \to Q$ is proper, $J$ is path-connected and $Q$ is
    simply connected, it follows that $\xi_{\nu_0}$ is also injective
    (see~\cite{H75}).

    It remains to show that $\inv\xi_{\nu_0}(Q)$ is connected.
    This will be done in two steps:
    \begin{enumerate*}
        \item showing that if $\inv\xi_{\nu_0}(Q)$ is not connected,
            then the two-dimensional standard family $\family_{(c,\id,\id)}$
            has more than one full vertex (i.e.\ a map whose renormalization
            has both branches full), and
        \item showing that any two-dimensional standard family has a unique
            full vertex.
    \end{enumerate*}

    If $\inv\xi_{\nu_0}(Q)$ is not connected,
    then there exist full vertices $p_1(\nu_0) \neq p_2(\nu_0)$ in
    $\family_{\nu_0}$,
    since every connected component of $\inv\xi_{\nu_0}$ maps onto $Q$.
    That is, $\xi_{\nu_0}(p_i(\nu_0)) = (1,1)$, for $i=1,2$.
    Both of these full vertex lies at the intersection of two smooth curves
    $\hat\chi_i^\pm \subset \family_{\nu_0}$ defined by the left
    (resp.\ right) branch of the renormalization having a full branch.
    The intersection is transversal, since $\xi_{\nu_0}$ maps $\hat\chi_i^\pm$
    diffeomorphically into $\{u=1\}$ and $\{v=1\}$, respectively.
    Hence we can perturb $\nu$ away from $\nu_0$ and get that $p_i(\nu_0)$
    persists in a neighborhood $V$ of $\nu_0$; i.e.\
    $\xi_\nu(p_i(\nu)) = (1,1)$, $\forall \nu \in V$.
    Since $\xi_\nu$ is a local diffeomorphism this transversality
    argument shows that $p_i(\nu)$ is well-defined for all $\nu$ such that
    $\family_\nu \cap \Dl_b \neq \emptyset$.
    Furthermore, the facts that $p_1(\nu_0) \neq p_2(\nu_0)$ and that $\xi_\nu$
    is a local diffeomorphism implies that $p_1(\nu) \neq p_2(\nu)$.
    In particular, the two-dimensional standard family $\family_{\nu_1}$,
    $\nu_1 = (c,\id,\id)$, has two full vertices (for every~$c$).
    This completes the first step.

    Note that if there was a unique full vertex then the above argument implies
    that $\Dl_b$ is connected.
    Hence, connectedness follows from step two which we now prove.

    In light of \cite{MdM01}*{Prop.~6.1} it suffices to show that there
    is a unique trivial vertex (i.e.\ whose renormalization has
    trivial branches) and that any connected component of the domain of
    $(a,b)$--renormalization in $\family_{\nu_1}$ has at most one full vertex.

    We will now show that there is a unique trivial vertex in
    $\family_{\nu_1}$.
    The maps $q$ with branches $q_\pm$ in $\family_{\nu_1}$ are given
    by~\eqref{q}.
    Note that $q_-$ depends on $u$ and $q_+$ depends on $v$.
    A trivial vertex lies at every intersection of the two curves
    \begin{equation} \label{Db-trivial-curves}
        \hat\sigma_l = \{ (u,v) \mid q_+^{-a}(c) = u \}
        \quad\text{and}\quad
        \hat\sigma_r = \{ (u,v) \mid q_-^{-b}(c) = 1-v \},
    \end{equation}
    inside $(c, 1) \times (1 - c, 1)$.
    Note, if $(u,v) \in \hat\sigma_l$ (resp.\ $(u,v) \in \hat\sigma_r$),
    then $\renorm q$ has a trivial left (resp.\ right) branch.
    For every $v \in [1-c,1]$ there is a unique $\sigma_l(v) \in [c,1]$ such
    that $(\sigma_l(v),v) \in \hat\sigma_l$.
    Hence $\hat\sigma_l$ is the graph of a function
    $\sigma_l:[1-c,1] \to [c,1]$ of~$v$.
    Similarly, $\hat\sigma_r$ is the graph of some function
    $\sigma_r:[c,1]\to[1-c,1]$ of~$u$.
    We claim that $\sigma_l$ and $\sigma_r$ are concave (the proof is below).
    From \eqref{Db-trivial-curves} it follows that $\sigma_l(1-c) = c$ and
    $\sigma_r(c) = 1-c$, and from \eqref{q}
    it follows that that the
    derivative of $\sigma_l$ (resp.\ $\sigma_r$) is unbounded at $1-c$
    (resp.~$c$).
    Hence $(\sigma_l(v),v) = (u,\sigma_r(u))$ has a unique solution in
    $(c,1)\times(1-c,1)$.

    In order to reach a contradiction, assume that some connected component $I$
    of the domain of $(a,b)$--renormalization has two full vertices $p_i$.
    As was noted above each $p_i$ lies at the transversal intersection of two
    curves $\hat\chi_i^\pm$.
    By \cite{MdM01}*{Prop.~6.1} $\hat\chi_i^\pm$ are graphs over the
    diagonal $\{u=v\}$ locally around $p_i$ and we may assume that
    $\hat\chi_i^+$
    lies over $\hat\chi_i^-$ (as graphs over the diagonal) at both $p_1$
    and~$p_2$.
    By assigning an orientation to $\partial I$ we see that this implies that
    $\hat\chi_1^+$ (resp.\ $\hat\chi_1^-$) goes toward (resp.\ away from) $p_1$
    and the opposite at $p_2$ (or vice versa).
    Hence $p_1$ and $p_2$ have neighborhoods which are mapped with different
    orientation to neighborhoods of $(1,1)$ by $\xi_{\nu_1}$.
    This contradicts the fact that $\deriv\xi_{\nu_1}$ preserves orientation.

    We will now prove that $\sigma_r$ is concave (the proof for $\sigma_l$
    is similar).
    By \eqref{Db-trivial-curves} $\sigma_r(u) = 1 - q_-^{-b}(c)$.
    Let $c_n = q_-^{-n}(c)$ and
    $g_n(u) = -\partial_u c_n = \deriv \sigma_r(u)$.
    We will show that $\deriv g_n < 0$, $\forall n > 1$.
    Write $g_n(u) = -\partial_u(q_-^{-1}(c_{n-1}))$ and differentiate to get
    \begin{equation*}
        \deriv q_-(c_n) g_n(u) = g_{n-1}(u) + \hat q(c_n),
    \end{equation*}
    where $\hat q_-$ is the full branch (i.e.\ $q_-$ with $u=1$).
    Apply $\partial_u$ to both sides to get
    \begin{equation*}
        q_-(c_n) \deriv g_n(u)
        =
        \deriv g_{n-1}(u) - g_n(u)
            \left(2 \deriv \hat q_-(c_n) - \deriv^2 q_-(c_n) g_n(u)\right).
    \end{equation*}
    From $g_0 = 0$ we get $g_1 > 0$ and $\deriv g_1 < 0$, so
    $g_n > 0$ and $\deriv g_n < 0$ by induction, $\forall n\geq 1$
    (use $q_- > 0$, $\deriv q_- > 0$, and $\deriv^2 q_- < 0$).

    Finally, the map $(u,v,c) \mapsto (\pi_{uv}\circ\renorm(u,v,c,\id,\id), c)$
    is injective and continuous, hence it is a homeomorphism.
\end{proof}

We claim that $\Dl_b$ contains a fixed point of $\renorm$.
To prove this we will reduce the fixed point problem from an
infinite-dimensional space
to a three-dimensional space using the following result:

\begin{homotopy-lemma}[\cite{MW14}*{Appendix A}]
    Let $Y$ be a normal topological space, let $X \subset Y$ be a closed
    subset and let $h: X \times [0,1] \to Y$ be a homotopy between $f$
    and~$g$.  If every extension of $g|_{\partial X}$ to $X$ has a fixed point
    and if $h(x,t) \neq x$ for all $x \in \partial X$ and $t \in [0,1]$,
    then $f$ has a fixed point.
\end{homotopy-lemma}

The homotopy we will use is
\begin{equation*}
    \pi_t(u,v,c,\phi,\psi)
    =
    (u, v, c, (1-t)\phi, (1-t)\psi).
\end{equation*}
Recall that $(1-t)\phi$ means scaling the nonlinearity of $\phi$ and
note that $\pi_1$ is the projection to the standard family.
From the Homotopy Lemma and the next two lemmas it immediately follows that
$\renorm$ has a fixed point in $\Dl_b$.

\begin{lemma}
    $\pi_t\circ\renorm f \neq f$, $\forall f \in \partial\Dl_b$,
    $\forall t\in[0,1]$.
\end{lemma}

\begin{proof}
    We will consider each boundary piece of $\Dl_b$ one at a time:
    \begin{enumerate}
        \item If $\renorm f$ has a full branch, then $u'=1$ or $v'=1$, but both
            $u<1$ and $v<1$ for a renormalizable map.
        \item If $\norm{\phi} = \delta_b$, then
            $\norm{\phi'} < \delta_b$
            by Lemma~\ref{Delta-b}(\ref{Delta-b-5}), hence
            $\norm{(1-t)\phi'} < \delta_b$ (and similarly for $\norm\psi$).
        \item If $c \in \partial\Delta_b$, then $c' \notin \Delta_b$ by
            Lemma~\ref{Delta-b}(\ref{Delta-b-2}).
        \item If $\renorm f$ has a trivial branch, then $f \notin \Dl_b$ by
            Lemma~\ref{Delta-b}(\ref{Delta-b-6}).  Hence this boundary
            piece is not part of $\partial\Dl_b$.
    \end{enumerate}
    This shows that $\pi_t\circ\renorm$ has no fixed point in $\partial\Dl_b$.
\end{proof}

\begin{lemma}
    Let $D$ be the intersection of $\Dl_b$ with the standard family and
    let $R = \pi_1\circ\renorm|_D$.
    Every extension of $R|_{\partial D}$ to $D$ has a fixed point.
\end{lemma}

\begin{proof}
    Let $\rho: \reals^3\setminus\{0\} \to \sphere^2$ be the radial projection
    $\rho(x) = x/\norm{x}$ and define the \key{displacement map} of~$R$,
    $d:\partial D \to \sphere^2$, by $d(f) = \rho(R(f) - f)$.
    Let $H_u \subset \sphere^2$ denote the open hemisphere around
    the positive $u$-axis and let $-H_u$ denote its antipodal set.
    Define $H_v$, $H_c$ and their antipodal sets similarly.
    By Lemma~\ref{Delta-b}, the $d$--image of each boundary piece of
    $\partial D$ lies in a hemisphere:
    if $u' = 1$, then $u < u' = 1$, so $d(f) \in H_u$;
    if $u' = 1-\gamma$, then $u' < u$, so $d(f) \in -H_u$,
    if $c = \partial^-\Delta_b$, then $c' > c$, so $d(f) \in H_c$;
    if $c = \partial^+\Delta_b$, then $c' < c$, so $d(f) \in -H_c$;
    and similarly, $v' = 1 \implies d(f) \in H_v$,
    $v' = 1-\gamma \implies d(f) \in -H_v$.
    We call what we have just described the \key{hemisphere property} of the
    displacement map.

    By Lemma~\ref{Db} $\partial D$ is a topological sphere so there is a
    well defined notion of the degree of $d$.
    We claim that $\deg d \neq 0$.
    Note that this claim would finish the proof, since if an extension of
    $R|_{\partial  D}$ to $D$ did not have a fixed point then its
    displacement map would extend to all of $D$ and consequently have degree
    zero.
    We now prove the claim by constructing a homotopy from $d$ to a
    homeomorphism.
    By Lemma~\ref{Db} there is a homeomorphism $h: K \to \partial D$,
    where $K = \partial[-1,1]^3$.
    We choose $h$ so that it takes faces of $K$
    onto boundary pieces of $\partial D$ in the following manner:
    the bottom face of $K$ is mapped onto $\{v'=1-\gamma\}\cap D$,
    the top face is mapped onto $\{v'=1\}\cap D$, the left face is mapped to
    $\{c=\partial^-\Delta_b\}\cap D$, and so on.
    To be clear, here we choose the right-handed coordinate system for $K$ so
    that ``right'' corresponds to increasing $c$ and ``up'' corresponds to
    increasing $v$.
    Define $i_K: K\to K$ by $i_K(x,y,z) = (-x,y,z)$.
    This is an orientation reversing homeomorphism, so $\deg i_K = -1$.
    Define $i: \partial D \to \sphere^2$ by
    $i = \rho \circ i_K \circ \inv h$.
    For each (closed) boundary piece $F$ of $\partial D$ define a homotopy
    from $d$ to $i$ by interpolating along geodesics.
    This is a well-defined continuous operation since $d(F)$ and $i(F)$ lie
    in the same open hemisphere because of the hemisphere property and because
    of how we defined $i$.
    In fact, this defines a homotopy between $d$ and $i$ all of $\partial D$
    since each piece $F$ is closed and their union is all of $\partial D$.
    Thus $\deg d = \deg i \neq 0$.
\end{proof}

This concludes the proof of Theorem~\ref{fixed-points}.
We end with the fundamental lemma which was used earlier.
Here $\partial_x = \tfrac{\partial}{\partial x}$ and the remaining
notation is defined in~\eqref{Rf}.

\begin{lemma} \label{DR2D}
    For every $\delta < \ubdist$, $\eps > 0$ and compact intervals
    $\Delta \subset (0,1)$ and $P \subset \reals^+$
    there exists $N < \infty$ such that if $f \in \lorenz_\delta$ is
    $(a,b)$--renormalizable, $\min\{a,b\} \geq N$ and $a/b \in P$, then
    \begin{equation} \label{uv-partials}
        \begin{aligned}
        \abs{U} \partial_u u' &=
            1 + \frac{1-u'}{\lambda'_0 - 1} + o(\inv[n]b),
            &
        \abs{V} \partial_v u' &=
            -\frac{u'}{\lambda'_1 - 1} + o(\inv[n]b),
            \\
        \abs{U} \partial_u v' &=
            -\frac{v'}{\lambda'_0 - 1} + o(\inv[n]b),
            &
        \abs{V} \partial_v v' &=
            1 + \frac{1-v'}{\lambda'_1 - 1} + o(\inv[n]b),
        \end{aligned}
    \end{equation}
    for every $n > 0$.  Furthermore
    \begin{equation} \label{det2d}
        \abs{U}\abs{V}
        \det
        \begin{pmatrix}
            \partial_u u' & \partial_v u' \\
            \partial_u v' & \partial_v v'
        \end{pmatrix}
        >
        \frac{\alpha(\alpha-1)}{(\alpha - c')(\alpha - 1 + c')}
        - \eps.
    \end{equation}
\end{lemma}

\begin{proof}
    We will begin by proving \eqref{uv-partials} for $\partial_u u'$ and
    $\partial_v u'$.  The proofs for $\partial_u v'$ and $\partial_v v'$ are
    identical so we will not include them here.

    By \eqref{Rf} $u' = \abs{q_-(C_-)}/\abs{U}$, hence
    $\partial\abs{U} u' + \abs{U}\partial u' = \partial\abs{q_-(C_-)}$.
    Let $C = [l,r]$.
    From
    $\abs{U} = \inv\Phi(r) - \inv\Phi(l)$,
    $\abs{q_-(C_-)} = q_-(c) - q_-(l) = u - \inv\Phi(l)$ we get
    \begin{equation} \label{du'}
        \abs{U} \partial u' =
            \partial u - (1 - u') \partial(\inv\Phi l)
            - u' \partial(\inv\Phi r).
    \end{equation}
    In order to use \eqref{du'} we need to evaluate terms like
    $\partial(\inv T x)$, where $T$ is a first-entry map to $C$ and
    $y = T(x)$ is a fixed point of $f^k$ for some $k$.  By the chain rule,
    $\partial y = \partial T(x) + \deriv T(x) \partial x$, we get
    \begin{equation} \label{dinvT}
        \partial(\inv T y) =
            \frac{\partial y - \partial T(\inv T y)}{\deriv T(\inv T y)}.
    \end{equation}
    The chain rule applied to $y = f^k(y)$ gives
    $\partial y = \partial f^k(y) + \deriv f^k(y) \partial y$.  In particular,
    $l = f^{a+1}(l) = \Phi\circ q_-(l)$ and
    $r = f^{b+1}(r) = \Psi\circ q_+(r)$, so
    \begin{equation} \label{dlr}
        \partial l =
            -\frac{\partial (\Phi\circ q_-)(l)}{\lambda'_0 - 1}
            \quad\text{and}\quad
        \partial r =
            -\frac{\partial (\Psi\circ q_+)(r)}{\lambda'_1 - 1}.
    \end{equation}
    Affine conjugation preserves derivatives, so
    $\deriv f^{a+1}(l) = \lambda'_0$ and
    $\deriv f^{b+1}(r) = \lambda'_1$.

    We will now calculate $\partial_u u'$.  Equations \eqref{dinvT},
    \eqref{dlr}, \eqref{q}
    and $\partial_u\Phi = 0$ gives
    \begin{equation} \label{du_Phi_l}
        \partial_u(\inv\Phi l) =
            \frac{\partial_u l}{\deriv\Phi(\inv\Phi l)} =
            -\frac{\partial_u q_-(l)}{\lambda'_0 - 1} =
            -\frac{1}{\lambda'_0 - 1}
            + \frac{(\abs{C_-}/c)^\alpha}{\lambda'_0 - 1}
    \end{equation}
    and, using $\partial_u q_+ = 0$ as well,
    \begin{equation} \label{du_Phi_r}
        \partial_u(\inv\Phi r) =
            \frac{\partial_u r}{\deriv\Phi(\inv\Phi r)} =
            -\frac{\partial_u \Psi(\inv\Psi r)}%
                {\deriv\Phi(\inv\Phi r) (\lambda'_1 - 1)}.
    \end{equation}
    By Proposition~\ref{distortion-invariance} we can choose $N$ such that
    $\renorm f \in \lorenz_{\delta'}$ with $\delta' < \ubdist$ and hence
    there exists a $\lambda > 1$ (not depending on~$f$) such that
    $\lambda'_0, \lambda'_1 \geq \lambda$ for $N$ large enough (see the remark
    after Lemma~\ref{bounded-space}).  This, together with
    Lemma~\ref{bounded-space}, $c \in \Delta$ and $a/b \in P$ shows that the
    second term in the right-hand side of \eqref{du_Phi_l} is of the order
    $o(b^{-n})$, $\forall n>0$.  We claim that $\abs{\partial_u\Psi}$ is of the
    order $O(b)$, the proof of which is provided below.  On the other hand,
    $\abs{\deriv\Phi}$ grows exponentially in $a$ by the Expansion Lemma and
    Proposition~\ref{distortion-invariance}.  Since $a/b \in P$, this implies
    that \eqref{du_Phi_r} is of the order $o(b^{-n})$, $\forall n>0$.  This
    finishes the proof for $\partial_u u'$.

    We will now calculate $\partial_v u'$.  Equations \eqref{dinvT},
    \eqref{dlr}, \eqref{q}
    and $\partial_v q_- = 0$ gives
    \begin{align}
        \partial_v(\inv\Phi l) &=
            \frac{\partial_v l - \partial_v\Phi(\inv\Phi l)}%
                {\deriv\Phi(\inv\Phi l)} =
            -\frac{\partial_v\Phi(\inv\Phi l)(1 + (\lambda'_0-1)^{-1})}%
                {\deriv\Phi(\inv\Phi l)},
        \label{dv_Phi_l}
        \\
        \partial_v(\inv\Phi r) &=
            \frac{\partial_v r - \partial_v\Phi(\inv\Phi r)}%
                {\deriv\Phi(\inv\Phi r)} =
            -\frac{\deriv\Psi(q_+ r) \partial_v q_+(r)}%
                {\deriv\Phi(\inv\Phi r)(\lambda'_1-1)}
            -\frac{\partial_v\Phi(\inv\Phi r)}%
                {\deriv\Phi(\inv\Phi r)}.
        \label{dv_Phi_r}
    \end{align}
    We claim that $\abs{\partial_v\Phi}$ is of the order $O(a)$, the proof of
    which is postponed.  As in the
    above, $\abs{\deriv\Phi}$ grows exponentially with $b$.  Hence
    \eqref{dv_Phi_l} and the second term of \eqref{dv_Phi_r} are both of the
    order $o(b^{-n})$.  The first term of \eqref{dv_Phi_r} can be estimated as
    follows.  There exists $x \in V$, $y \in U$ such that
    $\deriv\Psi(x) = \abs{C}/\abs{V}$ and $\deriv\Phi(y) = \abs{C}/\abs{U}$.
    Let $\rho_x = \deriv\Psi(q_+ r)/\deriv\Psi(x)$ and
    $\rho_y = \deriv\Phi(\inv\Phi r) / \deriv\Phi(y)$.  Then, using
    \eqref{q},
    we get
    \begin{equation*}
        -\frac{\deriv\Psi(q_+ r) \partial_v q_+(r)}%
            {\deriv\Phi(\inv\Phi r)(\lambda'_1-1)} =
        \frac{\abs{U}}{\abs{V}} \frac{\rho_x}{\rho_y}
        \frac{1-(\abs{C_+}/(1-c))^\alpha}{\lambda'_1-1}.
    \end{equation*}
    But $\rho_x/\rho_y$ is of the order $O(e^{\delta'})$ by
    Lemma~\ref{nonlinearity-prop} and $\delta'$ is
    exponentially small in $b$.  So is $\abs{C_+}$ by Lemma~\ref{bounded-space}
    and hence
    $\partial_v(\inv\Phi r) = ((\lambda'_1-1)^{-1} + o(b^{-n})) \abs U/\abs V$.
    Now put all the above together in \eqref{du'} to finish the proof for
    $\partial_v u'$.

    We will now prove the claim that $\abs{\partial_u\Psi} \leq O(b)$;
    that $\abs{\partial_v\Phi} \leq O(a)$ follows from an identical
    argument.  An induction argument using the chain rule shows that
    \begin{equation*}
        \partial_u\Psi(x) =
            \sum_{i=0}^{b-1} \deriv f^{b-1-i}(f^{i+1}\circ\psi(x))
            \partial_u f_-(f^i\circ\psi(x)).
    \end{equation*}
    By \eqref{q},
    $\partial_u f_-(t) = \deriv\phi(q_- t) q_-(t)/u$.  By
    the mean-value theorem there exists $x_i \leq f^{b-1-i}(x)$ such that
    $\deriv f^{b-1-i}(x_i) f^{i+1}\circ\psi(x) = \Psi(x)$.  Since
    $\abs{\deriv\phi} \leq e^\delta$ and $\abs{\Psi} \leq 1$,
    \begin{equation*}
        \partial_u\Psi(x) \leq
            e^\delta \sum_{i=0}^{b-1}
            \frac{\deriv f^{b-1-i}(f^{i+1}\circ\psi(x))}{\deriv f^{b-1-i}(x_i)}
            \frac{q_-(f^i\circ\psi(x))}{f^{i+1}\circ\psi(x)}.
    \end{equation*}
    By the Expansion Lemma $0$ is an attracting fixed point of
    $\inv f_-$ with multiplier bounded by $\inv\lambda$.
    Hence each term of the sum is bounded for every $b$, so
    $\abs{\partial_u\Psi} \leq K b$.

    Let us finish by proving \eqref{det2d}.  The right-hand sides of
    \eqref{uv-partials} all have bounded modulus since
    $\lambda_0',\lambda_1' \geq \lambda > 1$ (see above).
    A calculation using this fact and the expressions \eqref{uv-partials} for
    the partial derivatives gives
    \begin{equation*}
        \abs{U}\abs{V} \det
        \begin{pmatrix}
            \partial_u u' & \partial_v u' \\
            \partial_u v' & \partial_v v'
        \end{pmatrix}
        =
        \frac{\lambda_0'\lambda_1' - v'\lambda_0' - u'\lambda_1'}%
            {(\lambda_0' - 1)(\lambda_1' - 1)} + o(\inv[n]b).
    \end{equation*}
    By Proposition~\ref{distortion-invariance} the diffeomorphic parts of the
    renormalization tend to identity maps as $N \to \infty$.  Hence,
    \eqref{q}
    implies that
    $\lambda_0' \to \alpha u'/c'$ and
    $\lambda_1' \to \alpha v'/(1 - c')$
    as $N \to \infty$.  This, together with $u',v' \leq 1$, gives
    \begin{equation*}
        \frac{\lambda_0'\lambda_1' - v'\lambda_0' - u'\lambda_1'}%
            {(\lambda_0' - 1)(\lambda_1' - 1)}
        \to
        \frac{\alpha(\alpha - 1)}%
            {(\alpha - \frac{c'}{u'})(\alpha - \frac{1 - c'}{v'})}
        \geq
        \frac{\alpha(\alpha - 1)}{(\alpha - c')(\alpha - 1 + c')}.
        \qedhere
    \end{equation*}
\end{proof}

\section{Internal structures and renormalization}
\label{the-internal-structure-of-renormalization}

A major problem with the classical renormalization operator in
\sref{preliminaries} is that it is \emph{not} differentiable \cite{FMP06}.
The solution to this problem is to avoid composition \cite{M98}, which
leads to the space of internal structures and a new renormalization operator.
This new renormalization operator is always differentiable (see
Theorem~\ref{R-diff}).

In this section we define internal structures and the renormalization operator
acting on internal structures.
An internal structure of a diffeomorphism $\phi$ is a sequence of
diffeomorphisms which when composed yields $\phi$
(\sref{internal-structure-of-diffeomorphisms}).
In the definition of the renormalization operator, composition of
diffeomorphisms is replaced by juxtaposition of internal structures;
this defines the new renormalization operator
(\sref{the-internal-structure-of-renormalization_renormalization}).
Limits of renormalization have internal structures whose diffeomorphisms are
pure.
Properties of pure internal structures are discussed in \sref{pure-maps}.
See \cite{MW14}*{\S7--8} for more details (nb.\ internal structures are also
called ``decompositions'').

\subsection{Internal structure of diffeomorphisms}
\label{internal-structure-of-diffeomorphisms}

A set $T$ is called a \key{time set} iff it is countable, has a total order and
an associated function $\depth: T \to \nats$ with finite level sets.
When emphasis is required,
we write $(T,\depth)$ and call it a time set.
Elements of $T$ are called \key{times}; those at depth 0 are called
\key{top level} times.

Let $\diff^3 \subset \diff^2$ be the subspace of $\Ck3$--diffeomorphisms with
norm $\norm{\phi} = \max\{\abs{\nonlin\phi},\abs{\deriv\nonlin\phi}\}$.
Let $\ldiff T$ be the space of all $\bar\phi: T \to \diff^3$ with the
$\ell^1$--norm $\norm{\bar\phi} = \sum\norm{\bar\phi(\tau)}$ and the linear
structure induced by $\diff^3$.
If $\bar\phi \in \ldiff T$, then the diffeomorphisms of $\bar\phi$ can be
composed in the order of $T$ to obtain a $\phi \in \diff^2$.
Explicitly, let $T_k = \{ \tau\in T\mid \depth\tau \leq k \}$ so that
$T_k = \{\tau_1,\dotsc,\tau_{n(k)}\}$ and $\tau_i < \tau_j$ if $i < j$, and let
$\phi_k = \bar\phi(\tau_{n(k)})\circ\dots\circ\bar\phi(\tau_1)$.
Then $\phi = \lim_k \phi_k$ and we write $\phi = \compose\bar\phi$.
It is essential here that the derivative of the nonlinearity is bounded.
We say that $\bar\phi$ is the \key{internal structure} of $\phi$ and we call
\begin{equation*}
    \compose:\ldiff T \to \diff^2
\end{equation*}
the \key{composition operator}.
The composition operator is Lipschitz on bounded subsets of $\ldiff T$.
See \cite{M98}*{Proposition~4.1} and \cite{MW14}*{Proposition~7.5} for proofs
of the above statements.
Any subset $S \subset T$ is also a time set, so the restriction of
$\bar\phi \in \ldiff T$ to $S$ can be composed.
Such \key{partial composition} is denoted $\compose\bar\phi|_S$.

\subsection{Renormalization}
\label{the-internal-structure-of-renormalization_renormalization}

Fix two time sets $(T_\pm,\depth_\pm)$.
Let $u,v,c \in \reals$, $\bar\phi_\pm \in \ldiff{T_\pm}$,
$\phi_\pm = \compose\bar\phi_\pm$, and assume that
$f = (u,v,c,\phi_\pm)$ is $(a,b)$--renormalizable.
Denote $\renorm f = (u',v',c',\phi'_\pm)$.
We will define time sets $(T'_\pm,\depth'_\pm)$, and
internal structures $\bar\phi'_\pm \in \ldiff{T'_\pm}$ such that
$\phi'_\pm = \compose\bar\phi_\pm'$.
The operator sending $(u,v,c,\bar\phi_\pm)$ to $(u',v',c',\bar\phi'_\pm)$
defines the
\key{renormalization operator acting on internal structures}.
Below $\square = \pm$ when $k=0$, else $\square = \mp$,
$n_- = a$ and $n_+ = b$.
\begin{enumerate}
    \item Define $T'_\pm$ to be the disjoint unions
        $(T_\mp \sqcup \{n_\pm\}) \sqcup \dots \sqcup (T_\mp \sqcup \{1\}) \sqcup T_\pm$,
        i.e.\
        \begin{equation*}
            T'_\pm
            =
            \big\{ (\tau,k) \mid \text{%
                $\tau \in T_{\Sm\square}$ if $k$ even,
                $\tau = (k+1)/2$ if $k$ odd}
                \big\}_{k=0}^{2n_\pm},
        \end{equation*}
        with order $(\tau_1, k_1) < (\tau_2, k_2)$ iff $k_1 < k_2$, or
        $k_1 = k_2$ and $\tau_1 < \tau_2$.
        Define
        \begin{equation*}
            \depth'_\pm(k,2k-1) = 0,
            \quad
            \depth'_\pm(\tau,2k) = 1 + \depth_{\Sm\square}(\tau).
        \end{equation*}
    \item Define the intervals $\{\bar C_\pm(\tau')\}_{\tau'}$ as follows:
        pull back $C$ under $f$
        \begin{equation*}
            \bar C_\pm(k, 2k - 1) = f_\mp^{-(n_\pm + 1 - k)}(C),
        \end{equation*}
        then pull back these top level intervals through the internal
        structures
        \begin{equation*}
            \bar C_\pm(\tau,2k)
            =
            \big(\!\compose\!\bar\phi_{\Sm\square}|_{\{t\geq\tau\}}
                \big)^{-1}\big(\bar C_\pm(k+1,2k+1)\big).
        \end{equation*}
        For $k=n_\pm$ we use the convention $\bar C_\pm(n_\pm+1,2n_\pm+1) = C$.
    \item Define $\bar\phi'_\pm$ by restricting to $\bar C_\pm$
        and rescaling (see \sref{preliminaries} for notation):
        \begin{align*}
            \bar\phi'_\pm(k,2k-1)
            &=
            \zoom*{q_\mp}{\bar C_\pm(k,2k-1)},
            \\
            \bar\phi'_\pm(\tau,2k)
            &=
            \zoom*{\bar\phi_{\Sm\square}(\tau)}{\bar C_\pm(\tau,2k)}.
        \end{align*}
\end{enumerate}

Note that the diffeomorphisms at the top level of $\bar\phi'_\pm$ are
rescaled restrictions of the standard family; we call them
pure maps.
An internal structure with pure maps at all times is called a
\key{pure internal structure}.
Renormalization fixed points and their unstable manifolds have pure internal
structures.

\subsection{Pure maps}
\label{pure-maps}

A map of the form $\zoom*{q}{I}$ is called a \key{pure map} iff
$c \not\in I$.
The map $s \mapsto \zeta_s$ is a bijection between $\reals$ and the set of pure
maps, where
\begin{equation} \label{pure-map}
    \zeta_s(x) = \frac{\left( 1 + (e^{s/(\alpha-1)} - 1)x \right)^\alpha - 1}%
        {e^{\alpha s/(\alpha-1)} - 1},\quad x \in [0,1].
\end{equation}
Note that $\zeta_0(x) = x$, $\zeta_s(x) \to x^\alpha$ as $s \to \infty$,
and $\zeta_s(x) \to 1 - (1 - x)^\alpha$ as $s \to -\infty$.
The parameter $s$ is called the \key{signed distortion} of $\zeta_s$, since
the distortion of $\zeta_s$ is~$\abs{s}$.
A calculation using \eqref{pure-map} shows that distortion is related to
nonlinearity by
\begin{equation} \label{pure-map-nonlinearity}
    \abs{\nonlin\zeta_s}
    =
    (\alpha-1) \big(e^{\abs s / (\alpha-1)} - 1\big)
    \quad\text{and}\quad
    (\alpha-1)\abs{\deriv\nonlin\zeta_s}
    =
    \abs{\nonlin\zeta_s}^2.
\end{equation}
Pure maps are invariant under rescaling; i.e.\ if $\phi$ is a pure map, then so
is $\zoom{\phi}{I}$.

Given a time set $T$, define $\lpure{T}$ to be the space of sequences
$\bar s: T \to \reals$ with the $\ell^1$--norm.
Let $\bar s \mapsto \bar\phi$ be defined by
$\bar\phi(\tau) = \zeta_{\bar s(\tau)}$.
Then $\bar\phi \in \ldiff T$ and $\forall S < \infty$ $\exists K < \infty$ such
that if $\norm{\bar s} \leq S$ then
$\norm{\bar s} \leq \norm{\bar\phi} \leq K\norm{\bar s}$,
by \eqref{pure-map-nonlinearity}.
We call $\bar\phi$ the \key{pure internal structure associated} with $\bar s$.
Since $\bar\phi \in \ldiff T$, $\phi = \compose\bar\phi \in \diff^2$, but in
fact $\phi \in \sdiff$ (it is even analytic), see \cite{MW14}*{\S7}.
We will write $\compose\bar s$ to mean the same thing as $\compose\bar\phi$.

Fix two time sets $T_\pm$.
Given $u,v,c \in \reals$, $\bar s_\pm \in \lpure{T_\pm}$, let $\bar\phi_\pm$ be
the internal structures associated with $\bar s_\pm$ and let
$\phi_\pm = \compose \bar\phi_\pm$.
We say that $(u,v,c,\bar s_\pm)$ is a \key{pure Lorenz map} iff
$(u,v,c,\phi_\pm) \in \lorenz$.
The space of pure Lorenz maps is a subset of the Banach space
$\reals^3 \times \lpure{T_-}\times\lpure{T_+}$.
The renormalization operator on pure Lorenz maps is defined in the obvious way:
let $(u',v',c',\bar\phi'_\pm)$ be the renormalization of $(u,v,c,\bar\phi_\pm)$
and define $\renorm(u,v,c,\bar s_\pm) = (u',v',c',\bar s'_\pm)$ in such a way
that $\bar\phi'_\pm$ are the internal structures associated with $\bar s'_\pm$.
This makes sense because pure maps are invariant under rescaling, so
$\bar\phi'_\pm$ are pure internal structures since $\bar\phi_\pm$ are.

\subsection{Results}

Here we collect results related to internal structures that will be needed in
later sections.
The main results are that the renormalization operator is differentiable
and that the norm of the internal structure of the renormalization is
small for $a$, $b$ large.
This subsection can be skipped on a first read-through.

\begin{theorem} \label{R-diff}
    $\renorm$ acting on internal structures is differentiable.
\end{theorem}

\begin{remark}
    $\deriv\renorm_{f^\star}$ has finitely many expanding eigenvalues at any
    fixed point $f^\star$.
    The idea of the proof is that the internal structures of $f^\star$ are pure
    and lie inside the space whose weighted $\ell^1$--norm is bounded,
    $\norm{\bar s}_\mu = \sum \mu^{-\depth\tau}\abs{\bar s(\tau)} < \infty$,
    for some $\mu \in (0,1)$.
    From this and the expressions for the partial derivatives it can be shown
    that $\deriv\renorm_{f^\star}$ is compact as an operator from $\ell^1$ with
    the $\norm\cdot_\mu$--norm to $\ell^1$ with the usual norm.
    Compactness can then be used to prove the statement.
    However, we do not need this result here so we leave it as a remark.
\end{remark}

Differentiability is essentially a consequence of the following lemma:

\begin{lemma} \label{evaluation-functional}
    $V_x: \ldiff T \to \reals$;
    $\bar\phi \mapsto \compose\bar\phi(x)$, is differentiable,
    $\forall x\in[0,1]$.
\end{lemma}

\begin{proof}
    Given $\tau \in T$, let $\phi_\tau = \bar\phi(\tau)$ and define the
    partial compositions
    $\phi_{\sm<\tau} = \compose\bar\phi|_{\{t < \tau\}}$,
    $\phi_{\sm>\tau} = \compose\bar\phi|_{\{t > \tau\}}$
    and $\phi_{\sm\leq\tau} = \phi_\tau\circ\phi_{\sm<\tau}$.
    Then $\compose\bar\phi = \phi_{\sm>\tau}\circ\phi_\tau\circ\phi_{\sm<\tau}$.
    Let us first show that $V_x$ is differentiable when we use the regular
    linear structure of $\Ck3$.
    Make a $\Ck3$--perturbation $\bar h: T \to \Ck3$ in the $\tau$--direction,
    i.e.\ $h_t = \bar h(t) = 0$, $\forall t\neq\tau$ (and $h_\tau \neq 0$).
    A Taylor expansion gives
    \begin{equation} \label{R-diff-Vx}
    \begin{aligned}
        V_x(\bar\phi + \bar h)
        &=
        \phi_{\sm>\tau}\circ(\phi_\tau + h_\tau)\circ\phi_{\sm<\tau}(x)
        \\
        &=
        V_x(\bar\phi) +
            \deriv\phi_{\sm>\tau}(\phi_{\sm\leq\tau}(x)) \cdot
            h_\tau(\phi_{\sm<\tau}(x)) + o(\abs{h_\tau}).
    \end{aligned}
    \end{equation}
    The linear operator which takes $h_\tau$ to the second term is the
    partial derivative of $V_x(\bar\phi)$ in the direction of~$\tau$.
    The partial derivative depends continuously on $\bar\phi$,
    since $\compose\bar\phi$ is $\Ck1$.
    Now $\deriv V_x(\bar\phi)$ is the sum of all partial derivatives over
    $\tau$, proving that $V_x$ is differentiable.

    To prove that $V_x$ is differentiable on $\ldiff T$, make a
    perturbation $\bar h \in \ldiff T$ in the $\tau$--direction,
    i.e.\ $h_t = \bar h(t) = \id$, $\forall t\neq\tau$ (and $h_\tau \neq \id$).
    The inverse of the nonlinearity operator is differentiable ($\inv\nonlin$
    has an explicit formula, see \cite{MW14}*{Lemma~B.7}, from which
    $\deriv\inv\nonlin$ can be calculated), so
    \begin{equation*}
        \phi_\tau\oplus h_\tau
        =
        \inv\nonlin(\nonlin\phi_\tau + \nonlin h_\tau)
        =
        \phi_\tau + \deriv\inv\nonlin(\nonlin\phi_\tau) \cdot \nonlin h_\tau
        + o(\abs{\nonlin h_\tau}).
    \end{equation*}
    Plug this into \eqref{R-diff-Vx} to prove that the partial derivative in
    the $\tau$--direction is well-defined and then argue as above to get
    differentiability for $V_x$ on $\ldiff T$.
\end{proof}

\begin{proof}[Proof of Theorem~\ref{R-diff}]
    Let $\sigma = (u,v,c,\bar\phi_\pm)$ be renormalizable, let $f_\sigma$ be its
    associated Lorenz map, and let $C$ be the return interval.
    We claim that $C$ depends differentiably on $\sigma$.
    If $p \in \partial C$, then $p$ is a repelling $n$--periodic point of
    $f_\sigma$, for some $n$.
    Define $F(\sigma,x) = f_\sigma^n(x) - x$.
    By Lemma~\ref{evaluation-functional} $\sigma \mapsto f_\sigma^n(x)$, and
    hence $F$, is differentiable.
    By the Implicit Function Theorem (using the fact that $p$ is repelling), if
    $\sigma$ is perturbed slightly then the periodic point persists and its
    new position depends differentiably on $\sigma$.

    The renormalization is defined by pulling back $C$ through internal
    structures, restricting and rescaling.
    By Lemma~\ref{evaluation-functional} this operation is differentiable, so
    the internal structures of the renormalization depend differentiably on
    $\sigma$.
    By \eqref{Rf} and the above the same is true for $u'$, $v'$ and~$c'$.
    Hence $\renorm$ is differentiable.
\end{proof}

The next result is an analog of Proposition~\ref{distortion-invariance} but
for pure Lorenz maps.
Since the composition operator is Lipschitz, given $\delta > 0$ there exists a
maximal $S_\delta > 0$ such that $\norm{\bar s_\pm} \leq S_\delta$ implies
$\norm{\compose\bar s_\pm} \leq \delta$.
This defines $S_\delta$ below.

\begin{proposition}[Bound on internal structures]
    \label{bound-internal-structure}
    For every closed interval $\Delta \subset (0,1)$ and $\delta < \ubdist$
    there exist $N < \infty$ and $K < \infty$ such that if
    $(u,v,c,\bar s_\pm)$ is a pure $(a,b)$--renormalizable Lorenz map,
    $\min\{a,b\} \geq N$, $c \in \Delta$, and
    $\norm{\bar s_\pm} \leq S_\delta$,
    then $\norm{\bar s'_\pm} \leq K \abs C$.
\end{proposition}

Nonlinearities are contracted under rescaling, whereas for the distortion we
have:

\begin{lemma} \label{pure-contraction}
    Let $\zeta_{s_j} = \zoom{\zeta_s}{I_j}$.
    For all $S < \infty$ there exists $K < \infty$ such that if $\abs{s} < S$,
    then $\sum{\abs{s_j}} \leq K \abs s \sum\abs{I_j}$.
\end{lemma}

\begin{proof}
    From the chain-rule,
    $\nonlin(g\circ h) = \nonlin g \circ h \cdot \deriv h + \nonlin h$,
    we get
    $\abs{\nonlin\zeta_{s_j}} \leq \abs{I_j} \abs{\nonlin\zeta_s}$.
    Combined with \eqref{pure-map-nonlinearity} this implies
    $\abs{s_j} \leq \abs{\nonlin\zeta_{s_j}} \leq \abs{I_j} \abs{\nonlin\zeta_s} \leq K \abs s \abs{I_j}$.
\end{proof}

\begin{proof}[Proof of Proposition~\ref{bound-internal-structure}]
    Let $f$ be the Lorenz map associated with $(u,v,c,\bar s_\pm)$, and
    let $\hat T'_\pm \subset T'_\pm$ be the top level times.
    By definition
    \begin{equation} \label{bound-internal-structure-s'}
        \norm{\bar s'_\pm}
        =
        \sum_{\tau'\in \hat T'_\pm} \abs{\bar s'_\pm(\tau')} +
            \sum_{\tau'\not\in \hat T'_\pm} \abs{\bar s'_\pm(\tau')}.
    \end{equation}
    The idea of the proof is that the top level term contributes a definite
    amount of distortion whereas the lower levels are contracted (we only
    prove a Lipschitz bound).

    Consider the top level term of \eqref{bound-internal-structure-s'}.
    A calculation shows that if $\zeta_s = \zoom{q}{I}$, then
    $\abs s \leq (\alpha-1) \abs I / \dist(I,c)$.
    Hence
    $\bar s'_\pm(k,2k-1) \leq K\abs*{\bar C_\pm(k,2k-1)} / d_k^\pm$,
    where $d_k^\pm$ is the distance from $\bar C_\pm(k,2k-1)$ to the critical
    point.
    The notation $\bar C_\pm$ is given by step~2 of the definition of
    $\renorm$.
    By the Expansion Lemma and Lemma~\ref{bounded-space} we can choose $N$
    such that $d_k^\pm \geq \rho$ for some $\rho>0$ not depending on~$f$.
    The assumption $\delta < \ubdist$ and the Expansion Lemma imply that
    $0$ and $1$ are attracting fixed points of $\inv f_\pm$ with uniformly
    bounded multipliers.
    Hence $\sum_k{\abs{\bar C_\pm(k,2k-1)}} \leq K \abs C$ and
    \begin{equation} \label{bound-internal-structure-top}
        \sum_{\tau'\in \hat T'_\pm} \abs{\bar s'_\pm(\tau')}
        \leq
        K\abs C / \rho.
    \end{equation}

    Consider the lower level term of \eqref{bound-internal-structure-s'}.
    Every partial composition $\compose\bar s_\pm|_{\{t\geq\tau\}}$ has norm
    bounded by $S_\delta$, so by Proposition~\ref{nonlinearity-prop}
    $\abs{\bar C_\pm(\tau,2k)} \leq K \abs{\bar C_\pm(k+1,2k+1)}$.
    By the above $\sum_k{\abs{\bar C_\pm(k,2k-1)}} \leq K \abs C$, so
    Lemma~\ref{pure-contraction} gives
    \begin{equation} \label{bound-internal-structure-lower}
        \sum_{\tau'\not\in \hat T'_\pm} \abs{\bar s'_\pm(\tau')}
        \leq
        K \sum{\abs{\bar C_\pm(k,2k-1)}} \norm{\bar s_\pm}
        \leq
        K \abs C \norm{\bar s_\pm}.
    \end{equation}

    Combine \eqref{bound-internal-structure-s'},
    \eqref{bound-internal-structure-top}
    and~\eqref{bound-internal-structure-lower} to see that
    $\norm{\bar s'_\pm} \leq K(\inv\rho + \norm{\bar s_\pm}) \abs C$.
\end{proof}

\section{Derivative estimates}
\label{derivative-estimates}

In this section we compute the partial derivatives of the renormalization
operator acting on pure internal structures.
These results will be applied in the next section.
This section can be skipped on a first read-through and referenced back to
later on.

We use the same notation in this section as in \sref{preliminaries},
with the following additions:
$\lambda_x = \deriv f(x)$ and $\lambda'_x = \deriv(\renorm f)(x)$.
We write $g = o(b^{-n})$ to mean that $b^n g \to 0$ as $b \to \infty$, for
every $n > 0$.
Note that we write $\partial_x$ instead of $\tfrac{\partial}{\partial x}$.

\begin{lemma} \label{DR3D}
    For every $\delta < \ubdist$ and compact intervals
    $\Delta \subset (0,1)$ and $P \subset \reals^+$
    there exists $N < \infty$ such that if $f \in \lorenz_\delta$ is
    $(a,b)$--renormalizable, $c \in \Delta$, $\min\{a,b\} \geq N$ and
    $a/b \in P$, then
    \begin{equation*}
        \partial_u\! \begin{pmatrix}
            u' \\ v' \\ c'
        \end{pmatrix}
        = \frac{w_0 + \eps_0}{\abs U},
        \quad
        \partial_v\! \begin{pmatrix}
            u' \\ v' \\ c'
        \end{pmatrix}
        = \frac{w_1 + \eps_1}{\abs V},
        \quad
        \partial_c\! \begin{pmatrix}
            u' \\ v' \\ c'
        \end{pmatrix}
        = \frac{w_2 - A w_0 + B w_1 + \eps_2}{\abs C},
    \end{equation*}
    where $A = \abs{O(a)}$, $B = \abs{O(b)}$,
    $\norm{\eps_i} = o(b^{-n})$, and
    $\{w_i\}$ are the columns of
    \begin{equation}
        W
        =
        \begin{pmatrix}
            1 + \frac{1-u'}{\lambda_0'-1} &
            -\frac{u'}{\lambda_1'-1} &
            \frac{u'\lambda_1'}{\lambda_1'-1}
                - \frac{(1-u')\lambda_0'}{\lambda_0' - 1} \\

            -\frac{v'}{\lambda_0'-1} &
            1 + \frac{1-v'}{\lambda_1'-1} &
            -\frac{v'\lambda_0'}{\lambda_0'-1}
                + \frac{(1-v')\lambda_1'}{\lambda_1' - 1} \\

            \frac{1-c'}{\lambda_0'-1} &
            -\frac{c'}{\lambda_1'-1} &
            1 - \frac{c'\lambda_1'}{\lambda_1'-1}
                - \frac{(1-c')\lambda_0'}{\lambda_0'-1}
        \end{pmatrix}.
    \end{equation}
\end{lemma}

\begin{proof}
    The statement about the upper-left $2\times2$ matrix was proved in
    Lemma~\ref{DR2D}.
    Here we will calculate $\partial_c u'$, $\partial_c c'$ and
    $\partial_u c'$.
    The calculations for $\partial_c v'$ and $\partial_v c'$ are identical so
    we will not include them here.

    First of all, let us discuss how to choose $N$.
    Let $\delta'$ denote the bound on the norm of the diffeomorphic parts of
    the renormalization, i.e.\ $\renorm f \in \lorenz_{\delta'}$.
    By Lemma~\ref{distortion-invariance} we may choose $N$ such that
    $\delta' < \ubdist$.
    For this choice of $N$ there exists $\lambda > 1$ (not depending on $f$)
    such that both $\lambda_x' \geq \lambda$ and $\lambda_x > \lambda$ for
    $x \in \{0,1\}$.
    This follows from \eqref{q}
    and Lemma~\ref{nonlinearity-prop}.
    Assume that $N$ has been chosen in this way.

    We will derive expressions for $\partial_c l$ and $\partial_c r$,
    where $C = [l,r]$.
    Define
    \begin{equation*}
        A = -\partial_c\Phi(\inv\Phi c),
        \qquad
        B = -\partial_c\Psi(\inv\Psi c),
    \end{equation*}
    where the notation is from~\eqref{Rf}.
    We claim that $A = \abs{O(a)}$, $B = \abs{O(b)}$ and that
    \begin{equation} \label{AB}
        A - \abs*{\partial_c\Phi(\inv\Phi x)} = o(b^{-n}),
        \quad
        B - \abs*{\partial_c\Psi(\inv\Psi y)} = o(b^{-n}),
        \quad
        \forall x \in U, y \in V.
    \end{equation}
    We postpone the proof of this claim.
    From \eqref{dlr} and \eqref{q}
    we get
    \begin{equation*}
        \partial_c l
        =
        \frac{\frac lc \lambda_0' - \partial_c\Phi(\inv\Phi l)}{\lambda_0' - 1}
        =
        \frac{\lambda_0' + A}{\lambda_0' - 1}
            + \frac{\frac{\abs{C_-}}{c} \lambda_0' -
                \partial_c\Phi(\inv\Phi l) - A}%
                {\lambda_0' - 1}.
    \end{equation*}
    We claim that the second fraction in the right-hand side is $o(b^{-n})$.
    This follows from \eqref{AB} and the facts that: $\abs C = o(b^{-n})$ by
    Lemma~\ref{bounded-space}, $c \in \Delta$, and that
    $\lambda_0' \geq \lambda > 1$.
    An identical argument for $\partial_c r$ gives
    \begin{equation} \label{dc_lr}
        \partial_c l = \frac{\lambda_0' + A}{\lambda_0' - 1} + o(b^{-n}),
        \qquad
        \partial_c r = \frac{\lambda_1' + B}{\lambda_1' - 1} + o(b^{-n}).
    \end{equation}

    We now calculate $\partial_c u'$.
    From \eqref{dinvT}, \eqref{AB} and \eqref{dc_lr} we get
    \begin{equation*}
        \partial_c(\inv\Phi l)
        =
        \frac{1}{\deriv\Phi(\inv\Phi l)}
            \left( \frac{\lambda_0' + A}{\lambda_0' - 1} + A + o(b^{-n})
            \right).
    \end{equation*}
    By the mean-value theorem there exists $x \in U$ such that
    $\deriv\Phi(x) = \abs C / \abs U$.
    Furthermore $\deriv\Phi(x) / \deriv\Phi(\inv\Phi l) = 1 + O(\delta')$ by
    Lemma~\ref{nonlinearity-prop}.
    Thus $\deriv\Phi(\inv\Phi l) = (1 + o(b^{-n})) \abs C / \abs U$, since
    $\delta'$ is exponentially small in $b$ by
    Proposition~\ref{distortion-invariance}.
    An identical argument for $\partial_c(\inv\Phi r)$ and using the bounds on
    $A$, $B$ and $\lambda_x'$ gives
    \begin{equation*}
        \frac{\abs C}{\abs U} \partial_c(\inv\Phi l)
        =
        \frac{\lambda_0' (1 + A)}{\lambda_0' - 1} + o(b^{-n}),
        \quad
        \frac{\abs C}{\abs U} \partial_c(\inv\Phi r)
        =
        \frac{\lambda_0' + B}{\lambda_0' - 1} + A + o(b^{-n}).
    \end{equation*}
    Insert these expressions into \eqref{du'} to get
    \begin{equation} \label{dc_u'}
        \abs C \partial_c u'
        =
        -\frac{(1-u') \lambda_0'}{\lambda_0' - 1}
            -\frac{u' \lambda_1'}{\lambda_1' - 1}
            -A \left( 1 + \frac{1-u'}{\lambda_0' - 1} \right)
            -B \frac{u'}{\lambda_1' - 1}
            + o(b^{-n}).
    \end{equation}
    An identical argument gives
    \begin{equation} \label{dc_v'}
        \abs C \partial_c v'
        =
        \frac{(1-v') \lambda_1'}{\lambda_1' - 1}
            + \frac{v' \lambda_0'}{\lambda_0' - 1}
            + B \left( 1 + \frac{1-v'}{\lambda_1' - 1} \right)
            + A \frac{v'}{\lambda_0' - 1}
            + o(b^{-n}).
    \end{equation}

    We will now calculate the partial derivatives of $c'$.
    By definition $\abs C c' = c - l$ and taking partial derivatives gives
    \begin{equation} \label{dc'}
        \abs C \partial c' = \partial c - (1-c') \partial l - c' \partial r.
    \end{equation}
    Inserting \eqref{dc_lr} into \ref{dc'} gives
    \begin{equation} \label{dc_c'}
        \abs C \partial_c c'
        =
        1 - \frac{(1-c')\lambda_0'}{\lambda_0' - 1}
            - \frac{c'\lambda_1'}{\lambda_1' - 1}
            - A\frac{1-c'}{\lambda_0' - 1}
            - B\frac{c'}{\lambda_1' - 1}
            + o(b^{-n}).
    \end{equation}
    From \eqref{du_Phi_l} and \eqref{du_Phi_r} we get expressions for
    $\partial_u l$ and $\partial_u r$, respectively.
    Together with \eqref{dc'} and the above trick we see that
    $\deriv\Phi(\inv\Phi l) \abs U / \abs C = 1 + o(b^{-n})$ and hence
    \begin{equation} \label{du_c'}
        \abs U \partial_u c'
        =
        \frac{1-c'}{\lambda_0' - 1} + o(b^{-n}).
    \end{equation}
    An identical argument gives
    \begin{equation} \label{dv_c'}
        \abs V \partial_v c'
        =
        -\frac{c'}{\lambda_1' - 1} + o(b^{-n}).
    \end{equation}

    Equations \eqref{dc_u'}, \eqref{dc_v'}, \eqref{dc_c'}, \eqref{du_c'} and
    \eqref{dv_c'} imply the lemma.
    It only remains to prove \eqref{AB} and the bounds on $A$ and $B$.

    Let $x \in V$ and note that $\Psi(x) = f^{b-i}(f^i\circ\psi(x))$.
    By the mean-value theorem there exists $x_i \in [0,f^i\circ\psi(x)]$
    such that $\deriv f^{b-i}(x_i) f^i\circ\psi(x) = \Psi(x)$.
    Using the chain-rule to compute $\partial_c\Psi$ and then applying
    \eqref{q}
    and the above we get
    \begin{equation} \label{dc_Psi}
        -\partial_c\Psi(x)
        = \frac 1c
            \sum_{i=0}^{b-1} \deriv f^{b-i}(f^i\circ\psi(x)) f^i\circ\psi(x)
        = \frac{\Psi(x)}{c}
            \sum_{i=0}^{b-1}
            \frac{\deriv f^{b-i}(f^i\circ\psi(x))}{\deriv f^{b-i}(x_i)}.
    \end{equation}
    Since $\lambda_0 \geq \lambda > 1$, $0$ is a uniformly attracting fixed
    point for $\inv f_-$ so the summands are bounded.
    Hence $-\partial_c\Psi(x) \leq K b \Psi(x) / c$ and since
    $c \in \Delta$ this proves the claim.
    The proof that $A = \abs{O(a)}$ is identical.

    We now prove \eqref{AB}.
    Let $x, y \in V$, $x_i = f^i\circ\psi(x)$, and $y_i = f^i\circ\psi(y)$.
    By \eqref{dc_Psi}
    \begin{align*}
        \abs*{\partial_c\Psi(y) - \partial_c\Psi(x)}
        &=
        \frac 1c
        \abs*{\sum_{i=0}^{b-1} \deriv f^{b-i}(x_i)
            \left( \frac{\deriv f^{b-i}(y_i)}{\deriv f^{b-i}(x_i)} y_i - x_i
            \right)}
        \\
        &\leq
        \frac 1c \sum_{i=0}^{b-1} \deriv f^{b-i}(x_i)
            \left(
            \abs*{\frac{\deriv f^{b-i}(y_i)}{\deriv f^{b-i}(x_i)} - 1} y_i
            + \abs*{y_i - x_i}
            \right).
    \end{align*}
    The distortion of $f^{b-i}$ on $f^i\circ\psi(V)$ is exponentially
    small in $b$ by
    Proposition~\ref{distortion-invariance}.
    In particular, the distortion is $o(b^{-n})$.
    Hence, using a mean-value theorem argument as in the above, we get
    \begin{equation*}
        \abs*{\partial_c\Psi(y) - \partial_c\Psi(x)}
        \leq
        \frac 1c K b \left( o(b^{-n}) + \abs{\Psi(y) - \Psi(x)} \right)
        = o(b^{-n}),
    \end{equation*}
    since $\abs*{\Psi(y) - \Psi(x)} \leq \abs C = o(b^{-n})$ by
    Lemma~\ref{bounded-space}.
\end{proof}

In the remainder of this section we use the notation of
\sref{the-internal-structure-of-renormalization}.
Let $T = (T_-,T_+)$ be a pair of time sets.
Since the composition operator is Lipschitz there exists a maximal
$S_\delta > 0$ such that if $\bar s_\pm \in \lpure{T_\pm}$ and
$\norm{\bar s_\pm} \leq S_\delta$, then $\norm{\compose\bar s_\pm} \leq \delta$.

\begin{lemma} \label{partials}
    For every $\delta < \ubdist$ and compact intervals
    $\Delta \subset (0,1)$ and $P \subset \reals^+$
    there exists $N < \infty$ and $K < \infty$ such that if
    $(u,v,c,\bar s_\pm)$ is a pure $(a,b)$--renormalizable Lorenz map,
    $c \in \Delta$, $\min\{a,b\} \geq N$, $a/b \in P$ and
    $\norm{\bar s_\pm} \leq S_\delta$, then
    \begin{gather}
        \abs{C} \abs{\partial_s u'} \leq K b, \qquad
        \abs{C} \abs{\partial_s v'} \leq K b, \qquad
        \abs{C} \abs{\partial_s c'} \leq K b, \label{partials2}
        \\
        \abs U \!\sum_{\tau' \in T'_\pm} \abs{\partial_u \bar s'_\pm(\tau')}
        = o(b^{-n}),
        \qquad
        \abs V \!\sum_{\tau' \in T'_\pm} \abs{\partial_v \bar s'_\pm(\tau')}
        = o(b^{-n}),
        \label{partials3}
        \\
        \abs C \!\sum_{\tau' \in T'_\pm} \abs{\partial_c \bar s'_\pm(\tau')}
        = o(b^{-n}),
        \qquad
        \abs C \!\sum_{\tau' \in T'_\pm} \abs{\partial_s \bar s'_\pm(\tau')}
        = o(b^{-n}).
        \label{partials4}
    \end{gather}
    Here $\partial_s$ denotes partial derivative with respect to an arbitrary
    variable of $\lpure{T_\pm}$.
\end{lemma}

\begin{proof}
    Let $f$ be the map associated with $(u,v,c,\bar s_\pm)$ and let
    $C = [l,r]$.
    Assume that $N$ has been chosen as in the beginning of
    the proof of Lemma~\ref{DR3D}, so that $\delta' < \ubdist$ and
    $\lambda_x' \geq \lambda > 1$ for $x\in\{0,1\}$.
    Throughout the proof assume without loss of generality that $\partial_s$ is
    the partial derivative with respect to the variable at a time
    $\tau \in T_-$ and write $s = \bar s_-(\tau)$.
    That is, we will perturb $s = \bar s_-(\tau)$ slightly and see how the
    renormalization changes.

    We will need the following bounds (see \eqref{Rf} for notation):
    \begin{equation} \label{ds_lr}
        \abs{\partial_s l} \leq K b, \quad \abs{\partial_s r} \leq K b,
        \quad
        \abs{\partial_s\Phi} \leq K b, \quad \abs{\partial_s\Psi} \leq K b.
    \end{equation}
    The first two bounds follow from the last two bounds,
    $\lambda_x' \geq \lambda$, and~\eqref{dlr}
    so let us prove the last two bounds.
    Note that
    \begin{equation} \label{ds-pure}
        \abs*{\partial_t\zeta_t(x)} \leq K x (1 - x).
    \end{equation}
    This can be seen by taking partial derivatives of \eqref{pure-map}
    to get
    \begin{equation} \label{ds_zeta_s}
        \partial_t\zeta_t(x)
        =
        \frac{\deriv\zeta_t(x) x - \deriv\zeta_t(1) \zeta_t(x)}%
            {\nonlin\zeta_t(1)}.
    \end{equation}
    From this \eqref{ds-pure} follows by Taylor expanding $\zeta_t(x)$ and
    $\deriv\zeta_t(x)$ around $0$ and~$1$.

    Let $\bar\phi$ be the internal structure associated with $\compose\bar s_-$.
    We write $\phi = \phi_{\sm>\tau}\circ\phi_\tau\circ\phi_{\sm<\tau}$,
    where $\phi = \compose\bar\phi$,
    $\phi_{\sm>\tau} = \compose\bar\phi|_{\{t>\tau\}}$,
    $\phi_{\sm<\tau} = \compose\bar\phi|_{\{t<\tau\}}$ and
    $\phi_\tau = \bar\phi(\tau) = \zeta_s$.
    We get
    \begin{equation*}
        \partial_s f_-(x)
        =
        \deriv\phi_{\sm>\tau}(\zeta_s\circ\phi_{\sm<\tau}\circ q(x))
        \partial_s\zeta_s(\phi_{\sm<\tau}\circ q(x)).
    \end{equation*}
    A Taylor expansion shows that (here
    $\phi_{\sm\geq\tau} = \phi_{\sm>\tau}\circ\zeta_s$)
    \begin{equation*}
        \partial_s\zeta_s(\phi_{\sm<\tau}\circ q(x))
        =
        \partial_s\zeta_s(f x) +
            \deriv\partial_s\zeta_s(t)
            (\inv\phi_{\sm\geq\tau}\circ f(x) - f(x)),
    \end{equation*}
    for some $t$.
    By differentiating \eqref{ds_zeta_s} we see that
    $\abs{\deriv\partial_s\zeta_s} \leq K$, since
    $\abs{\deriv^2\zeta_s(x)} \leq K$ and $\deriv\zeta_s(x) \leq K$.
    By Lemma~\ref{nonlinearity-prop}
    $\inv K \leq \abs{\deriv\inv\phi_{\sm\geq\tau}} \leq K$, so
    \begin{equation*}
        \abs{\inv\phi_{\sm\geq\tau}\circ f(x) - f(x)}
        \leq
        K f(x) (1 - f(x)).
    \end{equation*}
    Combine these two facts with \eqref{ds-pure} to see that
    \begin{equation} \label{ds_f}
        \abs{\partial_s f(x)} \leq K f(x) (1 - f(x)).
    \end{equation}
    Note that this holds for both branches of $f$, since $s$ was assumed to be
    at a time $\tau \in T_-$, so $\partial_s f_+ = 0$.
    The chain-rule and \eqref{ds_f} gives
    \begin{equation} \label{ds_fn}
        \abs{\partial_s f^{n+1}(x)}
        \leq
        K \sum_{i=0}^n \deriv f^{n-i}(f^{i+1}x) f^{i+1}(x) (1 - f^{i+1}(x)).
    \end{equation}
    The same argument as was used to prove $\abs{\partial_c\Psi} \leq K b$
    in the proof of Lemma~\ref{DR3D} combined with \eqref{ds_fn}
    shows that $\abs{\partial_s f^n} \leq K n$.
    Since $\Phi = f_+^a\circ\phi$ and $\Psi = f_-^b\circ\psi$,
    where $\psi = \compose\bar s_+$, this
    implies that $\abs{\partial_s\Phi} \leq Kb$ (nb.\ $a/b \in P$) and
    $\abs{\partial_s\Psi} \leq Kb$.

    From \eqref{dc'} and \eqref{ds_lr} we immediately get
    the bound on $\partial_s c'$ claimed in \eqref{partials2}.
    By the mean value theorem $\deriv\Phi(x) = \abs C / \abs U$ for some
    $x \in U$.
    Since the distortion of $\Phi|_U$ is bounded by $\delta'$ we get that
    $\abs{\deriv\Phi} \leq K \abs C / \abs U$.
    From \eqref{dinvT}, \eqref{ds_lr} and
    $\abs{\deriv\Phi} \leq K \abs{C}/\abs{U}$
    we get
    $\abs{\partial_s(\inv\Phi l)} \leq K b \abs{U}/\abs{C}$ and similarly
    for $\abs{\partial_s(\inv\Phi r)}$.
    Inserting this into \eqref{du'} proves the bound on $\partial_s u'$
    in \eqref{partials2}.
    The bound on $\partial_s v'$ follows from an identical argument.

    \smallskip
    We will now estimate the partial derivatives $\partial s'_\pm(\tau')$.
    Assume without loss of generality that $\tau' \in T'_+$ and let
    $s' = \bar s'_+(\tau')$.
    By definition $\zeta_{s'} = \zoom{h}{I}$ for some map $h$ and
    interval $I = [x,y]$.
    By step~3 of
    \sref{the-internal-structure-of-renormalization_renormalization} either:
    \begin{enumerate*}[label=(\roman*)]
        \item $\tau'$ is at the top level and $h = q_-$, or
            \label{ds'1}
        \item
            $\depth'_+(\tau') > 0$ and $h = \zeta_{\hat s}$, where we may
            assume without loss of generality that 
            $\hat s = \bar s_-(\hat\tau)$, for some $\hat\tau \in T_-$.
            \label{ds'2}
    \end{enumerate*}
    Since distortion is invariant under rescaling
    $s' = \log\{\deriv h(y)/\deriv h(x)\}$.
    The chain-rule gives
    \begin{equation} \label{ds'}
        \partial s'
        =
        \nonlin h(y)\partial y
        - \nonlin h(x)\partial x
        + \frac{\partial(\deriv h)(y)}{\deriv h(y)}
        - \frac{\partial(\deriv h)(x)}{\deriv h(x)}
        =
        \partial s'_0 + \partial s'_1,
    \end{equation}
    where $\partial s'_0$ and $\partial s'_1$ denote the two difference terms.

    Consider the difference term $\partial s'_0$ of \eqref{ds'}.
    By definition $I$ is mapped diffeomorphically to $C=[l,r]$ by some
    first-entry map $F:I\to C$.
    In case \ref{ds'1} $F$ is of the form $F = f_-^k$, else
    $F = f_-^k\circ\phi_{\sm\geq\hat\tau}$, for some $k$.
    Either way, we have that
    \begin{equation*}
        \partial x = \frac{\partial l - \partial F(x)}{\deriv F(x)},
        \qquad
        \partial y = \frac{\partial r - \partial F(y)}{\deriv F(y)}.
    \end{equation*}
    We can bound $\abs{\partial F} \leq K b$ using the same estimates we did
    for $\partial\Phi$ and $\partial\Psi$ in the above and in the proofs of
    Lemmas \ref{DR2D} and~\ref{DR3D}.
    The distortion of $F$ tends to zero as $N \to \infty$ by
    Proposition~\ref{distortion-invariance}.
    Hence $\abs{I} \abs{\deriv F} \geq \abs{C}/K$.
    Using the above, $\nonlin\zeta_{s'}(1) = \abs{I}\nonlin h(y)$
    and $\nonlin\zeta_{s'}(0) = \abs{I}\nonlin h(x)$
    we get
    \begin{equation*}
        \abs{\partial s'_0} =
        \abs*{\nonlin h(y)\partial y - \nonlin h(x)\partial x}
        \leq \frac{K \abs{\nonlin\zeta_{s'}}
            \max\{\abs{\partial l}, \abs{\partial r}, b\}}%
            {\abs C}.
    \end{equation*}
    We know from the proof of Lemma~\ref{DR2D} that $\partial_u l$ is of the
    order $\abs C / \abs U$ and $\partial_v r$ is of the order
    $\abs C / \abs V$.
    All other partial derivatives $\partial l$ and $\partial r$ are at most
    of the order of~$b$, by the above and the proofs of Lemmas \ref{DR2D}
    and~\ref{DR3D}.
    Apply \eqref{pure-map-nonlinearity} to see that
    \begin{equation*}
        \abs{\partial_u s'_0} \leq K \abs{s'} / \abs U,
        \quad
        \abs{\partial_v s'_0} \leq K \abs{s'} / \abs V,
        \quad
        \abs{\partial_\star s'_0} \leq K b \abs{s'} / \abs C,
    \end{equation*}
    for $\star = c, s$.
    By Lemma~\ref{bounded-space} and Proposition~\ref{bound-internal-structure},
    $b\norm{\bar s'_+} = o(b^{-n})$.
    Since $\sum\abs{s'} = \norm{\bar s'_+}$ it only remains to bound the
    $\partial s_1'$ term.

    Consider the difference term $\partial s'_1$ of \eqref{ds'}.
    A calculation using \eqref{q}
    shows that
    $\partial_u s_1' = \partial_v s_1' = 0$ and hence there is nothing more
    to prove for \eqref{partials3}.
    It is similarly straightforward from \eqref{q}
    to calculate $\partial_c s'_1 = -\abs{\inv F(C)}/c$ in case~\ref{ds'1} and
    $\partial_c s'_1 = 0$ otherwise.
    From the Expansion Lemma and Lemma~\ref{bounded-space} we get
    $\sum \abs{f_-^{-i}(C)} \leq K\abs C = o(b^{-n})$,
    and the first equation of \eqref{partials4} follows.
    The argument for the second equation of \eqref{partials4} follows similarly
    from $\partial_s s'_1 = 0$ in case~\ref{ds'1}, and the claim that in
    case~\ref{ds'2} $\abs{\partial_s s'_1} \leq K \abs{\inv F(C)}$ if
    $s=\hat s$ and $\partial_s s'_1 = 0$ if $s\neq\hat s$.
    To prove the claim, use \eqref{ds_zeta_s} to get
    $\partial_s s'_1 = (\nonlin\zeta_s(y)y - \nonlin\zeta_s(x)x) / \nonlin\zeta_s(1)$
    and then Taylor expand $\nonlin\zeta_s(y)y$ around $x$ and use that
    $\deriv(\nonlin\zeta_s(x)x)$ and $\nonlin\zeta_s(1)$ are bounded.
\end{proof}

\section{Unstable manifolds}
\label{unstable-manifolds}

In this section we prove the existence of unstable manifolds at the fixed
points of~\sref{existence-of-fixed-points}.
We begin by deriving conditions on $\deriv\renorm$ which are sufficient
for the existence of local unstable manifolds and use the results of
\sref{derivative-estimates} to show that they are satisfied.
After this we describe the global unstable manifolds.
Topologically full families only need two parameters so the unstable manifolds
were expected to be two-dimensional, but we prove that they are at least
three-dimensional.
The ``extra'' unstable dimension is related to the movement of the critical
point and it causes infinitely renormalizable maps to appear inside the
unstable manifold.
In particular, the infinitely renormalizable maps cannot form a stable manifold
as was expected.

\subsection{Cone fields and local unstable manifolds}
\label{cone-fields-and-local-unstable-manifolds}

In this subsection we will give conditions for the existence of a
local unstable manifold that are suitable for our setup.
We could not find a reference applicable to our situation, because:
\begin{enumerate*}
    \item we have to use unconventional bounds like the third inequality
        of~\eqref{pointwise-deriv-bound},
    \item we do not have good enough bounds on the derivative to get
        hyperbolicity, and
    \item our map lives on an infinite-dimensional space and it is not a
        diffeomorphism.
\end{enumerate*}
The first issue is the most significant difference to textbook examples.
Our proof is an adaptation of \cite{KH95}*{Theorem~6.2.8}.

The setup is as follows.
We have a smooth map $F: D \subset X\times Y \to X\times Y$ on an open
neighborhood $D$ of $0$, which is a fixed point of~$F$.
Here $X = \reals^n$ for some $n$ and $Y$ is a Banach space.
We will use the notation
\begin{equation*}
    F(x,y) = (\xi(x,y), \eta(x,y)),
\end{equation*}
and write $z = (x,y)$.
The derivative of $F$ is assumed to satisfy the bounds
\begin{equation} \label{pointwise-deriv-bound}
    \frac{\norm{\deriv_x\xi(z)u}}{\norm u} \geq \mu,
    \;
    \frac{\norm*{\deriv_y\xi(z)}}{\norm{\deriv_x\xi(z)u}}
    \leq \frac\nu{\norm u},
    \;
    \frac{\norm*{\deriv_x\eta(z) u}}{\norm*{\deriv_x\xi(z) u}}
    \leq \lambda,
    \;
    \frac{\norm*{\deriv_y\eta(z)}}{\norm{\deriv_x\xi(z)u}}
    \leq \frac{1 - \tau}{\norm u},
\end{equation}
$\forall z \in D$ and $\forall u \in X \setminus \{0\}$.
These bounds state that the maximum expansion of $\deriv_y\xi$,
$\deriv_x\eta$ and $\deriv_y\eta$ are comparable to the minimum expansion
of~$\deriv_x\xi$.

We will begin by giving conditions for the existence of an invariant cone field.
To this end, let $H_\theta$ denote the standard horizontal cone
\begin{equation*}
    H_\theta = \{ (u,v) \in X\times Y \mid \norm v \leq \theta \norm u \},
\end{equation*}
and write $w = (u,v)$.

\begin{lemma}[Invariant cone field] \label{cone-field}
    Define
    \begin{equation*}
        \theta' = \frac{\lambda + (1-\tau)\theta}{1 - \nu\theta},
        \quad
        \theta_0 = \frac{2 \lambda}\tau,
        \quad
        \theta_1 = \frac\tau\nu - \frac{2 \lambda}\tau,
        \quad
        \theta_2 = \frac{\mu - 1}{\mu\nu + 1}.
    \end{equation*}
    If $2 \sqrt{\nu \lambda} < \tau$, then
    $\theta_0 < \theta_1$, $\theta' < \theta$, and
    \begin{equation*}
        \deriv F(z)(H_\theta) \subset H_{\theta'},
        \quad \forall z \in D,
        \; \forall \theta \in (\theta_0, \theta_1).
    \end{equation*}
    If furthermore
    $\mu > (\tau + 2\lambda) / (\tau - 2\lambda\nu)$,
    then $\theta_2 > \theta_0$ and
    \begin{equation*}
        \frac{\norm*{\deriv F(z)w}}{\norm w}
        \geq
        \mu \frac{1 - \nu\theta}{1 + \theta} > 1,
        \quad \forall z \in D,\; \forall w\in H_\theta\setminus\{0\},
        \; \forall \theta < \theta_2.
    \end{equation*}
\end{lemma}

\begin{proof}
    Let $w = (u,v) \in H_\theta$ and let $w' = (u',v') = \deriv F(z)w$.
    Then
    \begin{equation*}
        \frac{\norm{v'}}{\norm{u'}}
        =
        \frac{\norm{\deriv_x\eta(z)u + \deriv_y\eta(z)v}}%
            {\norm{\deriv_x\xi(z)u + \deriv_y\xi(z)v}}
        \leq
        \frac{\lambda + (1 - \tau)\theta}{1 - \nu \theta}
        = g(\theta).
    \end{equation*}
    Assume without loss of generality that $\nu > 0$.
    Solving $g(\theta) = \theta$ and using $4\nu\lambda < \tau^2$ gives
    two fixed points $\theta_\pm > 0$ and $g(\theta) < \theta$ for
    $\theta_- < \theta < \theta_+$.
    Use $\sqrt{1-t} \geq 1 - t$ for $t = 4\lambda\nu\inv[2]\tau$ to see
    that
    $\theta_- \leq \theta_0 < \theta_1 \leq \theta_+$.
    This proves invariance.

    Expansion follows from the assumption $\theta < (\mu - 1) / (\mu\nu + 1)$
    and
    \begin{equation*}
        \frac{\norm{w'}}{\norm w}
        \geq
        \frac{\norm{u'}}{\norm u (1 + \theta)}
        \geq
        \frac{\mu - \mu \nu \theta}{1 + \theta} > 1.
    \end{equation*}
    Insert $\mu > (\tau + 2\lambda) / (\tau - 2\lambda\nu)$ into the definition
    of $\theta_2$ to see that $\theta_2 > \theta_0$.
\end{proof}

\begin{lemma}[Local unstable manifold] \label{loc-unstable-mfd}
    Under the assumptions of Lemma~\ref{cone-field},
    $F$~has a local unstable manifold at~$0$ which:
    \begin{enumerate}
        \item is the graph of a $\theta_0$--Lipschitz map
            $\gamma^*: U \to Y$, where $U \subset X$ is a ball around~$0$,
            \label{loc-unstable-mfd-graph}
        \item is unique in the sense that if
            $F(W) \supset W$, where $W \ni 0$ is the graph of a
            $\theta$--Lipschitz map on $U$ and
            $\theta < \min\{\theta_1,\theta_2,\theta_3\}$,
            where $\theta_3 = \tau (2\nu)^{-1}$,
            then $W = \graph\gamma^*$.
            \label{loc-unstable-mfd-unique}
    \end{enumerate}
\end{lemma}

\begin{proof}
    Let $U\subset X$ be a closed ball around $0$ and define $G$ to be
    the set of $\theta$--Lipschitz maps $\gamma: U \to Y$ fixing~$0$,
    where $\theta_0 < \theta < \min\{\theta_1,\theta_2\}$.
    Assume that $U$ is small enough so that $\graph\gamma \subset D$,
    $\forall \gamma \in G$.
    The graph transform $\Gamma: G \to G$ is defined by sending $\gamma$ to
    $\gamma'$ where
    \begin{equation*}
        \graph \gamma' = F(\graph \gamma) \cap (U\times Y).
    \end{equation*}
    We claim that $\Gamma$ is well-defined for $U$ small enough.
    To see that the right-hand side is the graph of some $\gamma': U \to Y$ we
    need only show that $x \mapsto \xi(x,\gamma(x))$ maps some $U$ injectively
    over itself, since $F(x,\gamma(x)) = (\xi(x,\gamma(x)), \dots)$.
    In order to so, note that by the smoothness of $\xi$, $\forall\eps > 0$ we
    may choose $U$ so that
    $\norm{\xi(x,0)} \geq \norm{\deriv_x\xi(0)x} - \eps\norm x$.
    This and the cone expansion constant of Lemma~\ref{cone-field} shows that
    \begin{equation} \label{loc-unstable-mfd-expansion}
        \begin{aligned}
            \norm*{\xi(x,\gamma(x))}
            &\geq
            \norm{\xi(x,0)} - \norm*{\xi(x,\gamma(x)) - \xi(x,0)}
            \geq (\mu - \eps)\norm x - \mu\nu \norm{\gamma(x)}
            \\
            &\geq (\mu(1 - \nu\theta) - \eps) \norm x
            > (1 + \theta - \eps) \norm x.
        \end{aligned}
    \end{equation}
    But $1 + \theta - \eps > 1$ for $\eps < \theta$, proving that
    $x \mapsto \xi(x,\gamma(x))$ maps some $U$ injectively over itself.
    The cone invariance of Lemma~\ref{cone-field} implies that $\gamma'$ is
    $\theta$--Lipschitz.

    The set $G$ is turned into a complete metric space by endowing it with the
    metric
    \begin{equation*}
        d(\gamma_1,\gamma_2)
        =
        \sup_{x\in U} \frac{\norm*{\gamma_1(x) - \gamma_2(x)}}{\norm x}.
    \end{equation*}
    The completeness of $Y$ and compactness of $U$ ensures that $G$ is
    complete.
    As a remark, the interpretation of this metric is that if $\rho =
    d(\gamma_1,\gamma_2)$, then $\rho$ is the smallest number such that the
    graph of $\gamma_1 - \gamma_2$ is contained in $H_\rho$.
    Choose $\theta < \theta_3$.
    This is possible, since $\theta_0/\theta_3 = 4\lambda\nu\tau^{-2} < 1$ by
    assumption.
    We claim that $\Gamma$ is a contraction in this metric with this choice of
    $\theta$ and $U$ small enough.
    To prove the claim let $\gamma_i \in G$ and let $\gamma_i' =
    \Gamma(\gamma_i)$, for $i=1,2$.
    By definition
    $(x_i', \gamma_i'(x_i')) = (\xi(x,\gamma_i(x)), \eta(x, \gamma_i(x)))$.
    By \eqref{loc-unstable-mfd-expansion} and the bounds on
    $\norm{\deriv_y\xi(z)}$ and $\norm{\deriv_y\eta(z)}$:
    \begin{align*}
        d(\gamma_1', \gamma_2')
        &=
        \sup_{x'} \frac{\norm*{\gamma'_1(x') - \gamma'_2(x')}}{\norm{x'}}
        \\
        &=
        \!\!\!\!
        \sup_{\xi(x,\gamma_1(x))}
        \!\!\!\!
            \frac{\norm*{\eta(x,\gamma_1(x)) - \eta(x,\gamma_2(x)) +
                    \gamma'_2(\xi(x,\gamma_2(x))) -
                    \gamma'_2(\xi(x,\gamma_1(x)))}}%
                {\norm*{\xi(x,\gamma_1(x))}}
        \\
        &\leq
        \sup_{x}
            \frac{\norm*{\eta(x,\gamma_1(x)) - \eta(x,\gamma_2(x))} +
                    \theta \norm*{\xi(x,\gamma_2(x)) - \xi(x,\gamma_1(x))}
                }{(\mu(1 - \nu\theta) - \eps) \norm x}
        \\
        &\leq
        \sup_x
        \frac{(\mu(1 - \tau) + \mu \nu \theta)
                \norm{\gamma_1(x) - \gamma_2(x)}}%
                {(\mu(1 - \nu\theta) - \eps) \norm x}
        \leq
        \frac{1 - \tau + \nu \theta}{1 - \nu\theta - \eps/\mu}
            d(\gamma_1, \gamma_2).
    \end{align*}
    Using $\nu\theta < \tau/2$ we see that the factor in front of
    $d(\gamma_1,\gamma_2)$ is smaller than~$1$ for $\eps$ small enough.
    Hence we may choose $U$ small so that $\Gamma$ is a contraction.

    Since $\Gamma$ is a contraction it has a unique fixed point
    $\gamma^* \in G$ by the Contraction Mapping Theorem.
    By construction the graph of $\gamma^*$ is an invariant manifold of $F$
    and it is unstable by the cone expansion of Lemma~\ref{cone-field}.
    We may choose $\theta$ arbitrarily close to $\theta_0$, so
    $\gamma^*$ is $\theta_0$--Lipschitz.
    This proves property~\ref{loc-unstable-mfd-graph}.
    Property~\ref{loc-unstable-mfd-unique} is a consequence of $\Gamma$ having
    a unique fixed point.
\end{proof}

\subsection{Global unstable manifolds}

We now apply \sref{cone-fields-and-local-unstable-manifolds} to get local
unstable manifolds of $\renorm$; these are then iterated to get global unstable
manifolds.
As the results below show these global manifolds are not complicated;
in fact, they are still graphs if we cut off maps whose renormalization has a
branch which is almost trivial.
They also show that the unstable manifold is at least three-dimensional and
that it contains a two-dimensional strong unstable manifold which is a full
family.

We need some notation before stating the theorems.
Let $\delta_b = 1/b^2$, let $f_{a,b}^\star$ denote an
$(a,b)$--fixed point (which exists for $a$ and $b$ large, by
Theorem~\ref{fixed-points}) and let $\sdiff_{\delta_b} \subset \sdiff$ denote
the ball of radius~$\delta_b$.
Note that we may assume $f_{a,b}^\star \in \lorenz_{\delta_b}$ by
Proposition~\ref{distortion-invariance}.
Let $Q$ and $\Dl_b$ be defined as in Lemma~\ref{Db}.
Choose $\Delta \subset (0,1)$ so that $c(\renorm f) \not\in \Delta$ if
$c(f) \in \partial\Delta_b$ (this is possible by
Lemma~\ref{Delta-b}(\ref{Delta-b-2})).

\begin{theorem} \label{2d-mfd}
    For every $\beta \in \posrationals$ there exist $N < \infty$, $K < \infty$,
    $\lambda < 1$ and $\theta > 0$ such that for every $b \geq N$ and
    $a/b = \beta$, $f_{a,b}^\star$ has a two-dimensional strong unstable
    manifold $\Wuu$ with the following properties:
    \begin{enumerate}
        \item $\Wuu$ is the graph of a $\theta$--Lipschitz map
            $\gamma_b^{\text{uu}}: Q \to \Delta_b \times \sdiff_{\delta_b} \times \sdiff_{\delta_b}$,
            \label{2d-mfd-graph}
        \item $\Wuu$ is unique in the sense that if
            $\gamma: Q \to \Delta_b\times \sdiff_{\delta_b} \times \sdiff_{\delta_b}$
            is \mbox{$\theta$--Lipschitz},
            $\renorm(\graph\gamma \cap \Dl_b) \supset \graph\gamma$
            and $f_{a,b}^\star \in \graph\gamma$, then
            $\gamma = \gamma_b^{\text{uu}}$,
            \label{2d-mfd-uniqueness}
        \item $\inv\renorm:\Wuu\to\Wuu$ is well-defined and
            $\norm{\renorm^{-n} f - f_{a,b}^\star} \leq K \lambda^n$,
            $\forall f \in \Wuu$, $\forall n \geq 0$,
            \label{2d-mfd-expansion}
        \item $f_{a,b}^\star$ is the unique representative of its topological class
            in families of the above type; that is, if
            $\gamma: Q \to \Delta_b\times \sdiff_{\delta_b} \times \sdiff_{\delta_b}$
            is any $\theta$--Lipschitz map whose graph contains $f_{a,b}^\star$,
            then the graph of $\gamma$ does not contain any other infinitely
            $(a,b)$--renormalizable maps,
            \label{2d-mfd-representative}
        \item $\Wuu$ extends to a $2$--dim unstable manifold which is a full
            family.
            \label{2d-mfd-full}
    \end{enumerate}
\end{theorem}

\begin{theorem} \label{3d-mfd}
    For every $\beta \in \posrationals$ there exist $N < \infty$, $K < \infty$ and
    $\lambda < 1$ such that the following holds.
    For every $b \geq N$ and $a/b = \beta$, there exists $\theta_b > 0$ such that
    $f_{a,b}^\star$ has a three-dimensional unstable manifold $\Wu$ with the
    following properties:
    \begin{enumerate}
        \item $\Wu$ is the graph of a $\theta_b$--Lipschitz map
            $\gamma_b^{\text u}: Q\times\Delta \to \sdiff_{\delta_b} \times \sdiff_{\delta_b}$,
            \label{3d-mfd-graph}
        \item $\Wu$ is unique in the sense that
            if
            $\gamma: Q\times\Delta \to \sdiff_{\delta_b} \times \sdiff_{\delta_b}$
            is $\theta_b$--Lipschitz
            and $\renorm(\graph\gamma \cap \Dl_b) \supset \graph\gamma$, then
            $\gamma = \gamma_b^{\text u}$,
        \item $\inv\renorm:\Wu\to\Wu$ is well-defined and
            $\norm{\renorm^{-n} f - f_{a,b}^\star} \leq K \lambda^n$,
            $\forall f \in \Wu$, $\forall n \geq 0$,
            \label{3d-mfd-expansion}
        \item any neighborhood of $f_{a,b}^\star$ in $\Wu$ contains
            an infinitely $(a,b)$--renormalizable map $f \neq f_{a,b}^\star$.
    \end{enumerate}
\end{theorem}

\begin{remark}
    \begin{enumerate*}
        \item We do not prove that there are exactly three unstable eigenvalues
            at $f_{a,b}^\star$; there may be more, but we do not think so.
        \item It is possible to prove that the unstable manifolds are $\Ck1$
            with the techniques used here.
            The key is proving that the composition operator is differentiable
            on pure internal structures.
            In fact, the unstable manifolds are analytic.
            This can be proved by extending the pure maps of the internal
            structures to neighborhoods in $\mathbb C$ and using a holomorphic
            graph transform.
            We omit proofs of these statements for the sake of brevity.
    \end{enumerate*}
\end{remark}

\subsection{Proofs of the unstable manifold theorems}
\label{proofs-of-the-unstable-manifold-theorems}

The proofs rely on the fact that the fixed point is a pure Lorenz map; i.e.\
it has pure internal structures indexed by time sets $T_b^\pm$ which are
uniquely determined by $b$ and $\beta$ (since the combinatorics is stationary).
The composition operator is Lipschitz so there exists a maximal
$S(\delta_b) < \infty$ such that if we define
\begin{equation*}
    E_b
    =
    \left\{
        (\bar s_-,\bar s_+) \in \lpure{T_b^-}\times\lpure{T_b^+} \mid
        \norm{\bar s_\pm} \leq S(\delta_b) \right\},
\end{equation*}
then $\compose\bar s_\pm \in \sdiff_{\delta_b}$,
$\forall (\bar s_-, \bar s_+) \in E_b$.
Define
\begin{equation*}
    \bar\Dl_b
    =
    \{ (u,v,c,\bar s) \in \reals^3\times E_b \mid
        (u,v,c,\compose\bar s_\pm) \in \Dl_b \}.
\end{equation*}
We will write $\bar f = (u,v,c,\bar s)$ to denote pure Lorenz maps in
$\bar\Dl_b$;
$\bar f_b^\star = (u_b^\star,v_b^\star,c_b^\star,\bar s_b^\star)$
denotes the renormalization fixed point.

The idea of the proofs is to construct unstable Lipschitz manifolds inside
the pure space $\reals^3\times E_b$.
By composing the internal structures this gives us manifolds inside $\lorenz$
which are Lipschitz since the composition operator is Lipschitz.
Explicitly, if $\gamma: \reals^3 \to E_b$ is Lipschitz, then $O\circ\gamma:
\reals^3 \to \sdiff\times\sdiff$ is also  Lipschitz, where
$O(\bar s_-,\bar s_+) = (\compose\bar s_-,\compose\bar s_+)$.
So, once we prove that the graph of some $\gamma$ is an unstable manifold in
the pure space, then the graph of $O\circ\gamma$ is an unstable manifold in the
space of Lorenz maps.

The proofs are divided into steps, starting with:
\begin{enumerate*}
    \item proving existence of a local unstable manifold,
    \item showing that the graph transform can be extended,
    \item growing the local unstable manifold to a global manifold.
\end{enumerate*}

\begin{proof}[Proof of Theorem~\ref{2d-mfd}]
\textit{Step 1.}
We will prove the existence of an expanding cone field and a local unstable
manifold.

Write $\renorm(\bar f) = (\xi(\bar f), \eta(\bar f))$ and
$z = (x,y)$, where $\xi(\bar f), x \in \reals^2$
$\eta(\bar f), y \in \reals\times E_b$.
Let $\hat x = (x_1/\abs{U},x_2/\abs{V})$ where $U$ and $V$ are as in
\sref{derivative-estimates}.
We claim that $\forall \bar f \in \bar\Dl_b$
\begin{equation} \label{2d-mfd-norms}
    \begin{aligned}
    \norm{\deriv_{u,v}\xi(\bar f) x} &\geq \frac{\norm{\hat x}}{K},
    &
    \norm{\deriv_{c,\bar s}\xi(\bar f)} &\leq \frac{K b}{\abs{C}},
    \\
    \norm{\deriv_{u,v}\eta(\bar f) x} &\leq K \norm{\hat x},
    &
    \norm{\deriv_{c,\bar s}\eta(\bar f)} &\leq \frac{K b}{\abs{C}}.
    \end{aligned}
\end{equation}
Note first that each of these operators can be thought of as a matrix whose
entries are the partial derivatives we estimated in
\sref{derivative-estimates} and that we are using the $\ell^1$--norm,
hence the operator norm can be bounded by the supremum over all column norms.
Let us prove the claim.

The first column of $\deriv_{c,\bar s}\xi(\bar f)$ is bounded by $Kb/\abs C$ by
Lemma~\ref{DR3D} and the remaining columns have the same bound by
Lemma~\ref{partials}\eqref{partials2}.

Each partial derivative in the first row of $\deriv_{c,\bar s}\eta(\bar f)$ is
bounded by $Kb/\abs C$ by Lemma~\ref{DR3D} (for the first entry) and
Lemma~\ref{partials}\eqref{partials2} (for the remaining entries).
The norm of each column, disregarding the first row, is much smaller than the
entries of the first row by Lemma~\ref{partials}\eqref{partials4}.

The first row of $\deriv_{u,v}\eta(\bar f) x$ has two entries which are bounded
by $K/\abs U$ and $K/\abs V$, respectively, by Lemma~\ref{DR3D}.
Using Lemma~\ref{partials}\eqref{partials3} we get that these are bounds for
the respective norms of the two columns as well.
Use the triangle inequality to get the desired bound on
$\norm{\deriv_{u,v}\eta(\bar f) x}$.

By Lemma~\ref{DR2D}, $\deriv_{u,v}\xi(\bar f) x = \tilde W \hat x$ where
$\det \tilde W \geq \inv K - o(b^{-n})$ and the columns of $\tilde W$ all have
bounded norm.
Lemma~\ref{l1-expansion} gives the desired bound on
$\norm{\deriv_{u,v}\xi(\bar f) x}$.

This concludes the proof of \eqref{2d-mfd-norms}.
We will now show how these bounds give us an invariant expanding cone field and
a local unstable manifold.

We will use \eqref{2d-mfd-norms} to determine constants $\mu$, $\nu$, $\lambda$
and $\tau$ of \eqref{pointwise-deriv-bound}.
Choose $\mu = \inf_{\bar f} \norm{\hat x}/(K \norm x)$.
Note that $\norm{\hat x}/\norm x$ is a linear combination of $\abs U^{-1}$ and
$\abs V^{-1}$, which is at least as large as the smaller of the two.
From this it follows that we may choose $\nu$, and hence $1-\tau$, of the order
$b \max\{\abs U, \abs V\}/\abs C$.
Finally $\lambda \leq K^2$.
From the Expansion Lemma it follows that $\mu \geq \rho^b$, for some $\rho > 1$
not depending on~$\bar f$.
The fractions $\abs C/\abs U$ and $\abs C/\abs V$ are proportional to the
respective derivative of the first-entry map to $C$ (since the first-entries
have bounded distortion by Proposition~\ref{distortion-invariance}) which are
exponentially large in $b$, again by the Expansion Lemma.
Hence $\nu \leq \sigma^b$ for some $\sigma < 1$ not depending on~$\bar f$.

From these bounds on the constants of \eqref{pointwise-deriv-bound} it follows
that the conditions of Lemma~\ref{cone-field} are satisfied by choosing $b$
large.
Furthermore, $\theta_0 \leq K$ so we may choose $\theta$ constant (independent
of $b$) and get an invariant expanding cone field over the standard cone
$H_\theta$.
Since $\nu \ll 1$, the expansion constant is of the order $\mu \gg 1$.

By Lemma~\ref{loc-unstable-mfd} there is a local unstable manifold
$W_\text{loc}$ which is the
graph of a $\theta$--Lipschitz map $\gamma_\text{loc}: B \to (0,1)\times E_b$
for some open set $B \subset Q$ containing $(u_b^\star,v_b^\star)$.

\smallskip
\noindent\textit{Step 2.}
We will now define the graph transform on graphs over $Q$.

Let $G$ be the set of $\theta$--Lipschitz maps
$\gamma: Q \to \Delta_b\times E_b$ such that $\bar f_b^\star \in \graph\gamma$.
We claim that $\graph\gamma \cap \bar\Dl_b$ is mapped diffeomorphically onto $Q$
by $\xi$, for every $\gamma \in G$.
This implies that $\renorm(\graph\gamma\cap\bar\Dl_b)$ is a graph and
from the invariant cone field it follows that it is $\theta$--Lipschitz.
Hence the graph transform on $G$ is well-defined.
We will now prove the claim.

Let $\ell(u,v) = \inv\xi(u,v)$.
By Lemma~\ref{DR2D} every $(u,v) \in Q$ is a regular value of
$\xi$ (since $\det\deriv_{u,v}\xi \neq 0$), so $\ell(u,v)$ is a manifold of
codimension two.
Now pick some arbitrary $\gamma_1 \in G$ and $(u,v) \in Q$.
We will prove the claim by showing that $\ell(u,v)$ meets the graph of
$\gamma_1$ in a unique point.
Let $\gamma_0 \in G$ be the standard two-dimensional family through
$\bar f_b^\star = (u_b^\star, v_b^\star, c_b^\star, \bar s_b^\star)$, i.e.\
$\gamma_0(u,v) = (u, v, c_b^\star, \bar s_b^\star)$.
Because of Lemma~\ref{Db}, the graph of $\gamma_0$ meets $\ell(u,v)$
exactly once.
Homotope $\gamma_0$ to $\gamma_1$ via
$\gamma_t = t \gamma_0 + (1-t)\gamma_1$ and note that $\gamma_t \in G$,
for all $t \in [0,1]$.
Now let us see for how long the intersection persists under this homotopy.
First of all note that any intersection must be transversal;
all tangent vectors of $\ell(u,v)$ lie in the complementary cone, since
$\xi(\ell(u,v))$ is the graph of a constant map
$(0,1)\times E_b \to Q$, and since the complementary cone field is invariant
under $\deriv\inv\renorm$.
Due to transversality there are only two possibilities: either the graph of
$\gamma_t$ contains a boundary point of $\ell(u,v)$ for some $t\in(0,1)$,
or $\ell(u,v)$ meets $\gamma_1$ in a unique point.
We will show that the first option cannot occur.

Let $\gamma \in G$ and consider the projection of the graph of $\gamma$ to
$\Delta_b\times E_b$.
The diameter of this projection is exponentially small in $b$ (since $\theta$
is fixed and $1-u$ and $1-v$ are exponentially small in $b$ as a consequence of
the Expansion Lemma; see its accompanying remark) but the diameter of
$\Delta_b\times E_b$ is at least of the order $1/b^2$ (by
Lemma~\ref{Delta-b}(\ref{Delta-b-4}) and since $\delta_b = 1/b^2$).
Furthermore, $(c_b^\star,\bar s_b^\star)$ is bounded away from the boundary of
$\Delta_b\times E_b$ by Lemma~\ref{Delta-b}(\ref{Delta-b-3}) and since
$\norm{\bar s_b^\star}$ is exponentially small in $b$
(by Lemma~\ref{bounded-space} and Proposition~\ref{bound-internal-structure}).
Hence the graph of $\gamma$ is far away from the boundary of $\ell(u,v)$ for
all $\gamma \in G$.
In particular, it holds for all $\gamma_t$ so $\gamma_1$ must meet $\ell(u,v)$
in a unique point.

\smallskip
\noindent\textit{Step 3.}
We will now grow the local unstable manifold and prove properties
\ref{2d-mfd-graph} to~\ref{2d-mfd-expansion}.

Let $\gamma_0 \in G$ be some map whose graph coincides with the local unstable
manifold on some neighborhood of $\bar f_b^\star$.
That is, $\gamma_0|_B = \gamma_\text{loc}|_B$ for some open neighborhood
$B \subset Q$ of $(u_b^\star, v_b^\star)$.
Let $\gamma_k = \Gamma^k(\gamma_0)$ and $W_k = \graph\gamma_k$, where
$\Gamma: G\to G$ is the graph transform from step~2.
We claim that there exists $n < \infty$ such that
$W_n \subset \renorm^n(W_\text{loc})$.
Since the global unstable manifold at $\bar f_b^\star$ is given by
$\bigcup_{k\geq0}\renorm^k(W_\text{loc})$, this implies that $\Wuu = W_n$ and
hence property~\ref{2d-mfd-graph} follows.

To see why the claim holds, pick an arbitrary $\bar f_n \in W_n$ and an arbitrary
curve $\sigma_n:[0,1]\to W_n$ such that $\sigma_n(0)=\bar f_b^\star$ and
$\sigma_n(1)=\bar f_n$.
Assume without loss of generality that $\sigma_n$ is differentiable.
Let $\sigma_k:[0,1]\to W_k$ be curves in the backward orbit of $\sigma_n$
and let $\abs{\sigma_k}$ denote the length of $\sigma_k$.
By step~1 the standard cone $H_\theta$ is expanded by $\deriv\renorm(\bar f)$ for
all $\bar f \in \bar\Dl_b$, so the tangent vectors along $\sigma_k$ are also expanded
since they lie inside $H_\theta$.
Hence there exists $\lambda < 1$ such that
$\abs{\sigma_k} \leq \lambda \abs{\sigma_{k+1}}$, for $k=0,\dotsc,n-1$.
If we let $\bar f_0 = \sigma_0(1)$, then it follows that
$\norm{\bar f_0 - \bar f_b^\star} \leq \abs{\sigma_0} \leq \lambda^n\abs{\sigma_n}$.
Since $W_n$ is the graph of a Lipschitz function over a bounded domain,
there exists $K < \infty$ (not depending on~$\bar f_n$) such that
$\abs{\sigma_n} \leq K$ for some choice of $\sigma_n$.
But $\bar f_n \in W_n$ was arbitrary, so this shows that
$\norm{\bar f - \bar f_b^\star} \leq K\lambda^n$ for every $\bar f \in \renorm^{-n}(W_n)$.
In particular, there exists $n < \infty$ (depending on $B$) such that
$\renorm^{-n}(W_n) \subset W_\text{loc}$, since
$\gamma_0|_B = \gamma_\text{loc}|_B$.
This concludes the proof of the claim and also proves
property~\ref{2d-mfd-expansion} (injectivity of $\inv\renorm$ follows from the
claim of step~2).

If the graph of $\gamma$ is invariant in the sense of
property~\ref{2d-mfd-uniqueness}, then it must contain $W_\text{loc}$ by
Lemma~\ref{loc-unstable-mfd}.
By the above claim it follows that $\gamma = \gamma_b^\text{uu}$,
which proves property~\ref{2d-mfd-uniqueness}.

\smallskip
\noindent\textit{Step 4.}
We will now prove property~\ref{2d-mfd-representative}.

Let $\gamma \in G$, $W = \graph\gamma$, and assume $\bar f \in W$ is
infinitely $(a,b)$--renormalizable.
Note that the only $(a,b)$--renormalizable maps which are not in $\bar\Dl_b$ can
be at most once renormalizable.
Explicitly, a twice renormalizable map must have $(u',v') \in Q$ by the
Expansion Lemma (for $b$ large enough) which together with Lemma~\ref{Db}
implies $\bar f \in \bar\Dl_b$.
By deforming $\gamma$ (without leaving $G$)\footnote{%
    If $\bar f$ was in the boundary of the cone $H_\theta + \bar f_b^\star$ then
    this deformation would take us outside $G$, but then we could simply make
    $\theta$ slightly larger and $\Gamma$ would still be well-defined.
    }
    we may additionally assume that
$W|_V = W_\text{loc}|_V$ on some neighborhood $V$ of $\bar f_b^\star$.
By step~3, $\renorm^n(W_\text{loc}\cap V) \supset \Wuu$ for some $n < \infty$.
Since $\bar f$ is infinitely renormalizable this means that $\bar f$ must have been
in $W_\text{loc}$ to begin with.
For the same reason, $\renorm^n \bar f \in \Wuu \cap \bar\Dl_b$, $\forall n\geq 0$.
Hence property~\ref{2d-mfd-expansion} implies that
$\norm{\bar f - \bar f_b^\star} \leq K\lambda^n$, $\forall n\geq 0$.
That is, $\bar f = \bar f_b^\star$.

\smallskip
\noindent\textit{Step 5.}
We will now prove property~\ref{2d-mfd-full}.
\nopagebreak

We claim that there is a full island $I\subset Q$.
Hence it is inside the domain of $\gamma_b^\text{u}$ so
$\renorm(\graph\gamma_b^\text{uu}|_I)$ is a full family by
\cite{MdM01}*{Prop.~2.1}; it is also a $2$--dim unstable manifold extending
$\Wuu$.
To prove the claim, note that $\gamma_b^\text{u}$ can be extended to a family
satisfying \cite{MdM01}*{Prop.~2.1} (see step~4 of the proof of
Theorem~\ref{3d-mfd}) and hence it contains a full island $I$; $I \subset Q$
for $b$ large, by the Expansion Lemma.
\end{proof}

\begin{proof}[Proof of Theorem~\ref{3d-mfd}]
\noindent\textit{Step 1.}
We will prove the existence of an expanding cone field and a local unstable
manifold.
The arguments used here are identical to step~1 of the proof of
Theorem~\ref{2d-mfd} but we write them out in detail since the resulting
invariant cone fields are quite different.

Write $\renorm(\bar f) = (\xi(\bar f), \eta(\bar f))$ and
$z = (x,y)$, where $\xi(\bar f), x \in \reals^3$ and
$\eta(\bar f), y \in E_b$.
Let $\hat x = (x_1/\abs U,x_2/\abs V, x_3/\abs C)$ where $U$, $V$ and $C$ are
as in \sref{derivative-estimates}.
We claim that $\forall \bar f \in \bar\Dl_b$
\begin{equation} \label{3d-mfd-norms}
    \begin{aligned}
    \norm{\deriv_{u,v,c}\xi(\bar f) x} &\geq \frac{\norm{\hat x}}{K b},
    &
    \norm{\deriv_{\bar s}\xi(\bar f)} &\leq \frac{K b}{\abs{C}},
    \\
    \norm{\deriv_{u,v,c}\eta(\bar f) x} &\leq \norm{\hat x} o(b^{-n}),
    &
    \norm{\deriv_{\bar s}\eta(\bar f)} &\leq \frac{o(b^{-n})}{\abs{C}}.
    \end{aligned}
\end{equation}
Note first that each of these operators can be thought of as a matrix whose
entries are the partial derivatives we estimated in
\sref{derivative-estimates} and that we are using the $\ell^1$--norm,
hence the operator norm can be bounded by the supremum over all column norms.
Let us prove the claim.

By adding the bounds of Lemma~\ref{partials}\eqref{partials2} we see that
all columns of $\deriv_{\bar s}\xi(\bar f)$ have norm bounded by $Kb/\abs C$,
proving the second inequality of~\eqref{3d-mfd-norms}.

All columns of $\deriv_{\bar s}\eta(\bar f)$ have norm bounded by
$o(b^{-n})/\abs C$
by the second equation of Lemma~\ref{partials}\eqref{partials4},
proving the fourth inequality of~\eqref{3d-mfd-norms}.

The three columns of $\deriv_{u,v,c}\eta(\bar f)$ have norms bounded by
$o(b^{-n})/\abs U$, $o(b^{-n})/\abs V$ and $o(b^{-n})/\abs C$, respectively,
by Lemma~\ref{partials} \eqref{partials3} and the first equation of
\eqref{partials4}.
Apply the triangle inequality to finish the proof of the third inequality
of~\eqref{3d-mfd-norms}.

Let us finally prove the first inequality of \eqref{3d-mfd-norms}.
By Lemma~\ref{DR3D} $\norm{\deriv_{u,v,c}\xi(\bar f) x} = \norm{\tilde W\xi}$ where
$\tilde W = W + (\eps_0, \eps_1, \eps_2 - A w_0 + B w_1)$.
Let $S$ be the area of the largest face of the convex hull of $\tilde W$.
The first two columns of $\tilde W$ have bounded norms whereas the third
has norm of order $O(b)$.  Hence $S \leq K b$.
Since all columns of $W$ have bounded norm and since the determinant is
continuous and alternating
$\det \tilde W = \det W + o(b^{-n})$.
From Lemma~\ref{l1-expansion} we get that
\begin{equation} \label{3d-mfd-det}
    \norm{\deriv_{u,v,c}\xi(\bar f) x} = \norm{\tilde W \xi}
    \geq \frac{\abs{\det W} - \abs{o(b^{-n})}}{K b} \norm\xi.
\end{equation}
In the construction of $\Dl_b$ we were free to choose $u'$ and $v'$ arbitrarily
close to~$1$.
This fact together with Proposition~\ref{distortion-invariance} and
\eqref{q}
show that $\lambda_0'$ and $\lambda_1'$ can be taken arbitrarily close to
$\alpha/c'$ and $\lambda_1' \to \alpha/(1-c')$, respectively.
Using Lemma~\ref{DR3D} we get
\begin{equation*}
    \det W = -\frac{2\alpha(2\alpha-1)(\alpha-1)c'(1-c')}%
        {(\alpha-c')^2 (\alpha-1+c')^2} + \eps,
\end{equation*}
for arbitrarily small $\eps$.
This is uniformly bounded away from~$0$ since $c' \in \Delta$ for all
$\bar f \in \bar\Dl_b$, and $\alpha>1$.
Hence \ref{3d-mfd-det} implies the first inequality of~\eqref{3d-mfd-norms}.

This concludes the proof of~\eqref{3d-mfd-norms}.
We will now show how these bounds give us an invariant expanding cone field and
a local unstable manifold.

Note that
$\norm{\hat x}/\norm x = t_1/\abs U + t_2/\abs V + t_3/\abs C$,
where $\sum\abs{t_i} = 1$.
We may assume $\abs U, \abs V < \abs C$ by choosing $b$ large enough, so
$\norm{\hat x}/\norm x \geq \abs C^{-1}$.
Using this and \eqref{3d-mfd-norms} it follows that
\eqref{pointwise-deriv-bound} is satisfied for
\begin{align*}
    \mu &= \inf_{\bar f\in\bar\Dl_b} \frac{1}{K b \abs C},
    &
    \nu &= K b^2,
    &
    \lambda &= o(b^{-n}),
    &
    1 - \tau &= o(b^{-n}).
\end{align*}
By Lemma~\ref{bounded-space}, $\mu \to \infty$ as $b \to \infty$.
From the above it follows that the conditions of Lemma \ref{cone-field}
are satisfied.
The angles of Lemmas \ref{cone-field} and~\ref{loc-unstable-mfd} satisfy
$\theta_0 \leq o(b^{-n})$ whereas $\theta_i \geq 1/(K b^2)$ for
$i=1,2,3$.
This shows that we may choose $\theta_b = 1 / b^3$ and get that the cone
$H_{\theta_b}$ is invariant and expanded.
Note the differences to Theorem~\ref{2d-mfd}: the angle of the invariant cone
cannot be chosen independently of~$b$ but instead we may choose it arbitrarily
small.
We also get a local unstable manifold $W_\text{loc}$ which is the
graph of a $\theta_b$--Lipschitz map $\gamma_\text{loc}: B \to E_b$
for some open neighborhood $B \subset Q\times \Delta_b$ of
$(u_b^\star, v_b^\star, c_b^\star)$,
where $\bar f_b^\star = (u_b^\star, v_b^\star, c_b^\star, \bar s_b^\star)$.

\smallskip
\noindent\textit{Step 2.}
We will now define the graph transform on graphs over $Q\times\Delta$.
The arguments use step~2 of the proof of Theorem~\ref{2d-mfd} but they are
somewhat different.

Let $G$ be the set of $\theta_b$--Lipschitz maps
$\gamma: Q\times\Delta \to E_b$ such that $\bar f_b^\star \in \graph\gamma$,
where $\theta_b$ is as in step~1.
We claim that $\xi|_W$ is a diffeomorphism to its image and that
$\xi(W) \supset Q\times\Delta$, where
$W = \bar\Dl_b \cap \graph\gamma$.
This implies that $\renorm(W)$ is a graph over $Q\times\Delta$ and
from the invariant cone field it follows that it is $\theta_b$--Lipschitz.
Hence the graph transform on $G$ is well-defined.
We will now prove the claim.

Every value of $\xi|_{W}$ is regular.
This follows from the cone invariance, which shows that every tangent plane of
the graph of $\gamma$ is in $H_{\theta_b}$ so its $(u,v,c)$--projection is
onto, and from Lemma~\ref{DR3D}, which shows that $\deriv_{u,v,c}\xi$ has
non-zero determinant.

By the above we can pull back $(u,v)\times(0,1)$ by $\xi|_{W}$ and
get a (nonempty) curve $\ell(u,v) \subset W$, $\forall (u,v)\in Q$.
This curve is a graph over $\Delta_b$ by step~2 of the proof of
Theorem~\ref{2d-mfd}.
To see this use that $(u,v) \mapsto (u,v,c_0,\gamma(u,v,c_0))$ is a
$\delta_b$--Lipschitz graph so its intersection with $W$ is a
diffeomorphic copy of $Q$ under $\renorm$ followed by a $(u,v)$--projection,
$\forall c_0\in\Delta_b$.

The $\xi$--image of $\ell(u,v)$ is in $(u,v)\times(0,1)$ and by
Lemma~\ref{Delta-b} it contains $(u,v)\times\Delta$.
Furthermore $\deriv_{u,v,c}\xi$ is of the form $(0,0,t)$ for tangent vectors of
$\ell(u,v)$ and $t\neq0$ since each value of $\xi|_{W}$is regular.
Thus $\xi$ is injective on $\ell(u,v)$.
This finishes the proof of the claim, since by the above $W$ is
foliated by $\{\ell(u,v)\}_{(u,v)\in Q}$.

\smallskip
\noindent\textit{Step 3.}
This is identical to step~3 in the proof of Theorem~\ref{2d-mfd}.

\smallskip
\noindent\textit{Step 4.}
Pick $c \in \Delta_b$ close to $c(\bar f_b^\star)$.
Let $\gamma_c: Q \to \{c\}\times E_b$ be the map whose graph contains all
points of $\Wu$ with critical point $c$, i.e.\
$\gamma_c(u,v) = (c,\gamma_b^\text{u}(u,v,c))$.
Let $(\bar\phi_{u,v},\bar\psi_{u,v}) = \gamma_b^\text{u}(u,v,c)$,
$\phi_{u,v} = \compose\bar\phi_{u,v}$,
$\psi_{u,v} = \compose\bar\psi_{u,v}$,
and define
\begin{equation*}
    F(u,v)
    =
    \left(
        \frac{\phi_{u,v}(u) - c}{1 - c}
        ,
        \frac{c - \psi_{u,v}(1 - v)}{c}
    \right).
\end{equation*}
$F$ takes each map in the graph of $\gamma_c$ to its critical
values, normalized to lie in~$[0,1]$.
Extend $F$ continuously to a rectangle set $\Lambda \supset Q$ such that
$F(\Lambda) = [0,1]^2$ and $F:\partial\Lambda \to \partial([0,1]^2)$ has
non-zero degree.
For example, set $\bar\phi_{u,v}(\tau)=0$ and $\bar\psi_{u,v}(\tau)=0$ outside
an $\eps$--neighborhood of~$Q$, $\forall\tau$, and interpolate linearly.
This extension simply means that we get a family where if one branch is
trivial or full then the critical value of the other branch runs through all
possible values as we change $(u,v)$.
From \cite{MdM01}*{Prop.~2.1} it follows that there exists
$(u,v)\in\Lambda$ such that $\bar f = (u,v,c,\bar\phi_{u,v},\bar\psi_{u,v})$ is
infinitely
$(a,b)$--renormalizable ($\bar f$ can be approximated by $\bar f_n$ whose
associated Lorenz maps $f_n$ have
finite critical orbits, e.g.\ take $\bar f_n$ such that $\renorm^n \bar f_n$ is
full).
Maps outside $\bar\Dl_b$ are at most once renormalizable (see step~4 of the proof
of Theorem~\ref{2d-mfd}), so $(u,v) \in Q$, and consequently $\bar f \in \Wu$.
\end{proof}

\appendix

\section{Collected results}

\begin{lemma} \label{power-law-distortion}
    Let $q$ be a power-law branch with critical point at~$0$.
    Given a triplet of points
    $0 < x < y < z < \infty$, we denote $L = [0,x]$, $M = [x,y]$,
    $R = [y,z]$, $LM = L \cup M$, $MR = M \cup R$, and $LR = L \cup M \cup R$.
    Then
    \begin{gather}
        \label{L-over-LM}
        \frac{\abs{q(L)}}{\abs{q(LM)}}
        = \left( \frac{\abs{L}}{\abs{LM}} \right)^\alpha\mkern-10mu,
        \\
        \label{M-over-LM}
        \frac{\abs{M}}{\abs{LM}} \leq \frac{\abs{q(M)}}{\abs{q(LM)}}
        \leq \alpha \frac{\abs{M}}{\abs{LM}},
        \\
        \label{M-over-MR-weak}
        \frac{\abs{q(M)}}{\abs{q(MR)}}
        \leq \frac{\abs{M}}{\abs{MR}}.
    \end{gather}
    Furthermore, for every $\tau > 0$ there exists $\rho \in (0,1)$ such that
    if $\abs{LR} \geq (1+\tau) \abs{LM}$ and $\abs{R} \geq \tau \abs{M}$, then
    \begin{equation}
        \label{M-over-MR}
        \frac{\abs{q(M)}}{\abs{q(MR)}} \leq \rho
        \frac{\abs{M}}{\abs{MR}}.
    \end{equation}
\end{lemma}

\begin{lemma} \label{l1-expansion}
    Let $M$ be a non-singular $n\times n$ matrix, let $V$ be the volume of the
    $n$--simplex $\sigma$ with vertices $\{0,e_1,\dots,e_n\}$, and let $A$ be
    the surface area of the largest face of the convex hull of
    $\{\pm m_i\}_{i=1}^n$.
    Here $e_i$ are the standard basis vectors of $\reals^n$ and $m_i$ are the
    columns of $M$.
    Then
    \begin{equation*}
        \frac{\norm{Mx}_1}{\norm{x}_1} \geq V \frac{\abs{\det M}}{A},
        \qquad \forall x \in \reals^n\setminus\{0\},
    \end{equation*}
    where $\norm{\cdot}_p$ denotes the $\ell^p$--norm.
\end{lemma}

\begin{proof}
    Note that $\norm{Mx}_1 / \norm{x}_1$ is the $\ell^1$--distance from the
    origin to a point on the boundary, $H$, of the convex hull of $\{\pm
    m_i\}_{i=1}^n$, for $x \neq 0$.
    Let $\xi \in H$ be a point at minimal $\ell^1$--distance to the origin.
    By changing signs of some columns of $M$ it is possible to obtain a
    matrix $\tilde M$ such that $\xi$ is in the image the simplex $\sigma$
    under $\tilde M$.
    By the change of variables formula
    $\vol(\tilde M \sigma) = \abs{\det M} V$, since
    $\abs{\det \tilde M} = \abs{\det M}$.
    At the same time
    $\vol(\tilde M \sigma) \leq \norm{\xi}_2 A$, since
    $\norm{\xi}_2 \geq \min\{\norm{x}_2 \mid x \in H \cap \tilde M\sigma\}$.
    Hence
    \begin{equation*}
        \frac{\norm{Mx}_1}{\norm{x}_1} \geq \norm{\xi}_1
        \geq \norm{\xi}_2
        \geq V \frac{\abs{\det M}}{A}. \qedhere
    \end{equation*}
\end{proof}

\begin{lemma}[\cite{M98}*{Lemma~10.3}] \label{nonlinearity-prop}
    If $g \in \diff^2$ (see \sref{preliminaries}), then
    \begin{equation*}
        e^{-\abs{y-x}\norm g} \leq \frac{\deriv g(y)}{\deriv g(x)}
            \leq e^{\abs{y-x}\norm g}, \qquad
        e^{-\norm g} \leq \deriv g(x) \leq e^{\norm g}.
    \end{equation*}
\end{lemma}

\begin{lemma}[Koebe Lemma \cite{J96}*{Lem.~2.4}] \label{koebe-lemma}
    If $\inv g \in \sdiff$ (see \sref{preliminaries}), then
    \begin{equation*}
        \abs*{\nonlin g(x)} \leq
        2 \min\{\abs{x},\abs{1-x}\}^{-1}.
    \end{equation*}
\end{lemma}

\end{document}